\documentclass[12pt]{amsart}

\usepackage{amsfonts}
\usepackage{amsmath}
\usepackage{amssymb}

\usepackage{mathrsfs}
\usepackage{amscd}
\usepackage[all]{xy}
\usepackage[pdftex,final]{graphicx}

\usepackage{hyperref}
\usepackage[hyperpageref]{backref}

\newtheorem{lem}{Lemma}[section]
\newtheorem{thm}[lem]{Theorem}
\newtheorem{prop}[lem]{Proposition}
\newtheorem{cor}[lem]{Corollary}
\theoremstyle{definition}

\newtheorem{exa}[lem]{Example}

\newtheorem{rem}[lem]{Remark}

\newcommand{\imp}{\Longrightarrow}

\newcommand{\Q}{\Bbb{Q}}
\newcommand{\F}[1]{\Bbb{F}_{#1}}

\newcommand{\Z}{\Bbb{Z}}

\newcommand{\R}{\Bbb{R}}
\newcommand{\C}{\Bbb{C}}
\newcommand{\K}{\mathcal{K}}

\newcommand{\LL}{L^{\tau}}

\newcommand{\e}{\epsilon}

\newcommand{\wn}[1]{\mathcal{W}_{#1}}

\newcommand{\spl}[2]{\operatorname{SL}_{#1}\left(#2\right)}
\newcommand{\pspl}[2]{\mathrm{PSL}_{#1}(#2)}
\newcommand{\gl}[2]{\mathrm{GL}_{#1}(#2)}
\newcommand{\pgl}[2]{\mathrm{PGL}_{#1}(#2)}

\newcommand{\apb}[1]{{\mathcal{P}}^c(#1)}

\newcommand{\pb}[1]{\mathcal{P}(#1)}

\newcommand{\arpb}[1]{{\mathcal{RP}}^c(#1)}
\newcommand{\rpb}[1]{\mathcal{RP}(#1)}

\newcommand{\arpbker}[1]{{\mathcal{RP}}^c_1(#1)}
\newcommand{\rpbker}[1]{\mathcal{RP}_1(#1)}

\newcommand{\qrpb}[1]{\widetilde{\mathcal{RP}}(#1)}

\newcommand{\qrpbker}[1]{\widetilde{\mathcal{RP}}_1(#1)}

\newcommand{\abl}[1]{{\mathcal{B}}^c(#1)}
\newcommand{\bl}[1]{\mathcal{B}(#1)}

\newcommand{\arbl}[1]{{\mathcal{RB}}^c(#1)}
\newcommand{\rbl}[1]{{\mathcal{RB}}(#1)}

\newcommand{\gpb}[1]{\left[ #1\right]} 
\newcommand{\gpbold}[1]{\left\{ #1\right\}} 
\newcommand{\agpb}[1]{\left[#1\right]^c} 
\newcommand{\agpbold}[1]{\left\{#1\right\}^c} 

\newcommand{\sus}[1]{\psi(#1) }
\newcommand{\suss}[2]{\psi_{#1}\left( #2\right)}

\newcommand{\ks}[2]{{\K}^{\mbox{\tiny $(#1)$}}_{#2}}

\newcommand{\bconst}[1]{C_{#1}}
\newcommand{\cconst}[1]{D_{#1}}

\newcommand{\pn}[2]{\mathbb{P}^{#1}(#2)}
\newcommand{\projl}[1]{\pn{1}{#1}}

\renewcommand{\forall}{\mbox{ for all }}

\newcommand{\id}[1]{\mathrm{Id}_{#1}}

\renewcommand{\ker}[1]{\mathrm{Ker}(#1)}
\newcommand{\image}[1]{\mathrm{Im}(#1)}

\newcommand{\ab}[1]{#1^{\mbox{\tiny ab}}}

\newcommand{\sgr}[1]{\mathrm{R}_{#1}}

\newcommand{\an}[1]{\left\langle{#1}\right\rangle}
\newcommand{\pf}[1]{\left\langle\!\left\langle{#1}\right\rangle\!\right\rangle}
\newcommand{\aug}[1]{\mathcal{I}_{#1}}

\newcommand{\Extpow}[3]{\wedge^{#1}_{#2}(#3)}

\newcommand{\asymm}{\circ}
\newcommand{\asym}[3]{\mathrm{S}^{#1}_{#2}(#3)}

\newcommand{\kind}[1]{K^{\mathrm{\small ind}}_3(#1)}

\newcommand{\ho}[3]{\mathrm{H}_{#1}\left( #2,#3 \right)}
\newcommand{\qho}[3]{\overline{\mathrm{H}}_{#1}\left( #2,#3 \right)}

\newcommand{\ind}[2]{\mathrm{Ind}^{#1}_{#2}} 

\newcommand{\Cor}[2]{\mathrm{cor}_{#1}^{#2}}
\newcommand{\Res}[2]{\mathrm{res}_{#2}^{#1}}

\newcommand{\supp}[1]{\mathrm{supp}(#1)}

\title{Bloch Groups of Rings}
\author{Rodrigo Cuitun Coronado,  Kevin Hutchinson}
\email{rodrigo.cuituncoronado@gmail.com}
\email{kevin.hutchinson@ucd.ie}
\date{\today}


\begin{document}
\maketitle
\tableofcontents
\begin{abstract} We give a definition of  (refined) Bloch groups of general commutative rings which agrees with the standard definition in the case of local rings whose residue field has at least $4$ elements. 
Under appropriate conditions on a ring $A$, satisfied by any field or local ring, these groups are closely related to third homology of $\mathrm{SL}_2(A)$ and to indecomposable $K_3$ of $A$. We analyze these conditions.
As examples, we calculate the  Bloch groups of $\mathbb{F}_2,\mathbb{F}_3,\mathbb{Z}$ and $\mathbb{Z}[\frac{1}{2}]$ and their relation to the homology of $\mathrm{SL}_2$ and $K$-theory.
\end{abstract}

\section{Introduction}
The scissors congruence group, or pre-Bloch group, $\pb{F}$ of a field $F$ with at least $4$ elements (see \cite{sah:dupont}) is a an abelian group defined  by an explicit  presentation, whose subgroup the Bloch group, $\bl{F}$, describes the indecomposable $K_3$ of the field, $\kind{F}$,  modulo some known torsion coming from the group $\mu_F$ of roots of unity (\cite{sus:bloch}). (See section \ref{sec:kind} below for a more precise statement.)
 The refined scissors congruence group $\rpb{F}$  and the refined Bloch group $\rbl{F}$ were introduced in \cite{hut:cplx13} to help to understand the kernel of the natural 
surjective homomorphism $\ho{3}{\spl{2}{F}}{\Z}\to\kind{F}$. These groups and their properties were then used to calculate $\ho{3}{\spl{2}{F}}{\Z[\frac{1}{2}]}$ for local fields $F$ as well as well as for $F=\Q$ (\cite{hut:rbl11},\cite{hut:sl2Q}). The definitions and techniques in these results extend readily from fields to local rings with sufficiently large residue fields and thus allow us to describe $\kind{A}$ and $\ho{3}{\spl{2}{A}}{\Z[\frac{1}{2}]}$ for many local rings $A$ (\cite{mirzaii:bwlocal},\cite{hut:slr}) . 

However, the definitions of these (refined) pre-Bloch groups do not give anything useful in the case of the fields $\F{2}$ and $\F{3}$ or in the case of local rings with small residue field. In \cite{hut:rbl11} an ad-hoc definition of Bloch groups of the fields $\F{2}$ and $\F{3}$ was given so that the main results could also be stated and proved in these cases. The ad hoc nature of these definitions is unsatisfactory, however. Furthermore, it was not clear how to define Bloch groups for local rings with small residue fields. In \cite{hut:slr}, for example, the main results are stated and proved only for local rings whose residue field is sufficently large. 

The purpose of the present article is to extend the definition and fundamental properties of (pre-)Bloch groups to all local rings, regardless of the size of the residue field. In fact we define functorially  the pre-Bloch group and refined pre-Bloch group of an arbitrary commutative ring $A$  in terms of a certain complex $\LL_\bullet$ of $\pgl{2}{A}$-modules (Section \ref{sec:ll}). In section \ref{sec:spectral} we examine the hyperhomology spectral sequence in which the pre-Bloch group naturally features. The Bloch group is the corresponding $E^{\infty}$-term and is a quotient of the hyperhomology  group $\ho{3}{\spl{2}{A}}{\LL_\bullet}$. In section \ref{sec:h3sl2}, we recall the precise conditions under which $\LL_\bullet$ is acyclic and hence the hyperhomology group   $\ho{3}{\spl{2}{A}}{\LL_\bullet}$ is naturally identified with $\ho{3}{\spl{2}{A}}{\Z}$.
 In section \ref{sec:classical}, we clarify the relationship between the Bloch group as defined here and the more classical description as a presentation with generators $\gpb{u}$ and a family of $5$-term dilogarithm relations.

 In section \ref{sec:prebloch} below, we develop the main properties to be used in applications in this generality. It turns out that some known properties admit simpler proofs in this more general setting. In particular, for \emph{any commutative ring $A$} we can define, in the refined pre-Bloch group $\rpb{A}$, the constant $\bconst{A}$ and the elements $\suss{1}{u}$ and $\suss{2}{u}$, $u\in A^\times$  and show that these satisfy certain  fundamental algebraic identities, including the \emph{key identity} $2\pf{u}\bconst{A}=\suss{2}{u}-\suss{1}{u}$. This identity is the  starting point in using the refined pre-Bloch group to calculate the third homology of $\spl{2}{A}$ for local fields and local rings $A$, as well as in calculating $\ho{3}{\spl{2}{\Q}}{\Z[\frac{1}{2}]}$. The original proof of this identity for fields with at least $4$ elements, in \cite{hut:rbl11}, is long, complicated and obscure. Its extension to local rings in \cite{hut:slr} was similarly complicated and again required the residue field to have at least $5$ elements. By contrast, the proof of this identity given below is short, direct and valid for all commutative rings.

In the final section of the article, we verify that the Bloch groups and  refined Bloch groups of $\F{2}$ and $\F{3}$ have the expected relation to $\kind{\F{2}}$ and $\kind{\F{3}}$ and to 
$\ho{3}{\spl{2}{\F{2}}}{\Z}$ and $\ho{3}{\spl{2}{\F{3}}}{\Z}$  (and, furthermore, that they agree with the ad hoc definitions mentioned above). We show that the pre-Bloch group of $\Z$ is cyclic of order $6$, generated by $\bconst{\Z}$  while the Bloch group of $\Z$, $\bl{\Z}$,  is the 
cyclic subgroup of order $3$ in $\pb{\Z}$.  We show that  $\bl{\Z[\frac{1}{2}]}$ is cyclic of order $6$ with generator $\bconst{\Z[\frac{1}{2}]}$  (and is naturally isomorphic to $\bl{\Q}$).

We include an appendix where we give complete details of the calculation of certain differentials in the spectral sequences relating Bloch groups to the homology of $\spl{2}{A}$ and $\gl{2}{A}$  (and which we rely on the body of the article). While versions of some of these calculations appear in earlier articles, we believe it will be of benefit to the reader to have them available here. Furthermore, in contrast to earlier calculations, all of the calculations in the appendix below are valid for arbitrary commutative rings.

In subsequent articles  (\cite{hut:blochf2}, \cite{hut:blochf3}), the second author will build on the results presented here to treat in detail the case of local rings $A$ with residue field $\F{2}$ or $\F{3}$, in particular giving 
a complete set of generators and relations for the the module $\rpb{A}$ and the group $\pb{A}$ in these cases.

Finally, we note that the following question is still open for many local rings:

Is it true that for any local ring $A$ one has exact sequences 
\[
\xymatrix{
0\ar[r]& \widetilde{\mathrm{tor}(\mu_A,\mu_A)}\ar[r]& \ho{3}{\spl{2}{A}}{\Z}\ar[r]& \rbl{A}\ar[r]& 0\\
}
\]
and
\[
\xymatrix{
0\ar[r]&\widetilde{\mathrm{tor}(\mu_A,\mu_A)}\ar[r]&\kind{A}\ar[r]&\bl{A}\ar[r]&0?\\
}
\]
We note that recent progress on these questions has been made  by Mirzaii and Torres P\'{e}rez, for example in  \cite{mirzaii:perezchar2} and \cite{mirzaii:perezpg}.

\subsection{Notation and conventions}
 In this article all  rings are commutative and have a unit. For a ring $A$, $A^\times$ denotes the group of units. $\sgr{A}$ denotes the group ring
$\Z[A^\times/(A^\times)^2]$ of the group of square classes of units. For a unit $u\in A^\times$, the image in $\sgr{A}$ will be denoted $\an{u}$. The augmentation ideal $\aug{A}$ of the ring $\sgr{A}$ is the additive subgroup generated by the elements $\pf{u}:=\an{u}-1$, $u\in A^\times$.  

For a (usually multiplicative) abelian group $G$, $\Extpow{n}{\Z}{G}$ will denote the $n$th exterior power  $\bigotimes^n{G}/I$ where $I$ is the subgroup generated by 
the elements $g_1\otimes\cdots \otimes g_n$ with $g_i=g_{i+1}$ for some $i$. We also let $G\wedge G:=\Extpow{2}{\Z}{G}$. The image of $g_1\otimes\cdots \otimes g_n$ in 
$\Extpow{n}{}{G}$ is denoted $g_1\wedge\cdots \wedge g_n$.

 Likewise, $\asym{n}{\Z}{G}$ denotes the quotient  $\bigotimes^n{G}/J$ where $J$ is the subgroup generated by
the elements $g_1\otimes\cdots g_i\otimes g_{i+1}\cdots \otimes g_n+g_1\otimes\cdots g_{i+1}\otimes g_{i}\cdots \otimes g_n$. The image of $g_1\otimes\cdots \otimes g_n$ in 
$\asym{n}{\Z}{G}$ is denoted $g_1\asymm\cdots \asymm g_n$.


\section{The complexes $L_\bullet$ and $\LL_\bullet$  and pre-Bloch groups}\label{sec:ll}

Let $A$ be a  ring and let $\Gamma(A)$ be the following associated graph. The vertices of $\Gamma(A)$ are equivalence classes, $[u]$, of unimodular rows $u=(u_1,u_2)\in A^2$ under scalar multiplication by units in $A$. The pair $\{ [u],[v]\}$ is an edge in $\Gamma(A)$ if the matrix 
\[
M=
\left[
\begin{array}{c}
u\\
v\\
\end{array}
\right]
\]
lies in $\gl{2}{A}$; i.e., if $\mathrm{det}(M)\in A^\times$.

Recall that a subset $S$ of the vertices of a graph $\Gamma$ is a \emph{clique} if every pair of elements of $S$ is an edge of the graph. We let $X_1(A)$ denote the set of vertices of $\Gamma(A)$ and for $n\geq 1$, we let 
\[
X_n=X_n(A):=\{ (x_0,\ldots,x_n)\in X_1(A)^n\ |\ \{x_0,\ldots,x_n\} \mbox{ is a clique in $\Gamma(A)$}\}.
\]
For $n\geq 0$, let $L_n=L_n(A):=\Z[X_{n+1}(A)]$. These groups form a complex $L_\bullet$ by equipping them with the standard simplicial boundary:
\[
d_n:L_n\to L_{n-1},\quad (x_0,\ldots,x_n)\mapsto \sum_{i=0}^n(-1)^i(x_0,\ldots,\widehat{x_i},\ldots,x_n).
\]
This is the (oriented) clique complex of the graph $\Gamma(A)$.

The sets $X_n(A)$ are naturally  right $\pgl{2}{A}$-sets, and thus $L_\bullet$ is a complex of right $\pgl{2}{A}$-modules (and hence also a complex of right modules over each of $\gl{2}{A},\pspl{2}{A}$ and $\spl{2}{A}$).

Let $\LL_\bullet:=\LL_\bullet(A)$ be the following truncation of the complex $L_\bullet$:
\[
\LL_n(A)=\left\{
\begin{array}{ll}
L_n(A),& n\leq 2\\
Z_2(L_\bullet):=\ker{L_2(A)\to L_1(A)},& n=3\\
0,& n>3
\end{array}
\right.
\]
(where the map $\LL_3(A)=Z_2(L_\bullet)\to \LL_2(A)=L_2(A)$ is the inclusion homomorphism).

 Let $\epsilon:L_0(A)=Z[X_0]\to \Z$ be the homomorphism sending each $x\in \in X_0$ to $1$. If we regard $\Z$ as a complex of (trivial) $\pgl{2}{A}$-modules concentrated in dimension $0$, then $\epsilon$ induces maps of complexes $\epsilon:\LL_\bullet \to \Z$ and $\epsilon:L_\bullet\to \Z$.

Observe that $L_\bullet$ and $\LL_\bullet$ define functors from commutative rings to complexes and that  there is a natural map of complexes of right $\pgl{2}{A}$-modules $L_\bullet\to \LL_\bullet$.

 If $M$ is any right $\pgl{2}{A}$-module then the module of coinvariants 
$M_{\spl{2}{A}}=M_{\pspl{2}{A}}$  is a module over the group $\pgl{2}{A}/\pspl{2}{A}\cong A^\times/(A^\times)^2$ and hence over the group ring $\sgr{A}:=\Z[A^\times/(A^\times)^2]$.
Here, the square class $\an{a}$ acts via right-multiplication by $X$ where $X\in \gl{2}{A}$ is any matrix with determinant $a$. 

We will use the following notations below: For any $a\in A$, $a$ denotes (the class of)  $(a,1)$ in $X_0(A)$. Given $a,b\in A$ we will also denote the class of $(a,b)$ by $a/b$.  Furthermore, we set $\infty:= 1/0=(1,0)\in X_0(A)$.
Note that it is true for any commutative ring  $A$ that the three vertices $0,1,\infty\in X_1(A)$ are connected to each other by edges; i.e,  $\{ 0,1,\infty\}$ is a $3$-clique. Thus we always have at least $6$ elements in $X_2(A)$; namely, the orbit of $(0,\infty,1)$ under $S_3$. 

Note that for $x\in A$, $\{ 0,x\}$ is an edge of $\Gamma(A)$ if and only if $x\in A^\times$ and $\{ 1,x\}$ is an edge if and only if  $1-x\in A^\times$.  For any commutative ring $A$, we let
$\wn{A}:=\{ x\in A\ |\ x(1-x)\in A^\times\}$.  Observe that if $x\in \wn{A}$ then $x^{-1},1-x\in \wn{A}$ and hence $\wn{A}$ is naturally acted on by the nonabelian group of order $6$ generated by these two involutions. Thus if $x\in \wn{A}$, then the following $6$ elements also lie in $\wn{A}$:
\[
x,\ \frac{1}{x},\ 1-x,\ \frac{x-1}{x},\ \frac{1}{1-x},\ \frac{x}{x-1}.
\]

 We summarize the basic facts about the $\pgl{2}{A}$-sets $X_n(A)$ :

 Virtually by definition, the group $\pgl{2}{A}$ acts transitively on $X_1(A)$, and hence on the set of $1$-simplices $Y_1(A)$; i.e., if $a=(a_1,a_2), b=(b_1,b_2)$ are the rows of an invertible matrix, then the resulting matrix 
\[
T_{a,b}:=\left[
\begin{array}{cc}
a_1&a_2\\
b_1&b_2\\
\end{array}
\right]\in \pgl{2}{A}
\]
has the property that $(\infty,0)\cdot T_{a,b}=(a,b)$.
 Furthemore, $\pgl{2}{A}$ acts transitively on $X_2(A)$:
Let $\{  a=(a_1,a_2),b=(b_1,b_2),c=(c_1,c_2)\}\in X_2(A)$. Let
\[
M_{a,b,c}:= 
\left[
\begin{array}{cc}
d(b,c)&0\\
0&d(c,a)\\
\end{array}
\right]\cdot T_{a,b}=
\left[
\begin{array}{cc}
a_1d(b,c)&a_2d(b,c)\\
b_1d(c,a)&b_2d(c,a)\\
\end{array}
\right]
\in \pgl{2}{A}
\]
where, whenever $x=(x_1,x_2), y=(y_1,y_2)$ are the rows of a $2\times 2$ invertible matrix, we define
\[
d(x,y):= \mathrm{det}
\left[
\begin{array}{cc}
x_1&x_2\\
y_1&y_2\\
\end{array}
\right]\in A^\times.
\]
Then $(\infty,0,1)\cdot M_{a,b,c}=(a,b,c)$. 

Observe that for any $(a,b,c)\in X_2(A)$, we have $\an{\mathrm{det}(M_{a,b,c})}=\an{d(a,b)d(b,c)d(c,a)}\in A^\times/(A^\times)^2$. For convenience below, we will use the notation
$d(a,b,c):=d(a,b)d(b,c)d(c,a)\pmod{(A^\times)^2}$.

Note that if $\{ a,b,c\}\in Y_2$ then
\[
T_{a,b}^{-1}= \left[
\begin{array}{cc}
b_2&-a_2\\
-b_1&a_1\\
\end{array}
\right]
\mbox{ and }
M_{a,b,c}^{-1}=T_{a,b}^{-1}\cdot \left[
\begin{array}{cc}
d(c,a)&0\\
0&d(b,c)\\
\end{array}
\right]\mbox{ in } \pgl{2}{A}.
\]

The following facts are easily verified:
\begin{lem}  \label{lem:cr}
\begin{enumerate}
\item If $( a,b,c)\in X_2$ then $c\cdot T_{a,b}^{-1}=(d(c,b),d(a,c))\in X_0$.
\item If $(a,b,c),(a,b,d)\in X_2$, then 
\[
u(a,b,c,d):=d\cdot M_{a,b,c}^{-1}=\left( d(c,a)d(d,b),d(a,d)d(b,c)\right)=\frac{ d(a,c)d(b,d)}{d(a,d)d(b,c)}\in A^\times\subset X_0.
\]
\end{enumerate}
\end{lem}

For a commutative ring $A$ let 
\[
B=B_A:= \left\{ 
\left(
\begin{array}{cc}
a&0\\
c&d
\end{array}
\right)\in \pgl{2}{A}\right\},\quad 
T=T_A:= \left\{ 
\left(
\begin{array}{cc}
a&0\\
0&d
\end{array}
\right)\in \pgl{2}{A}\right\}.
\]
Given a group $G$ and a natural map  $j:G\to\pgl{2}{A}$, let $B(G):=j^{-1}( B)$ and $T(G)=j^{-1}(T)$. Similarly, we let $Z=Z(\pgl{2}{A})=A^\times\cdot I\subset \pgl{2}{A}$.

Define $Z_0(A):=\{ 1\}$ and for $n\geq 1$ let
 \[
Z_n=Z_n(A):=\{ (z_1,\ldots,z_n)\in \wn{A}^n\ |\ z_i/z_j\in \wn{A}\mbox{ for all }i\not=j\}.
\]

Observe that $(\infty,0,1,z_1,\ldots,z_m)\in X_{m+2}$ if and only if $(z_1,\ldots,z_m)\in Z_m$.

\begin{prop}\label{prop:xnsl2} Let $A$ be a ring. 
\begin{enumerate}
\item $\pgl{2}{A}$ acts transitively on $X_0(A)$. The stabilizer of $(\infty)$ is the subgroup $B_A$.
\item $\pgl{2}{A}$ acts transitively on $X_1(A)$. The stabilizer of $(\infty,0)$ is $T_A$.
\item For all $n\geq 2$, $\pgl{2}{A}$ acts freely on $X_n(A)$ and there is a bijection of right $\pgl{2}{A}$-sets
\begin{eqnarray*}
X_n(A)&\leftrightarrow& \bigsqcup_{(u_3,\ldots,u_n)\in Z_{n-2}}(\infty,0,1,u_3,\ldots,u_n)\cdot \pgl{2}{A}\\
([a_0],\ldots,[a_n])&\leftrightarrow & (\infty,0,1,u_3,\ldots,u_n)\cdot M_{a_0,a_1,a_2}\\
\end{eqnarray*}
where $u_i:=u(a_0,a_1,a_2,a_i)$ for $i\geq 3$.
\end{enumerate}

\end{prop}

\begin{cor}\label{cor:lnapgl}
For any ring $A$, there are natural isomorphisms of right $\Z[\pgl{2}{A}]$-modules:
\begin{enumerate}
\item $L_0(A)\cong \Z[B_A\backslash \pgl{2}{A}]\cong \Z\otimes_{\Z[B_A]}\Z[\pgl{2}{A}]=\ind{\pgl{2}{A}}{B_A}\Z$,
\item $L_1(A)\cong \Z[T_A\backslash \pgl{2}{A}]\cong \Z\otimes_{\Z[T_A]}\Z[\pgl{2}{A}]=\ind{\pgl{2}{A}}{T_A}\Z$,
\item For all $n\geq 2$, $L_n(A)\cong \Z[\pgl{2}{A}][Z_{n-2}]$.
\end{enumerate}
\end{cor}

For $v\in A^\times$, let where $M_v:=M_{\infty,0,v}=\mathrm{diag}(v,1)\in \pgl{2}{A}$. Note that $A^\times \cong T=\{ M_v\ |\ v\in A^\times\}$. As a right $\pspl{2}{A}$-set we have 
\[
\pgl{2}{A}\cong \bigsqcup_{v\in A^\times/(A^\times)^2}M_v\cdot\pspl{2}{A}, 
\]
 and hence 
\[
\Z[\pgl{2}{A}]\cong\Z[A^\times/(A^\times)^2\times \pspl{2}{A}]
\]
as right $\Z[\pspl{2}{A}]$-modules.
Thus we deduce:
\begin{cor}\label{cor:lnapsl}
For any ring $A$, there are natural isomorphisms of right $\Z[\pspl{2}{A}]$-modules:
\begin{enumerate}
\item $L_0(A)\cong \Z[B(\pspl{2}{A})\backslash \pspl{2}{A}]\cong \Z\otimes_{\Z[B(\pspl{2}{A})]}\Z[\pspl{2}{A}]=\ind{\pspl{2}{A}}{B(\pspl{2}{A})}\Z$,
\item $L_1(A)\cong \Z[T(\pspl{2}{A})\backslash \pspl{2}{A}]\cong \Z\otimes_{\Z[T(\pspl{2}{A})]}\Z[\pspl{2}{A}]=\ind{\pspl{2}{A}}{T(\pspl{2}{A})}\Z$,
\item For all $n\geq 2$, $L_n(A)\cong \Z[A^\times/(A^\times)^2\times\pspl{2}{A}][Z_{n-2}]$.
\end{enumerate}
\end{cor}

Taking appropriate coinvariants, and recalling that  $\sgr{A}$ is the group ring $\Z[A^\times/(A^\times)^2]$, we immediately deduce:

\begin{cor}\label{cor:lnacoinv} Let $A$ be a ring.
\begin{enumerate}
\item The abelian groups $(L_0(A))_{\pgl{2}{A}}$ and $(L_1(A))_{\pgl{2}{A}}$ are isomorphic to  $\Z$.
\item For all $n\geq 2$, there are isomorphisms of abelian groups $(L_n(A))_{\gl{2}{A}}\cong \Z[Z_{n-2}]$.
\item The $\sgr{A}$-modules $(L_0(A))_{\pspl{2}{A}}$ and $(L_1(A))_{\pspl{2}{A}}$ are isomorphic to  $\Z$.
\item For all $n\geq 2$, there are isomorphisms of $\sgr{A}$-modules 
$(L_n(A))_{\pspl{2}{A}}\cong \sgr{A}[Z_{n-2}]$.
\end{enumerate}
\end{cor}

\begin{cor}\label{cor:empty} If $A$ is a ring satisfying $\wn{A}=\emptyset$ then $L_n(A)=0$ for all $n\geq 3$. 
\end{cor}

Note that $\wn{\Z/2}=\emptyset$ and hence if $A$ admits a ring homomorphism to $\Z/2=\F{2}$, then $\wn{A}=\emptyset$. For example, $L_n(\Z)=0$ for all $n\geq 3$.

For any ring $A$, we define the refined pre-Bloch group and the pre-Bloch group:
\begin{eqnarray*}
\rpb{A}:=\LL_3(A)_{\pspl{2}{A}}\\ 
\pb{A}:=\LL_3(A)_{\pgl{2}{A}}.\\
\end{eqnarray*}
 (Recall that $\LL_3(A):=Z_2(L_\bullet(A))=\ker{d_2:L_2(A)\to L_1(A)}$.)

Given $x\in\wn{A}$, we have $(\infty,0,1,x)\in X_4(A)$. We let $\agpb{x}$ denote the corresponding class  in each of $L_3(A)_{\spl{2}{A}}$ and $L_3(A)_{\gl{2}{A}}$.

Thus $d_3((\infty,0,1,x))=(0,1,x)-(\infty,1,x)+(\infty,0,x)-(\infty,0,1)\in \LL_3(A)$. We denote by $\gpb{x}$ the corresponding class in each of  
$\rpb{A}$ and $\pb{A}$. So the natural maps induced by $d_3$
\[
L_3(A)_{\spl{2}{A}}\to \rpb{A}\mbox{ and }\  L_3(A)_{\gl{2}{A}}\to \pb{A}
\]
 send $\agpb{x}$ to $\gpb{x}$.

By Lemma \ref{lem:cr}, if $(a,b,c,d)\in L_3(A)$, then 
\[
(a,b,c,d)=(\infty,0,1,u(a,b,c,d))\cdot M_{a,b,c}=\agpb{u(a,b,c,d)}\cdot M_{a,b,c}=\agpb{\frac{d(a,c)d(b,d)}{d(a,d)d(b,c)}}\cdot M_{a,b,c}
\]
and hence its image in $\LL_3(A)$ is $\gpb{\frac{d(a,c)d(b,d)}{d(a,d)d(b,c)}}\cdot M_{a,b,c}$ and its image in $\rpb{A}$ is 
\[
\an{d(a,b,c)}\gpb{\frac{d(a,c)d(b,d)}{d(a,d)d(b,c)}}.
\]

\begin{rem}\label{rem:gpbold} The elements $\agpb{x}=(\infty,0,1,x)\in \pb{A}$ are used in Dupont-Sah \cite{sah:dupont} (in the case $A=\C$). Elsewhere (eg. in the articles Suslin \cite{sus:bloch} and Hutchinson \cite{hut:cplx13}), the elements $\agpbold{x}:=(0,\infty,1,x)$ are used instead, and then we may also define $\gpbold{x}:=d_3(\agpbold{x})=(\infty,1,x)-(0,1,x)+(0,\infty,x)-(0,\infty,1)$. Note that $(\infty,0,1,x)\cdot\omega=(0,\infty,1,x^{-1})$ in $L_3(A)$ and thus we have 
\[
\agpbold{x}= \agpb{\frac{1}{x}}\cdot\omega \mbox{ and } \gpbold{x}=\gpb{\frac{1}{x}}\cdot\omega \mbox{ in } L_3(A) \mbox{ and }\LL_3(A),
\] 
and hence
\[
\gpbold{x}=\an{-1}\gpb{\frac{1}{x}} \mbox{ in }\rpb{A}\mbox{ and }\gpbold{x}=\gpb{\frac{1}{x}} \mbox{ in } \pb{A}.
\]
These transformations account for some variations  in the explicit formula for elements and relations in the Bloch groups given below from those in other texts. 
\end{rem}

By Proposition \ref{prop:xnsl2} (3) and Corollary \ref{cor:lnacoinv} (4), $\LL_2(A)_{\pspl{2}{A}}\cong \sgr{A}$, via a map, $D$ say,  sending $(a,b,c)$ to the square class 
$\an{\det{(M_{a,b,c})}}=\an{d(a,b,c)}$. 
We let $\lambda_1$ denote the $\sgr{A}$-module map 
\[
\rpb{A}=\LL_3(A)_{\pspl{2}{A}}\to \LL_2(A)_{\pspl{2}{A}}\cong\sgr{A}
\]
induced by the inclusion $\LL_3(A)\to\LL_2(A)$. 

\begin{lem}\label{lem:lambda1}
Let $A$ be a ring. 
\begin{enumerate}
\item The image of $\lambda_1$ is contained in $\aug{A}$.
\item For all $u\in \wn{A}$, $\lambda_1(\gpb{u})=-\pf{u^{-1}}\pf{1-u^{-1}}$.
\end{enumerate}
\end{lem}
\begin{proof}\ 
\begin{enumerate}
\item The map $\sgr{A}\cong \LL_2(A)_{\pspl{2}{A}}\to \LL_1(A)_{\pspl{2}{A}}\cong \Z$ induced by the map \\
$d_2:\LL_2(A)\to \LL_1(A)$ is easily seen to be the augmentation
homomorphism\\
 $\epsilon:\sgr{A}\to \Z$. Thus the image of $\lambda_1$ is contained in $\ker{\epsilon}=\aug{A}$.

\item For $u\in \wn{A}$, $\gpb{u}=(0,1,u)-(\infty,1,u)+(\infty,0,u)-(\infty,0,1)$. Now
\begin{eqnarray*}
D(0,1,u)=\an{-1\cdot (1-u)\cdot u}=\an{(u-1)u}=\an{(u-1)/u}=\an{1-u^{-1}}\\
D(\infty,1,u)=\an{1\cdot (1-u)\cdot .-1}=\an{u-1}=\an{u}\an{1-u^{-1}}\\
\end{eqnarray*}
while $D(\infty,0,u)=\an{u}=\an{u^{-1}}$. Thus
\[
\lambda_1(\gpb{u})=\an{1-u^{-1}}-\an{u^{-1}}\an{1-u^{-1}}+\an{u^{-1}}-\an{1}=-\pf{u^{-1}}\pf{1-u^{-1}}
\]
as required. 
\end{enumerate}
\end{proof}

For any ring $A$, we set 
\[
\rpbker{A}:=\ker{\lambda_1}=\ker{\LL_3(A)_{\pspl{2}{A}}\to\LL_2(A)_{\pspl{2}{A}}} =H_3\left( \LL_\bullet(A)_{\pspl{2}{A}}\right).
\]

\begin{rem} 
If we consider $\LL_\bullet(A)_{\pgl{2}{A}}$ instead, note that the map $\LL_3(A)_{\pgl{2}{A}}\to\LL_2(A)_{\pgl{2}{A}}$  (induced by inclusion $\LL_3(A)\to\LL_2(A)$) is the 
zero map, since the map $\Z\cong \LL_2(A)_{\gl{2}{A}}\to \LL_1(A)_{\pgl{2}{A}}\cong \Z$ is the identity map. Thus we have 
\[
\pb{A}=H_3\left( \LL_\bullet(A)_{\pgl{2}{A}}\right)
\]
for any ring $A$. 
\end{rem}

\section{Bloch groups   and the hyperhomology spectral sequence $E(G,\LL)$}\label{sec:spectral}

Let $A$ be a ring. For any group $G$ and natural group homomorphism $j:G\to \pgl{2}{A}$, we let $T(G,\LL)$ denote the total complex of the double complex 
\[
D_{\bullet,\bullet}=D_{\bullet,\bullet}(G,\LL):=\LL_\bullet\otimes_{\Z[G]}F_\bullet
\]
where $F_\bullet$ is a fixed projective resolution 
of $\Z$ over $\Z[G]$. The homology of this total complex is, by definition, the hyperhomology $\ho{\bullet}{G}{\LL}$. 

More precisely, we set $D_{p,q}=\LL_q\otimes F_p$. Filtering the double complex vertically then gives a spectral sequence $E^r(G,\LL)$ satisfying
\[
E^1_{p,q}(G,\LL)=\ho{p}{G}{\LL_q} \implies \ho{p+q}{G}{\LL}.
\]
We thus have:

\begin{lem}Let $A$ be a ring.  $E^2_{0,3}(\spl{2}{A},\LL)=E^2_{0,3}(\pspl{2}{A},\LL)=\rpbker{A}$ and 
$E^2_{0,3}(\gl{2}{A},\LL)=E^2_{0,3}(\pgl{2}{A},\LL)=\pb{A}$. 

Furthermore, there are natural edge homomophisms 
\begin{eqnarray*}
\ho{3}{\pspl{2}{A}}{\LL}\mbox{ or }\ho{3}{\spl{2}{A}}{\LL}\to\rpbker{A},\\
 \ho{3}{\pgl{2}{A}}{\LL}\mbox{ or }\ho{3}{\gl{2}{A}}{\LL}\to \pb{A}.
\end{eqnarray*}
\end{lem}

For $G=\pgl{2}{A}$ or $\gl{2}{A}$, let $\mathcal{R}=\mathcal{R}_G:= \Z$, equipped with augmentation $\epsilon=\mathrm{id}:\Z\to \Z$. For $G=\pspl{2}{A}$ or $\spl{2}{A}$, let 
$\mathcal{R}=\mathcal{R}_G:= \sgr{A}$, the group ring $\Z[A^\times/(A^\times)^2]$ with its natural augmentation.

\begin{lem}\label{lem:e1} For $G=\pspl{2}{A}$, $\spl{2}{A}$, $\pgl{2}{A}$ or $\gl{2}{A}$,
 the $E^1$-page of the spectral sequence $E(G,\LL)$  has the form
\[
\xymatrix{
0\ar[d]&0\ar[d]&0\ar[d]&\cdots\\
(\LL_3)_G\ar^-{d^1}[d]&\ho{1}{G}{\LL_3}\ar^-{d^1}[d]&\ho{2}{G}{\LL_3}\ar^-{d^1}[d]&\cdots\\
\mathcal{R}\ar^-{\epsilon}[d]&\mathcal{R}\otimes_{\Z}\ho{1}{Z(G)}{\Z}\ar^-{\epsilon\otimes H_1(\mathrm{inc})}[d]&\mathcal{R}\otimes_{\Z}\ho{2}{Z(G)}{\Z}\ar^-{\epsilon\otimes H_2(\mathrm{inc})}[d]&\cdots\\
\Z\ar^-{0}[d]&\ho{1}{T(G)}{\Z}\ar^-{H_1(\mathrm{inc})\circ(\tilde{\omega}-\mathrm{id})}[d]&\ho{2}{T(G)}{\Z}\ar^-{H_2(\mathrm{inc})\circ(\tilde{\omega}-\mathrm{id})}[d]&\ho{3}{T(G)}{\Z}\ar^-{H_3(\mathrm{inc})\circ(\tilde{\omega}-\mathrm{id})}[d]\\
\Z&\ho{1}{B(G)}{\Z}&\ho{2}{B(G)}{\Z}&\ho{3}{B(G)}{\Z}\\
}
\]
where $\omega$ denotes the map induced on homology by conjugation by $\tilde{\omega}:=\left[
\begin{array}{cc}
0&-1\\
1&0\\
\end{array}
\right]$.
\end{lem}

\begin{proof} By Proposition \ref{prop:xnsl2} and Corollary \ref{cor:lnacoinv}, for $G$ as above we have $\LL_0\cong \Z[B(G)\backslash G]$ and hence 
$E^1_{p,0}(G,\LL)=\ho{p}{G}{\LL_0}\cong \ho{p}{B(G)}{\Z}$ by Shapiro's Lemma. 

 Similarly, $\LL_1\cong \Z[T(G)\backslash G]$ and hence $E^1_{p,1}(G,\LL)\cong \ho{p}{T(G)}{\Z}$. 

By Corollary \ref{cor:lnacoinv} again, $\LL_2\cong \Z[Z\backslash \gl{2}{A}]$ and hence $E^1_{p,2}(\gl{2}{A},\LL)\cong\ho{p}{Z}{\Z}=\mathcal{R}\otimes_{\Z}\ho{p}{Z(G)}{\Z}$. 

By Proposition \ref{prop:xnsl2}, for $G=\pspl{2}{A}$ or $\spl{2}{A}$, there is a $\Z[G]$-decomposition
\[
\LL_2\cong \oplus_{\an{u}\in A^\times/(A^\times)^2}\Z[Z(G)\backslash G]\cdot (\infty,0,1,u)
\]
and hence
\[
E^1_{p,2}(G,\LL)=\ho{p}{G}{\LL_2}\cong \oplus_{\an{u}\in A^\times/(A^\times)^2}\ho{p}{Z(G)}{\Z}\cong \mathcal{R}\otimes_\Z \ho{p}{\Z(G)}{\Z}.
\]
This accounts for the $E^1$-terms in the spectral sequences. The calculation of the differentials is given in Corollary \ref{cor:d1p1} and Lemma \ref{lem:d1p2} below. 
\end{proof}

\begin{rem} When $G=\pspl{2}{A}$ or $\spl{2}{A}$ all terms and differentials in the spectral sequence $E^r(G,\LL)$ are naturally $\sgr{A}$-modules and homomorphisms respectively. 
\end{rem}
\begin{cor}\label{cor:11} Let $A$ be a ring. For $G=\pgl{2}{A}$, $\gl{2}{A}$, $\pspl{2}{A}$   or $\spl{2}{A}$,
 $
E^2_{1,1}(G,\LL)=0.
$
\end{cor}
\begin{proof} $E^2_{1,1}(G,\LL)$ is the homology of 
\[
\xymatrix{
E^1_{1,2}(G,K)\ar^-{d^1}[r] &E^1_{1,1}(G,K)\ar^-{d^1}[r]& E^1_{1,0}(G,K).\\
}
\]
 This is thus the homology of 
\[
\xymatrix{
\mathcal{R}\otimes Z(G)=\mathcal{R}\otimes H_1(Z(G),\Z)\ar^-{\epsilon\otimes \mathrm{inc}}[r] &T(G)=H_1(T(G),\Z)\ar^-{\mathrm{inc}\ \circ({\omega}-\mathrm{id})}[rr]&&\ho{1}{B(G)}{\Z}\\
}
\]
where $\omega$ denotes the map induced by conjugation by $\omega$.

Since $\mathrm{inc}: T(G)=\ho{1}{T(G)}{\Z}\to \ho{1}{B(G)}{\Z}$is a split injection, this is equal to the homology of 
\[
\xymatrix{
\mathcal{R}\otimes Z(G)\ar^-{\epsilon\otimes \mathrm{inc}}[r] &T(G)\ar^-{{w}-\mathrm{id}}[r]& T(G),\\
}
\]
which is easily verified to be $0$ for $G=\gl{2}{A}$ or $G=\spl{2}{A}$.
\end{proof}

It follows that for $G=\pgl{2}{A}$, $\gl{2}{A}$, $\pspl{2}{A}$   or $\spl{2}{A}$ the differential $d^2:E^2_{0,3}\to E^2_{1,1}$ is the zero map.
\begin{cor} \label{cor:e3}
Let $A$ be a ring. Then $E^3_{0,3}(\spl{2}{A},\LL)=\rpbker{A}$ and $E^3_{0,3}(\gl{2}{A},\LL)=\pb{A}$.
\end{cor}

The \emph{Bloch groups} of the ring are defined to be the $E^\infty_{0,3}$-terms:

Let $A$ be a ring. We define the \emph{refined Bloch group} of $A$, $\rbl{A}$, to be the $\sgr{A}$-module $E^\infty_{0,3}(\spl{2}{A},\LL)$. Thus 
\[
\rbl{A}:= \ker{d^3:\rpbker{A}\to E^3_{2,0}(\spl{2}{A},\LL)}.
\]

 We define the\emph{ Bloch group of $A$}, $\bl{A}$, to be the abelian group  $E^\infty_{0,3}(\gl{2}{A},\LL)$. Thus 
\[
\bl{A}:= \ker{d^3:\pb{A}\to E^3_{2,0}(\gl{2}{A},\LL)}.
\]


From the definitions, we have 
\begin{lem} For any ring $A$ there are commutative diagrams of $\sgr{A}$-modules
\[
\xymatrix{
\ho{3}{\spl{2}{A}}{\LL}\ar@{>>}[r]\ar[d]&\rbl{A}\ar[d]\\
\ho{3}{\gl{2}{A}}{\LL}\ar@{>>}[r]&\bl{A}\\
}
\]
where the horizontal arrows are surjections.
\end{lem}

\begin{rem} Since the diagram 
\[
\xymatrix{
\rpbker{A}=E^3_{0,3}(\spl{2}{A},\LL)
\ar^-{d^3}[r]\ar^-{=}[d]&E^3_{2,0}(\spl{2}{A},\LL)\ar[d]\\
\rpbker{A}=E^3_{0,3}(\pspl{2}{A},\LL)\ar^-{d^3}[r]&E^3_{2,0}(\pspl{2}{A},\LL)\\
}
\]
commutes, we always  have  $\rbl{A}:=E^\infty_{0,3}(\spl{2}{A},\LL)\subset E^\infty_{0,3}(\pspl{2}{A},\LL)$. However, examples  
 show that we do not have equality
in general, since the map $E^3_{2,0}(\spl{2}{A},\LL)\to E^3_{2,0}(\pspl{2}{A},\LL)$ may not be surjective. See, for example, Remark \ref{rem:psl2z} below.
\end{rem}

\section{Connection with $\ho{3}{\spl{2}{A}}{\Z}$}\label{sec:h3sl2}

For the purposes of these notes, we introduce the following terminology:

Let us say that a commutative ring $A$ is \emph{$L_\bullet$-acyclic in dimensions $\leq r$} (resp. $\LL_\bullet$-acyclic in dimensions $\leq r$)  if the map $\epsilon:L_\bullet(A)\to \Z$  
(resp. $\epsilon:\LL_\bullet(A)\to \Z$) induces an isomorphism on homology in dimensions $\leq r$.
Equivalently, $A$ is $L_\bullet$-acyclic in dimensions $\leq r$ if the sequence 
\[
L_{r+1}(A)\to L_r(A)\to\cdots \to L_0(A)\to \Z\to 0
\]
is exact.

We will say that $A$ is \emph{$L_\bullet$-acyclic} if it is $L_\bullet$-acyclic in dimensions $\leq r$ for all $r\geq 0$. Note that $A$ is $\LL_\bullet$-acyclic if and only if it is 
$L_\bullet$-acyclic in dimensions $\leq 1$ (since $\LL_3(A):= \ker{d_2}$). 

From the definition, we immediately have

\begin{prop} \label{prop:lacyclic} Let $A$ be a $\LL_\bullet$-acyclic ring and let $G$ be a subgroup of $\gl{2}{A}$.
There are isomorphisms 
 \[
\ho{n}{G}{\LL_\bullet}\cong \ho{n}{G}{\Z}
\]
for $n\geq 0$.

Thus, the spectral sequence $E^r(G,\LL_\bullet)$ converges to $\ho{\bullet}{G}{\Z}$. 
\end{prop}
\begin{proof}
 The  equivalence $\LL_\bullet \to \Z$ induces an isomorphism $\ho{\bullet}{G}{\LL_\bullet}\cong\ho{\bullet}{G}{\Z}$. 
\end{proof}
\begin{thm} \label{thm:local} A local ring\footnote{For convenience, the term ``local ring $A$ with residue field $k$" will include the case when $A=k$ is a field.}  with residue field $k$ is $L_\bullet$-acyclic in dimensions $\leq |k|-1$.

In particular, every local ring  is $\LL_\bullet$-acyclic. 

\end{thm}

\begin{proof}
We have  $H_r(L_\bullet(A))=0$ for $r<|k|$  by \cite[Lemma 3.21]{hut:slr}. 

Furthermore
\[
\xymatrix{
 L_1(A)\ar[r]& L_0(A)\ar^-{\epsilon}[r]&\Z\ar[r]& 0\\
}
\]
is exact: $\ker{\epsilon}$ is generated by the elements $\bar{u}-\bar{v}$, $\bar{u}\not=\bar{v}\in X_1$. Here 
$u=(u_1,u_2)$ and $v=(v_1,v_2)$ 
are elements of $U_2(A)$ representing $\bar{u}$ and $\bar{v}$ in $X_1$. Let $k$ be the residue field of $A$ and let  $x,y$ be the images respectively of $\bar{u},\bar{v}$  in $X_1(k)=\projl{k}$. Choose $z\in \projl{k}\setminus\{ x,y\}$ and let $\bar{w}$ be an inverse image of $z$ in $X_1(A)$. Then $(\bar{u},\bar{w}),
(\bar{v},\bar{w})\in X_2(A)$ and $d_1\left((\bar{u},\bar{w})-(\bar{v},\bar{w})\right)=\bar{u}-\bar{v}\in \LL_0$. 
\end{proof}

The conditions under which a ring $A$ is $\LL_\bullet$-acyclic  have been determined in \cite{hut:ge2arxiv}:

Recall that a ring $A$ is a \emph{ $\mathrm{GE}_2$-ring} (P. M. Cohn,\cite{cohn:gln})  if $\spl{2}{A}=E_2(A)$ where $E_2(A)$ is the subgroup generated by elementary matrices. 

For a ring $A$, $K_2(2,A)$ denotes the rank one $K_2$ of $A$ and $C(2,A)$ denotes the subgroup of $K_2(2,A)$ generated by symbols (see, for example, \cite[Appendix]{hut:ge2}).

\begin{thm}[{\cite[Theorem 3.3, Theorem 7.2]{hut:ge2},\cite[Theorem 5.2]{mirzaii:perezpg}}]\label{thm:acyclic}
Let $A$ be a ring. Then $H_0(L_\bullet(A))\cong \Z[E_2(A)\backslash \spl{2}{A}]$ and 
\[
H_1(L_\bullet(A))\cong \bigoplus_{E_2(A)\backslash \spl{2}{A}}\left(\frac{K_2(2,A)}{C(2,A)}\right)^{\mathrm{ab}}.
\]
\end{thm}

\begin{cor}\label{cor:acyclic} The ring $A$ is $\LL_\bullet$-acyclic if and only if $A$ is a $\mathrm{GE}_2$-ring and the group $K_2(2,A)/C(2,A)$ is perfect. 
\end{cor}

For a ring $A$, we have $K_2(2,A)=C(2,A)$ if and only if $A$ is \emph{universal for $\mathrm{GE}_2$} in the terminology of P. M. Cohn  \cite{cohn:gln} (see \cite[Appendix]{hut:ge2} for a proof of this statement). Thus if $A$ is a $\mathrm{GE}_2$-ring which is universal for $\mathrm{GE}_2$ then $A$ is $\LL_\bullet$-acyclic. There are many examples of such rings in \cite{cohn:gln} and elsewhere.

\begin{exa}\label{exa:eucge2}
Note that any Euclidean domain is a $\mathrm{GE}_2$-ring, but there are many examples of PIDs which are not $\mathrm{GE}_2$-rings (see, for example, \cite{czz}).
\end{exa}
\begin{exa}\label{exa:vas} L. Vaserstein (\cite{vas}) has shown  that a ring of $S$-integers in a number field with infinitely many units is a $\mathrm{GE}_2$-ring.
\end{exa}
\begin{exa}\label{exa:acycz}
 $\Z$ is a Euclidean domain and $K_2(2,\Z)$ is generated by the symbol $c(-1,-1)$. Thus $\Z$ is $\LL_\bullet$-acyclic.
\end{exa}
\begin{exa}\label{exa:acycz1m}
Morita (\cite[Proposition  2:13]{morita:k2zs}) has shown that if $D$ is Dedekind domain and if $\pi\in D$ is prime element for which  $D^\times$ surjects onto $(D/\an{\pi})^\times$ and if $K_2(2,D)$ is generated by symbols then the same holds for the ring $D[\frac{1}{\pi}]$. 

This fact, and an argument by induction, shows that the Euclidean domains $\Z[\frac{1}{m}]$
are $\LL_\bullet$-acyclic whenever there exist primes $p_1,\cdots, p_t$ such that $m=p_1^{a_1}\cdots p_t^{a_t}$  and $(\Z/p_i)^\times$ is generated by $\{ -1, p_1,\ldots, p_{i-1}\}$ 
for all $i\leq t$.  In particular, $\Z[\frac{1}{2}]$ and $\Z[\frac{1}{3}]$ are $\LL_\bullet$-acyclic.
\end{exa}

\begin{exa}\label{exa:acyc1p} On the other hand, consider the ring $\Z[\frac{1}{p}]$, $p\geq 5$ is a prime. The calculations of Morita (\cite{morita:mab}) imply that the groups 
$\left(K_2(2,\Z[\frac{1}{p}])/C(2,\Z[\frac{1}{p}])\right)^{\mathrm{ab}}$ are non zero (see \cite[Lemma 6.15]{hut:ge2}). Thus $\Z[\frac{1}{p}]$ is not $\LL_\bullet$-acyclic 
for any $p\geq 5$.

In fact, Morita writes down explicit elements of infinite order in $\left(K_2(2,\Z[\frac{1}{p}])/C(2,\Z[\frac{1}{p}])\right)^{\mathrm{ab}}$. This allows us to write down explicit cycles 
in $L_1(\Z[\frac{1}{p}])$ which represent homology classes of infinite order  (\cite[Example 7.4]{hut:ge2}). For example, the cycle
\[
(\infty,0)+\left( 0,-\frac{1}{10}\right)+\left(-\frac{1}{10},-\frac{3}{5}\right)+\left(-\frac{3}{5},\infty\right)\in L_1\left(\Z\left[\frac{1}{5}\right]\right)
\]
represents a homology class of infinite order.
\end{exa}

\begin{exa}\label{exa:laurent}
For any field $k$, the ring of Laurent polynomials $k[t,t^{-1}]$  is a Euclidean domain. Furthermore, $K_2(2,k[t,t^{-1}])$ is generated by symbols (see \cite{moritarehmann:laurent}).
Thus the ring $k[t,t^{-1}]$ is $\LL_\bullet$-acyclic for any field $k$.
\end{exa}


\section{Classical Bloch group and classical refined Bloch group}\label{sec:classical}

Recall that, for a commutative ring $A$,  $\wn{A}:= \{ a\in A^\times\ |\ 1-a\in A^\times\}$   and for $n\geq 1$, $Z_n=Z_n(A)=\{ (u_1,\ldots,u_n)\in \wn{A}^n\ |\ u_i/u_j\in \wn{A} \mbox{ for all } i\not=j\}$.

Let $\arpb{A}$, the \emph{classical refined pre-Bloch group of $A$},  be the $\sgr{A}$-module generated by $\agpb{u},u\in \wn{A}$ subject to the relations 
\[
0=\agpb{x}-\agpb{y}+\an{x}\agpb{ \frac{y}{x}}-\an{x-1}\agpb{\frac{1-y}{1-x}}
+\an{1-x^{-1}}\agpb{\frac{1-y^{-1}}{1-x^{-1}}},\quad \forall (x,y)\in Z_2(A).
\]

Similarly,  let $\apb{A}$,  the \emph{classical  pre-Bloch group of $A$}, denote the $\Z$-module with generators $\agpb{u}, u\in \wn{A}$ and relations
\[
0=\agpb{x}-\agpb{y}+\agpb{ \frac{y}{x}}-\agpb{\frac{1-y}{1-x}}
+\agpb{\frac{1-y^{-1}}{1-x^{-1}}},\quad \forall (x,y)\in Z_2(A).
\]

\begin{rem} If we use instead the alternative generators $\agpbold{u}:=\an{-1}\agpb{\frac{1}{u}}$, then the defining relations of $\arpb{A}$  instead have the form
\[
0=\agpbold{x}-\agpbold{y}+\an{x}\agpbold{ \frac{y}{x}}
-\an{x^{-1}-1}\agpbold{\frac{1-x^{-1}}{1-y^{-1}}}+\an{1-x}\agpbold{\frac{1-x}{1-y}}.
\]
(See Remark \ref{rem:gpbold} above.)
\end{rem}

These modules are naturally related to the homology of the complex $L_\bullet(A)$:

 By Proposition \ref{prop:xnsl2}, $L_3(A)$ is freely generated as a right $\Z[\pgl{2}{A}]$-module by the elements $\agpb{u}=(\infty,0,1,u)$, $u\in \wn{A}$ and $L_4(A)$ is freely generated by the elements $\agpb{x,y}:=(\infty,0,1,x,y)$, $(x,y)\in Z_2$. Furthermore, for any $(x,y)\in Z_2$, 
\begin{eqnarray*}
d_4(\agpb{x,y})&=&(0,1,x,y)-(\infty,1,x,y)+(\infty,0,x,y)-(\infty,0,1,y)+(\infty,0,1,x)\\
&=&\agpb{\frac{1-y^{-1}}{1-x^{-1}}}\cdot M_{0,1,x}-\agpb{\frac{1-y}{1-x}}\cdot M_{\infty,1,x}+\agpb{\frac{y}{x}}\cdot M_{\infty,0,x}-\agpb{y}+\agpb{x} \\
\end{eqnarray*} 
which maps to 
\[
\agpb{x}-\agpb{y}+\an{x}\agpb{\frac{y}{x}}-\an{x-1}\agpb{\frac{1-y}{1-x}}+\an{1-x^{-1}}\agpb{\frac{1-y^{-1}}{1-x^{-1}}}
\]
in $L_3(A)_{\pspl{2}{A}}$.

Thus 
\[
\arpb{A}\cong (L_3(A)/\image{d_4})_{\pspl{2}{A}} \mbox{ and }\apb{A}\cong(L_3(A)/\image{d_4})_{\pgl{2}{A}}.
\]

Note that the map $d_3:L_3(A)\to \LL_3(A)$ induces a homomorphism of right $\pgl{2}{A}$-modules $L_3(A)/\image{d_4}\to \LL_3(A)$ sending $\agpb{u}$ to $\gpb{u}$. Hence there is an induced homomorphism of $\sgr{A}$-modules 
$\alpha : (L_3(A)/\image{d_4})_{\pspl{2}{A}}\to \LL_3(A)_{\pspl{2}{A}}=\rpb{A}$ and a homomorphism of abelian groups $\beta:(L_3(A)/\image{d_4})_{\pgl{2}{A}}\to \LL_3(A)_{\pgl{2}{A}}=\pb{A}$.


\begin{thm} \label{lem:gen}
Suppose that the ring $A$ is $L_\bullet$-acyclic in dimensions $\leq 2$.  (Recall that this implies that $A$ is $\LL_\bullet$-acyclic. This condition is satisfied by any local ring whose residue field has order at least $3$ .)
\begin{enumerate}
\item  The  $\sgr{A}$-homorphism $\alpha:\arpb{A}\to\rpb{A}$ and the $\Z$-module homomorphism $\beta:\apb{A}\to\pb{A}$
are both surjective.
\item If furthermore $A$ is $L_\bullet$-acyclic in dimension $\leq 3$  (e.g., if $A$ is a local ring with residue field of size $\geq 4$) then the homomorphisms $\alpha$ and $\beta$ induce isomorphisms 
\[
\arpb{A}\cong \rpb{A} \mbox{ and } \apb{A}\cong \pb{A}. 
\]
\end{enumerate}
\end{thm}

\begin{proof}
For any subgroup $G$ of $\gl{2}{A}$ or $\pgl{2}{A}$ we have a  complex of right $\Z[G]$-modules
\[
\xymatrix{
L_4(A)_G\ar^-{d_4}[r]&L_3(A)_G\ar^-{d_3}[r]&\LL_3(A)_G\ar[r]&0
}
\]
which induces a homomorphism
\[
\alpha_G:\frac{L_3(A)_G}{d_4(L_4(A)_G)}\to \LL_3(A)_G.
\]
When $G=\spl{2}{A}$, we have $\alpha_G=\alpha$. When $G=\gl{2}{A}$, $\alpha_G=\beta$. 
\begin{enumerate}
\item If $A$ is $L_\bullet$-acyclic in dimensions $\leq 2$ the map $L_3(A)\to \LL_3(A)$ is surjective. It follows that the map $\alpha_G$ is surjective for any $G$.
\item If $A$ is $L_\bullet$-acyclic in dimensions $\leq 3$ then the sequence $L_4(A)\to L_3(A)\to \LL_3(A)\to 0$ is exact. Taking $G$-coinvariants, the sequence remains exact by right-exactness of coinvariants. Hence the 
map $\alpha_G$ is an isomorphism in this case. 
\end{enumerate}
\end{proof}

\begin{exa} It follows that if $A$ is a local ring with maximal ideal $M$ and residue field $\F{3}$, then $\rpb{A}$ is generated over $\sgr{A}$ by the set $\{ \gpb{u}\ | u\in \wn{A}=-1+{M}\}$.
In particular, $\rpb{\F{3}}$ is a cyclic $\sgr{\F{3}}$-module with generator $\gpb{-1}$.
\end{exa}


Let $A$ be a ring.  Let $\lambda_1:\arpb{A}\to\sgr{A}$ be the $\sgr{A}$-module homomorphism $\gpb{a}\mapsto \pf{a}\pf{1-a}$ for $a\in \wn{A}$. 
Let $\lambda:\apb{A}\to \asym{2}{\Z}{A^\times}$ be the $\Z$-module homomorphism $\gpb{u}\mapsto (1-u^{-1}) \asymm u^{-1}$ for $u\in\wn{A}$ (equivalently,
$\agpbold{u}\mapsto (1-u)\asymm u$).
Let $\lambda_2:\arpb{A}\to\asym{2}{\Z}{A^\times}$ denote the composite homomorphism 
\[
\xymatrix{
\arpb{A}\ar[r]&\apb{A}\ar^-{\lambda}[r]&\asym{2}{\Z}{A^\times}.
}
\]
 
We define the $\sgr{A}$-modules
\[
\arpbker{A}:=\ker{\lambda_1},\quad \arbl{A}:= \ker{\lambda_2:\arpbker{A}\to \asym{2}{\Z}{A^\times}}.
\]
$\arbl{A}$ is the\emph{ classical refined Bloch group of $A$}.

We define the $\Z$-module
\[
\abl{A}:=\ker{\lambda:\apb{A}\to\asym{2}{\Z}{A^\times}}.
\]
This is the \emph{ classical  Bloch group of $A$}.

We relate these modules to the  hyperhomology groups $\ho{n}{G}{L_\bullet}$ for certain subgroups $G$ of $\gl{2}{A}$:

\begin{prop}\label{prop:arbl}  Let $A$ be a ring. Let $T:=T(\spl{2}{A})$ and $B:=B(\spl{2}{A})$.  For any subgroup $G$ of $\gl{2}{A}$ there is a  spectral sequence $E^r_{p,q}(G,L)\Longrightarrow \ho{p+q}{G}{L_\bullet}$.
\begin{enumerate}
\item $\arpbker{A}\cong E^3_{0,3}(\spl{2}{A},L)$.
and $\apb{A}\cong E^3_{0,3}(\gl{2}{A},L)$.
\item If the inclusion $T\to B$ induces an isomorphism $\ho{2}{T}{\Z}\to \ho{2}{B}{\Z}$ then $\arbl{A}= E^\infty_{0,3}(\spl{2}{A},L)$.
\item If the inclusion $T_A\to B_A$ induces an isomorphism $\ho{2}{T_A}{\Z}\cong\ho{2}{B_A}{\Z}$ then $\abl{A}=E^\infty_{0,3}(\gl{2}{A},L)$.
\end{enumerate}
\end{prop}

\begin{proof}
\begin{enumerate}
\item Note that $E^2_{0,3}(G,L)=(\LL_3(A))_G/(d_4(L_4(A)))_G$ for any subgroup $G$ of $ \gl{2}{A}$ and hence $E^2_{0,3}(\spl{2}{A}, L)=\arpbker{A}$  and 
$E^2_{0,3}(\gl{2}{A}, L)=\apb{A}$. However, $E^2_{1,1}(G,L)=0$ when $G=\spl{2}{A}$ or $\gl{2}{A}$ (as in the proof  of Corollary \ref{cor:11}) and hence 
 $\arpbker{A}\cong E^3_{0,3}(\spl{2}{A},L)$.
and $\apb{A}\cong E^3_{0,3}(\gl{2}{A},L)$ as claimed.
\item We consider the relevant terms and differentials in the spectral sequence $E(\spl{2}{A},L)$.  Recall that \[
E^\infty_{0,3}=E^4_{0,3}=\ker{d^3:E^3_{0,3}=\arpbker{A}\to E^3_{2,0}}.
\]

We first calculate $E^3_{2,0}$: The term $E^2_{1,2}$ is the homology of the sequence
\[
\xymatrix{
\sgr{A}[Z_1]\otimes\ho{1}{\mu_2(A)}{\Z}\ar^-{\partial\otimes\mathrm{id}}[r]&\sgr{A}\otimes \ho{1}{\mu_2(A)}{\Z}\ar^-{\epsilon\otimes\mathrm{id}}[r]&\ho{1}{T}{\Z}\\
}
\]
where $\partial:\sgr{A}[Z_1]\to\sgr{A}$ is the $\sgr{A}$-module map $[u]\mapsto \pf{u}\pf{1-u}$ and $\epsilon:\sgr{A}\to \Z$ is the augmentation homomorphism. It follows that 
$E^2_{1,2}=I(A)\otimes\mu_2(A)$ where $I(A):=\ker{\bar{\epsilon}}:\sgr{A}/J_A\to \Z$ with $J_A$ equal to the ideal of $\sgr{A}$ generated by the elements $\pf{u}\pf{1-u}$, $u\in Z_1=\wn{A}$.

Now $E^2_{2,0}=E^1_{2,0}=\ho{2}{B}{\Z}=\ho{2}{T}{\Z}\cong A^\times\wedge A^\times$ and the differential
\[
d^2:E^2_{1,2}\cong I(A)\otimes \mu_2(A)\to A^\times\wedge A^\times \cong E^2_{2,0}
\]
sends $\pf{u}\otimes \eta$ to $u\wedge \eta$. Thus $E^3_{2,0}\cong (A^\times\wedge A^\times)/(A^\times\wedge \mu_2(A))$.

The inclusion map $\spl{2}{A}\to \gl{2}{A}$ induces a map of spectral sequences and a commutative diagram 
\[
\xymatrix{
E^3_{0,3}=\arpbker{A}\ar^-{d^3}[d]\ar[r]&\apb{A}=E^3_{0,3}(\gl{2}{A},L)\ar^-{d^3}[d]\\
\frac{A^\times\wedge A^\times}{A^\times\wedge \mu_2(A)}\ar[r]&E^3_{2,0}(\gl{2}{A},L)\\
}
\]
Now, by Corollary \ref{cor:e3p0} below,  \[
E^3_{2,0}(\gl{2}{A},L)\cong \bar{H}_2(B_A,\Z):=\ho{2}{B_A}{\Z}/\Cor{T_A}{B_A}(\omega-1)
\]
 (where here $\omega$ denotes the map on $\ho{2}{T_A}{\Z}$ induced by conjugation by the matrix $\omega$ on $T_A$). Let $p:\bar{H}_2(B_A,\Z)\to
\bar{H}_2(T_A,\Z):= \ho{2}{T_A}{\Z}/(\omega -1)$ be the map induced by the projection to the diagonal  $B_A\to T_A$. 
Let\\
 $\pi:\ho{2}{T_A}{\Z}=T_A\wedge T_A\to \asym{2}{\Z}{A^\times}$ be the (surjective)
 homomorphism sending \\
$\mathrm{diag}(a,b)\wedge \mathrm{diag}(c,d)$ to $a\circ d+b\circ c$ and let $\bar{\pi}$ the resulting well-defined map $\bar{H}_2(T_A,\Z)\to\asym{2}{\Z}{A^\times}$.
Then, by Proposition \ref{prop:lambda} (2), the diagram
\[
\xymatrix{
\apb{A}\ar^-{d^3}[r]\ar^-{\lambda}[rd]&E^3_{2,0}(\gl{2}{A},L)\ar^-{\bar{\pi}}[d]\\
&\asym{2}{\Z}{A^\times}\\
}
\]
commutes.
 Now the inclusion $T\to T_A$ induces the homomorphism 
\[
j:A^\times\wedge A^\times=\ho{2}{T}{\Z}\to \ho{2}{T_A}{\Z}\to \asym{2}{\Z}{A^\times}, u\wedge v\mapsto -2(u\circ v)
\]
which in turn induces the \emph{injective map} $\bar{j}:E^3_{2,0}=(A^\times\wedge A^\times)/(A^\times \wedge \mu_2(A))\to \asym{2}{\Z}{A^\times}$. Since $\bar{j}$ is injective,
the commutativity of the diagram 
\[
\xymatrix{
\arpbker{A}\ar[r]\ar^-{\lambda_2}[rd]\ar^-{d^3}[d]&\apb{A}\ar^-{\lambda}[d]\\
E^3_{2,0}\ar^-{\bar{j}}[r]&\asym{2}{\Z}{A^\times}\\
}
\]
establishes the result, since $\ker{d^3}=\ker{\lambda_2}$.
\item Under the hypothesis $\ho{2}{T_A}{\Z}\cong \ho{2}{B_A}{\Z}$ and by Corollary \ref{cor:e3p0}, $E^3_{2,0}(\gl{2}{A},L)\cong \frac{T_A\wedge T_A}{\omega -1}$. 
By Proposition \ref{prop:lambda} (1)  the image of $d^3$ lies in the direct summand $\asym{2}{\Z}{A^\times}$.
Thus, by Proposition \ref{prop:lambda} (2) again,  the differential $d^3:E^3_{3,0}(\gl{2}{A},L)\to E^3_{2,0}(\gl{2}{A},L)$  is equal to $\lambda$ (up to sign). 
\end{enumerate}
\end{proof}
\begin{cor} If $A$ is $L_\bullet$-acyclic in dimensions $\leq 3$ and if $\ho{2}{T}{\Z}\cong\ho{2}{B}{\Z}$ (resp. if $\ho{2}{T_A}{\Z}\cong\ho{2}{B_A}{\Z}$) for the ring $A$, then there is a natural isomorphism\\
 $\arbl{A}\cong\rbl{A}$  (resp. $\abl{A}\cong\bl{A}$).
\end{cor}

The hypothesis of Proposition \ref{prop:arbl} (2) applies, for example, to many subrings of $\Q$:
\begin{lem} \label{lem:1/6}
Let $A$ be a subring of $\Q$ satisfying $6\in A^\times$. Then the natural map 
\[
\ho{2}{T}{\Z} \to \ho{2}{B}{\Z}
\]
is an isomorphism.
\end{lem}
\begin{proof} By considering the Hochschild-Serre spectral sequence associated to the group extension 
\[
1\to U\cong A\to  B\to  T\to 1,
\]
it is enough to establish the vanishing of  $\ho{0}{T}{\ho{2}{U}{\Z}}$ and $\ho{1}{T}{\ho{1}{U}{\Z}}$. 
Furthermore, since $A$ is a colimit of infinite cyclic groups, $\ho{2}{A}{\Z}=0$. We  need only prove the vanishing 
of $\ho{1}{T}{\ho{1}{U}{\Z}}=\ho{1}{A^\times}{A}$ where  $u\in A^\times$ acts on $A$ as multiplication by $u^2$.  However, the 
pair of maps $(\iota_u, u\cdot ): (A^\times,A)\to (A^\times,A)$ induces the identity on 
the groups $\ho{i}{A^\times}{A}$, where $\iota_u$ is conjugation by $u$. 
Taking $u=2\in A^\times$, we deduce that multiplication by $4=2^2$ is the identity map on 
$\ho{1}{A^\times}{A}$ and hence that this group is annihilated by $3$. But $3\in A^\times$, by hypothesis, and so acts invertibly on 
$\ho{1}{A^\times}{A}$. Thus  $\ho{1}{A^\times}{A}=0$ as required.    
\end{proof}

We also note:

\begin{lem} \label{lem:1/2}
If $A=\Z[\frac{1}{2}]$ then $\ho{2}{T}{\Z}\cong \ho{2}{B}{\Z}$.
\end{lem}
\begin{proof}
As in the proof of Lemma \ref{lem:1/6}, is is enough to show that $H:=\ho{1}{A^\times}{A}=0$. 
 We have a (split) short exact sequence of groups $1\to \mu_2\to A^\times\to 2^{\Z} \to 1$, where here
$2^{\Z}:=\{2^i:i\in \Z\}\cong \Z$. Since $2$ acts 
invertibly on $A$, the groups $\ho{i}{\mu_2}{A}$ vanish for $i\geq 1$ and the Hochschild-Serre spectral sequence 
induces and isomorphism $\ho{i}{A^\times}{A}\cong  \ho{i}{2^{\Z}}{ A}$ for all $i$. In particular, 
$\ho{1}{A^\times}{A}=\ho{1}{2^{\Z}}{A}=\{ a\in A\ | 3a=0\}=0$. 
\end{proof}

However, there are many rings for which the (split surjective)  map $\ho{2}{B}{\Z}\to\ho{2}{T}{\Z}$ has a nontrivial kernel:

Since the subgroup $\mu_2=\{ \pm I\}$ is central in $B$, multiplication induces a natural homomorphism $B\times\mu_2\to B$ and hence there is an induced product on homology
\[
\ho{p}{B}{\Z}\otimes\ho{q}{\mu_2}{\Z}\to \ho{p+q}{B}{\Z}, x\otimes y\mapsto x\ast y.
\]
\begin{lem}\label{lem:z4}  Let $A$ be any ring admitting a homomorphism to the ring $\Z/4$. 

Let $\rho:=
\left[
\begin{array}{cc}
1&0\\
1&1\\
\end{array}
\right]
\in B$. Then the element $\rho\ast -I\in \ker{\ho{2}{B}{\Z}\to\ho{2}{T}{\Z}}$ has order $2$.
\end{lem}

\begin{proof} When $A=\Z/4$, $B$ is the abelian group $C\times\mu_2$ where $C$ is cyclic of order $4$ with generator $\rho$ (and $T=\mu_2$). Thus 
$\ho{2}{B}{\Z}\cong B\wedge B$ is cyclic of order $2$ with generator $\rho\wedge -I=\rho\ast -I$.
\end{proof}

\begin{lem}\label{lem:1/p} Let $p$ be an odd prime and let $A=\Z[\frac{1}{p}]$. Then \[
\ho{2}{B}{\Z}=\ho{2}{T}{\Z}\times H
\]
where $H$ is cyclic of order $2$ with generator 
$\rho\ast -I$.
\end{lem}
\begin{proof} From the Hochschild-Serre spectral sequence associated to the (split) extension
$1\to U\cong A\to B\to T\cong A^\times\to 1$ (where $u\in A^\times$ acts on $A$ as multiplication by $u^2$), we derive a (split) short exact sequence 
$0\to \ho{1}{A^\times}{A}\to\ho{2}{B}{\Z}\to\ho{2}{T}{\Z}\to 0$.

 Now $A^\times=\mu_2\times p^{\Z}$. The sequence of terms of low degree for the extension $1\to\mu_2\to A^\times\to p^{\Z}\to 1$, with module $A$, has the form
\[
\cdots\to \ho{2}{p^\Z}{A}\to\ho{1}{\mu_2}{A}_{p^{\Z}}\to\ho{1}{A^\times}{A}\to\ho{1}{p^\Z}{A}\to 0.
\]
Now $p^{\Z}$ is infintite cyclic and $p$ acts on $A$ as multiplication by $p^2$. Since $A$ is torsion-free, $\ho{1}{p^{\Z}}{A}\cong \{ a\in A\ |\ (p^2-1)a=0\}=0$. 
Also $\ho{2}{p^{\Z}}{A}=0$. Now $p$ acts trivially on $\ho{1}{\mu_2}{A}=A/2A\cong\Z/2$. Thus $\ho{1}{A^\times}{A}\cong \ho{1}{\mu_2}{A}_{p^{\Z}}$ is cyclic of order $2$.
The statement of the Lemma now follows from Lemma \ref{lem:z4}.
\end{proof}
\section{Review: The Bloch group and indecomposable $K_3$ of fields}  \label{sec:kind}
For any field $F$ there is a natural surjective homomorphism $\ho{3}{\spl{2}{F}}{\Z}\to \kind{F}$. (For infinite fields, see \cite[Lemma 5.1]{hut:h3sl}. For the case of finite fields, see
 \cite[Corollary 3.9]{hut:cplx13}.)  Furthermore, there is a natural map $\kind{F}\to \bl{F}$ fitting into a commutative diagram
\[
\xymatrix{
\ho{3}{\spl{2}{F}}{\Z}\ar[r]\ar[d]&\rbl{F}\ar[d]\\
\kind{F}\ar[r]&\bl{F}.\\
}
\]

We note that the inclusion map $\mu_F\to\spl{2}{F}, \zeta\mapsto \mathrm{diag}(\zeta,\zeta^{-1})$ induces a homomorphism 
\[
\mathrm{tor}(\mu_F,\mu_F)=\ho{3}{\mu_F}{\Z}\to\ho{3}{\spl{2}{F}}{\Z}\to \kind{F}.
\]

For a finite cyclic group $C$ of even order, there is a unique nontrivial extension $\tilde{C}$ of $C$ by $\Z/2$. If $C$ is cyclic of odd order, let $\tilde{C}=C$. Since $\mu_F$ , and hence $\mathrm{tor}(\mu_F,\mu_F)$,  is a union of finite cyclic groups we can define $\widetilde{\mathrm{tor}(\mu_F,\mu_F)}$ to be the union of $\widetilde{\mathrm{tor}(\mu_n(F),\mu_n(F)}$.

\begin{thm}\label{thm:kind} For any field $F$, there is a natural short exact sequence
\[
0\to \widetilde{\mathrm{tor}(\mu_F,\mu_F)}\to \kind{F}\to \bl{F}\to 0.
\]
\end{thm}
\begin{proof}
For infinite fields, this is \cite[Theorem 5.2 ]{sus:bloch}. For case of finite fields with at least $4$ elements, see \cite[Corollary 7.5]{hut:cplx13}. The result is extended to the fields $\F{2}$ and $\F{3}$ in \cite{hut:rbl11}. (The ad hoc definitions of $\bl{\F{2}}$ and $\bl{\F{3}}$ which are given in that paper are justified below.)
\end{proof}

We will use the following observation below:

\begin{lem}\label{lem:omega} Let $F$ be a field with $\mathrm{char}(F)\not=2$ and $\mu_F=\mu_2$. Let $\tilde{\omega}$ be the element  $
\left[
\begin{array}{cc}
0&-1\\
1&0\\
\end{array}
\right]$ of order $4$  in $\spl{2}{F}$ . 

Then the map $\ho{3}{\an{\tilde{\omega}}}{\Z}\to \ho{3}{\spl{2}{F}}{\Z}\to\kind{F}$ induces an isomorphism 
\[
\ho{3}{\an{\tilde{\omega}}}{\Z}\cong \widetilde{\mathrm{tor}(\mu_F,\mu_F)}\subset \kind{F}.
\]
\end{lem}

\begin{proof}  Let $E=F(\sqrt{-1})$. $\tilde{\omega}$ has eigenvalues $\pm\sqrt{-1}$ in $E$. It is conjugate in $\gl{2}{E}$ to $J:=\mathrm{diag}(\sqrt{-1},-\sqrt{-1})$.
We consider $\spl{n}{F}$ as a subgroup of $\spl{n+1}{F}$ via the homomorphism which sends the matrix $M\in\spl{n}{E}$ to $\left[
\begin{array}{cc}
M&0\\
0&1\\
\end{array}
\right]\in\spl{n+1}{E}$.  Then $\tilde{\omega}$ is conjugate to $J$ in $\spl{3}{E}$. (If $C\in\gl{2}{E}$ conjugates $\tilde{\omega}$ to $J$ in $\gl{2}{E}$, then $\left[
\begin{array}{cc}
C&0\\
0&\mathrm{det}(C)^{-1}\\
\end{array}
\right]$ conjugates $\tilde{\omega}$ to $J$ in $\spl{3}{E}$.) Thus the image of $\ho{3}{\an{\tilde{\omega}}}{\Z}$ is equal to the image of $\ho{3}{\an{J}}{\Z}$ in $\ho{3}{\spl{2}{E}}{\Z}$ .

For any field $L$, the map $\ho{3}{\spl{2}{L}}{\Z}\to\kind{L}$ factors through a map\\
 $\ho{3}{\spl{n}{L}}{\Z}\to\kind{L}$ for all $n\geq 2$. It thus follows that 
the image of $\ho{3}{\an{\tilde{\omega}}}{\Z}$ is equal to the image of $\ho{3}{\an{J}}{\Z}$ in $\kind{E}$. This image is $\mathrm{tor}(\mu_4,\mu_4)$, which is the unique subgroup of order $4$ 
in $\widetilde{\mathrm{tor}(\mu_E,\mu_E)}$.

On the other hand, the commutativity of the diagram
\[
\xymatrix{
\widetilde{\mathrm{tor}(\mu_F,\mu_F)}\ar[r]\ar[d]&\kind{F}\ar[d]\\
\widetilde{\mathrm{tor}(\mu_E,\mu_E)}\ar[r]&\kind{E}\\
}
\]
shows that the image of $\widetilde{\mathrm{tor}(\mu_F,\mu_F)}$ in $\kind{E}$ is also the unique subgroup of $\widetilde{\mathrm{tor}(\mu_E,\mu_E)}$  order $4$. Since the map $\kind{F}\to \kind{E}$ is injective  (see \cite[Corollary 4.6]{levine:k3}), the result follows.
\end{proof}

\section{Special elements and algebraic identities in pre-Bloch groups}\label{sec:prebloch}

\subsection{The constant $\cconst{A}$}

Let $\{ x,y,z\}$ be a $3$-clique in $\Gamma(A)$. Then 
\[
d_2\left( (x,y,z)+(y,x,z)\right)=(x,y)+(y,x) = d_2\left( (z,x,y)+(z,y,x)\right) \mbox{ in } L_1(A).
\]
It follows that 
\[
C(x,y,z):=(z,x,y)+(z,y,x) -[(x,y,z)+(y,x,z)]\in \ker{d_2}=\LL_3(A).
\]
We also let $C(x,y,z)$ denote the resulting class in $\rpb{A}$ (and also in $\pb{A}$). 

In particular, we define the element 
\[
\cconst{A}:=C(\infty,0,1)=(1,\infty,0)+(1,0,\infty)-(\infty,0,1)-(0,\infty,1)\in \LL_3(A)
\] 
(and we will also denote the image of this element in $\rpb{A}$ and $\pb{A}$ by $\cconst{A}$).

Observe that, for any $3$-clique $\{ x,y,z\}$, we have $C(x,y,z)=C(\infty,0,1)\cdot M_{x,y,z}=\cconst{A}\cdot M_{x,y,z}$, and hence $C(x,y,z)=\an{d(x,y,z)}\cconst{A}$ in 
$\rpb{A}$ and $C(x,y,z)=\cconst{A}$ in $\pb{A}$.

\begin{lem}\label{lem:cconstM}
Let $A$ be a ring. Let $M=M_{0,1,\infty}=
\left[
\begin{array}{cc}
0&-1\\
1&1\\
\end{array}
\right] \in \pspl{2}{A}$ (and note that $M^2=M_{1,\infty,0}$). We have 
\[
\cconst{A}\cdot(I+M+M^2)=0 \mbox{ in } \LL_3(A)
\]
and hence, for any $3$-clique $ \{ x,y,z\}$,  $3\cdot C(x,y,z)=0$ in $\rpb{A}$  (and in $\pb{A}$). 
\end{lem}
\begin{proof} 
We have 
\begin{eqnarray*}
\cconst{A}&=&(1,\infty,0)+(1,0,\infty)-(\infty,0,1)-(0,\infty,1).\\
\mbox{ Hence }\cconst{A}\cdot M&=& (\infty,0,1)+(\infty,1,0)-(0,1,\infty)-(1,0,\infty)\\
\mbox{ and } \cconst{A}\cdot M^2&=&(0,1,\infty)+(0,\infty,1)-(1,\infty,0)-(\infty,1,0).\\
\end{eqnarray*}
The sum of the three expressions on the right is zero.
\end{proof}

Let $\left[
\begin{array}{cc}
0&1\\
1&0\\
\end{array}
\right]:=\omega\in\pgl{2}{A}$.

\begin{lem}\label{lem:cconst-1} Let $A$ be a ring. Then $\cconst{A}=\cconst{A}\cdot\omega$ in $\LL_3(A)$ and hence $\an{-1}\cconst{A}=\cconst{A}$ in $\rpb{A}$.
\end{lem}
\begin{proof}
Since $(\infty,0, 1)\cdot\omega=(0,\infty,1)$, from the definition we have 
\[
\cconst{A}=C(\infty,0,1)=C(0,\infty,1)=C(\infty,0,1)\cdot\omega=\cconst{A}\cdot\omega.
\]
 Note that  $\det{\omega}=-1$.
\end{proof}

\subsection{ The constant $\cconst{A}$ as an element of $\ho{3}{\spl{2}{A}}{\Z}$}
\begin{prop}\label{prop:da}
Let $M:= M_{0,1,\infty}=\left[
\begin{array}{cc}
0&1\\
-1&-1
\end{array}
\right]\in \spl{2}{\Z}$. Let $C$ be the cyclic group (of order $3$) generated by $M$.

 The inclusion map $C\to \spl{2}{\Z}$ induces a composite homomorphism, $\phi$ say, \\
$\ho{3}{C}{\Z}\to\ho{3}{\spl{2}{\Z}}{\Z}\to \rpb{\Z}$. 

Consider $1\in \Z/3\cong \ho{3}{C}{\Z}$. Then $\phi(1)=\cconst{\Z}$. 
\end{prop}
\begin{proof} We can calculate $\phi$ as follows (see, for example, \cite[p 56]{hut:cplx13}): Let $F_\bullet$ be the standard periodic resolution of $\Z$ as a $\Z[G]$-module. Let
$\beta:F_\bullet\to \LL_\bullet$ be an augmentation-preserving map of $\Z[C]$-complexes. Then $\beta$ induces a map
\[
\ho{3}{C}{\Z}=H_3\left( (F_\bullet)_G\right)\to H_3( (\LL_\bullet)_G)\to H_3( (\LL_\bullet)_{\spl{2}{\Z}})=\rpbker{\Z}\subset \rpb{\Z}
\]
which coincides with the map $\phi$. 

Now it is straightforward to verify that $\beta$ can be given as follows:

 $\beta_0:F_0=\Z[C]\to \LL_0(\Z), 1\mapsto (1)$, 

$\beta_1:F_1=\Z[C]\to \LL_1(\Z), 1\mapsto (1,\infty)$,

$\beta_2:F_2=\Z[C]\to \LL_2(\Z), 1\mapsto (0,1,\infty)+(1,\infty,0)+(1,0,\infty)$,

and $\beta_3:F_3=\Z[C]\to \LL_3(\Z), 1\mapsto (\infty,0,1)+(\infty,1,0)-(0,1,\infty)-(1,0,\infty)$.

This last term is $C(0,1,\infty)$ and so is equal to  $d(0,1,\infty)C(\infty,0,1)=\cconst{\Z}$ in $\rpb{\Z}$. Since $1\in \ho{3}{C}{\Z}$ is represented by $1\in \Z[C]=F_3$, the result follows. 
\end{proof}

\begin{cor} Let $A$ be a commutative ring. Let $C,M$ be as in Proposition \ref{prop:da}. Let $\phi_A$ be the composite homomorphism
$
\ho{3}{C}{\Z}\to \rpb{\Z}\to \rpb{A}.
$
Then $\phi_A(1)=\cconst{A}$. 
\end{cor}

\begin{rem}\label{rem:cconst}  Thus, for any  ring $A$ it is natural to also let $\cconst{A}\in \ho{3}{\spl{2}{A}}{\Z}$ denote the  image of $1\in \ho{3}{C}{\Z}$ in $\ho{3}{\spl{2}{A}}{\Z}$.
\end{rem}
\begin{rem} Note that the non-trivial automorphism of $C$, $M\mapsto M^{-1}=M^2$,  induces the identity on $H_3(C,\Z)$. Thus the choice of generator $1\in \Z/3\cong H_3(C,\Z)$ does not depend on the choice of generator -- $M$ or $M^2$ -- of $C$.
\end{rem}

\begin{cor}\label{cor:cconstA} For every ring $A$ we have $\cconst{A}\in\rbl{A}$.
\end{cor}

\begin{proof} Since $\cconst{A}$ lies in the image of the edge homomorphism $\ho{3}{\spl{2}{A}}{\Z}\to \rpb{A}=E^2_{0,3}(\spl{2}{A},\LL)$, it lies in $E^\infty_{0,3}(\spl{2}{A},\LL)=\rbl{A}$. In particular, $d^3(\cconst{A})=0$ in $E^3_{2,0}$ for any ring $A$.
\end{proof}

\subsection{The elements $\suss{1}{u}$}  

For any $u\in A^\times$, we  define the element
\[
\suss{1}{u}:=  (\infty,0,u)+(0,\infty,u)-(0,\infty,1)-(\infty,0,1) \mbox{ in }\rpb{A}.
\]

\begin{lem}\label{lem:suss1} Let $A$ be a ring.
\begin{enumerate}
\item For all $u\in A^\times$, $\suss{1}{u}\cdot\omega=\suss{1}{u^{-1}}$.
\item For all $u,v\in A^\times$, $\suss{1}{uv}=\suss{1}{u}\cdot M_v+\suss{1}{v}$.
\item For all $u\in A^\times$, $\an{-1}\suss{1}{u}=\suss{1}{u^{-1}}$.
\item For all $u,v\in A^\times$ we have
\[
\suss{1}{uv}=\an{v}\suss{1}{u}+\suss{1}{v}.
\]
In other words, the map $\psi_1:A^\times \to \rpb{A}$ is a $1$-cocycle for the natural $A^\times$-action on $\rpb{A}$. 
\end{enumerate}
\end{lem}
\begin{proof}
\begin{enumerate}
\item Since $(\infty,0,1,u)\cdot\omega=(0,\infty,1,u^{-1})$ for any $u\in A^\times$, we see from the definition that $\suss{1}{u}\cdot\omega=\suss{1}{u}$.
\item For any $v\in A^\times$, we have $(\infty,0,1,u)\cdot M_v=(\infty,0,v,uv)$. It follows that 
\[
\suss{1}{u}\cdot M_v=(\infty,0,uv)+(0,\infty,uv)-(\infty,0,v)-(0,\infty,v)=\suss{1}{uv}-\suss{1}{v}.
\]
\item 
This follows from (1) since $\det{\omega}=-1$.
\item This follows from (2) since $\det{(M_v)}=v$.
\end{enumerate}
\end{proof}

\begin{cor}\label{cor:suss1} Let $A$ be a  ring.
\begin{enumerate}
\item For all $x\in A^\times$, $\suss{1}{x}+\suss{1}{-1}=\suss{1}{-x^{-1}}$.
\item $2\suss{1}{-1}=0$.
\end{enumerate}
\end{cor}
\begin{proof}
\begin{enumerate}
\item Let $x\in A^\times$. Then $\suss{1}{x}+\suss{1}{-1}=\an{-1}\suss{1}{x^{-1}}+\suss{1}{-1}=\suss{1}{-1\cdot x^{-1}}$ by Lemma \ref{lem:suss1}.
\item Take $x=-1$ in (1), and note that $\suss{1}{1}=0$. 
\end{enumerate}
\end{proof}

\begin{prop} \label{prop:suss1} Let $A$ be a ring and let $u\in\wn{A}$. Then 
\[
\suss{1}{u}=\gpb{u}+\gpb{u^{-1}}\cdot \omega \mbox{ in } \LL_3(A) 
\]
and hence $\suss{1}{u}=\gpb{u}+\an{-1}\gpb{u^{-1}}$ in $\rpb{A}$.
\end{prop}

\begin{proof} Let $X=(\infty,0,1,u)+(0,\infty,1,u)\in L_3(A)$. Then $X=(\infty,0,1,u)+(\infty,0,1,u^{-1})\cdot \omega =\agpb{u}+\agpb{u^{-1}}\cdot\omega$. It is straightforward to verify that $d_3(X)=\suss{1}{u}$.
\end{proof}

We denote  the image of $\suss{1}{x}$ in $\pb{A}$ by $\sus{x}$.  

Thus for any $x\in \wn{A}$, we have $\sus{x}=\gpb{x}+\gpb{x^{-1}}$ in $\pb{A}$.  In particular, $\sus{-1}=2\gpb{-1}$ in $\pb{A}$ whenever $-1\in \wn{A}$; i.e., whenever $2\in A^\times$. 
\begin{cor} Let $A$ be a  ring. We have the following identities in $\pb{A}$. 
\begin{enumerate}
\item For all $x,y\in A^\times$, $\sus{xy}=\sus{x}+\sus{y}$.
\item  For all $x\in A^\times$, $2\sus{x}=\sus{x^2}=0$.
\item $4\gpb{-1}=0$  if $2\in A^\times$.
\end{enumerate}
\end{cor}
\begin{proof}
\begin{enumerate}
\item This follows immediately from the cocycle property of $\suss{1}{x}$.
\item For all $x\in A^\times$, $\sus{x^{-1}}=\sus{x}$  by Lemma \ref{lem:suss1} (1), and hence $0=\sus{x\cdot x^{-1}}=\sus{x}+\sus{x^{-1}}=2\sus{x}$.
\item Take $x=-1$ in (2). 
\end{enumerate}
\end{proof}
\begin{rem} 
On the other hand, for any ring $A$ the elements $\suss{1}{u}\in\rpb{A}$ are usually of infinite order: see Proposition \ref{prop:ks} below.
\end{rem}

\begin{prop}\label{prop:cconst}
 Let $A$ be a  ring and let $u\in \wn{A}$. Then
\[
\cconst{A}=\suss{1}{1-u^{-1}}-\gpb{u}\cdot
M_{1,\infty,0}-\gpb{1-u}\cdot M_{1,0,\infty}-\suss{1}{u}\mbox{ in } \LL_3(A),
\]
and hence 
\[
\cconst{A}=\suss{1}{1-u^{-1}}-(\gpb{u}+\an{-1}\gpb{1-u})\mbox{ in } \rpb{A}, 
\]
and
\[
\cconst{A}=\sus{1-u^{-1}}-(\gpb{u}+\gpb{1-u})\mbox{ in } \pb{A}. 
\]
\end{prop}
\begin{proof} Let $u\in\wn{A}$ and let
\[
Y:=(\infty,0,1,u)+(0,\infty,1,u)-(1,\infty,0,u)-(1,0,\infty,u)\in L_3(A).
\]
Then $d_3(Y)=\cconst{A}\in \LL_3(A)$.

Note that 
\[
M_{1,\infty,0}=
\left[
\begin{array}{cc}
1&1\\
-1&0\\
\end{array}
\right],\quad
M_{1,0,\infty}=
\left[
\begin{array}{cc}
1&1\\
0&-1\\
\end{array}
\right].
\]
Thus 
\begin{eqnarray*}
Y&=&(\infty,0,1,u)+(\infty,0,1,u^{-1})\cdot\omega-\left(\infty,0,1,\frac{1}{1-u}\right)\cdot M_{1,\infty,0}-\left( \infty,0,1,\frac{1}{1-u^{-1}}\right)\cdot M_{1,0,\infty}\\
&=&\agpb{u}+\agpb{u^{-1}}\cdot \omega-\agpb{\frac{1}{1-u}}\cdot M_{1,\infty,0}-\agpb{\frac{1}{1-u^{-1}}}\cdot M_{1,0,\infty}.\\
\end{eqnarray*}
Hence 
\begin{eqnarray*}
\cconst{A}=d_3(Y)&=&\gpb{u}+\gpb{u^{-1}}\cdot \omega-\gpb{\frac{1}{1-u}}\cdot M_{1,\infty,0}-\gpb{\frac{1}{1-u^{-1}}}\cdot M_{1,0,\infty}\\
&=&\suss{1}{u}-\gpb{\frac{1}{1-u}}\cdot M_{1,\infty,0}-\gpb{\frac{1}{1-u^{-1}}}\cdot M_{1,0,\infty}\in \LL_3(A).\\
\end{eqnarray*}
Since this holds for all $u\in \wn{A}$, we can replace $u$ by $1-u^{-1}$ to obtain the stated expression.
\end{proof}


\subsection{The elements $\suss{2}{u}$ and the key identity}


For any ring $A$, and for any unit $u\in A^\times$ we define 
\[
\suss{2}{u}:=(u,\infty,0)+(u,0,\infty)-(1,\infty,0)-(1,0,\infty)\in\LL_3(A).
\]
We also let $\suss{1}{u}$ denote the image of these elements in $\rpb{A}$.

Analogously to the elements $\suss{1}{u}$, we have:

\begin{lem}\label{lem:suss2} Let $A$ be a ring.
\begin{enumerate}
\item For all $u\in A^\times$, $\suss{2}{u}\cdot\omega=\suss{2}{u^{-1}}$ in $\LL_3(A)$.
\item For all $u,v\in A^\times$ we have
\[
\suss{2}{uv}=\suss{2}{u}\cdot M_v+\suss{2}{v}.
\]
\item For all $u\in A^\times$, $\an{-1}\suss{2}{u}=\suss{2}{u^{-1}}$ in $\rpb{A}$
\item For all $u,v\in A^\times$, $\suss{2}{uv}=\an{v}\suss{2}{u}+\suss{2}{v}$.
In other words, the map $\psi_2:A^\times \to \rpb{A}$ is a $1$-cocycle for the natural $A^\times$-action on $\rpb{A}$. 
\end{enumerate}
\end{lem}

\begin{proof} The proof is almost identical to that for $\suss{1}{u}$. 
\end{proof}

\begin{lem} Let $A$ be a  ring. For all $u\in \wn{A}$, 
\[
\suss{2}{u}=\gpb{u}\cdot M_{1,u,\infty}+\gpb{u^{-1}}\cdot M_{1,u,0} \mbox{ in }\LL_3(A),
\]
and hence 
\[
\suss{2}{u}= \an{u-1}\gpb{u}+\an{1-u^{-1}}\gpb{u^{-1}}\in \rpb{A}.
\]
\end{lem}
\begin{proof} Let $u\in \wn{A}$.
Let $X=(1,u,\infty,0)+(1,u,0,\infty)\in L_3(A)$ and observe that $d_3(X)=\suss{2}{u}$.
Now
\[
M_{1,u,\infty}=
\left[
\begin{array}{cc}
1&1\\
-u&-1\\
\end{array}
\right],\quad
M_{1,u,0}=
\left[
\begin{array}{cc}
1&1\\
-1&-u^{-1}\\
\end{array}
\right].
\]
It follows that $(1,u,\infty,0)=(\infty,0,1,u)\cdot M_{1,u,\infty}=\agpb{u}\cdot M_{1,u,\infty}$ and $(1,u,0,\infty)=(\infty,0,1,u^{-1})\cdot M_{1,u,0}=\agpb{u^{-1}}\cdot M_{1,u,0}$. 

Finally, note that $\det{(M_{1,u,\infty})}=d(1,u,\infty)=u-1$ and $\det{(M_{1,u,0})}=d(1,u,0)=1-u^{-1}$.
\end{proof}

Thus $\psi_1$ and $\psi_2$ both define $1$-cocycles in $\rpb{A}$. Both are lifts of  the homomorphism $\psi:A^\times \to \pb{A}$.  In general, however, they do not coincide. In fact, they differ by a $1$-coboundary, as the following \emph{key identity} shows: 

\begin{thm}\label{thm:key} Let $A$ be a ring. For all $u\in A^\times$ we have 
\[
\cconst{A}\cdot \left( M_u-I\right)=\suss{2}{u}-\suss{1}{u} \mbox{ in }\LL_3(A),
\]
and hence
\[
\pf{u}\cconst{A}=\suss{2}{u}-\suss{1}{u} \mbox{ in } \rpb{A}.
\]
\end{thm}
\begin{proof} Let $u\in A^\times$.  Then 
\begin{eqnarray*}
\suss{2}{u}-\suss{1}{u}&=&(u,\infty,0)+(u,0,\infty)-(1,\infty,0)-(1,0,\infty)\\
&&- (\infty,0,u)-(0,\infty,u)+(\infty,0,1)+(0,\infty,1)\\
&=& (u,\infty,0)+(u,0,\infty)-(\infty,0,u)-(0,\infty,u)\\
&&-(1,\infty,0)-(1,0,\infty)+(\infty,0,1)+(0,\infty,1)\\
&=& \cconst{A}\cdot M_u-\cconst{A}.\\
\end{eqnarray*}
\end{proof}

\begin{rem} This identity was previously proven, by an obscure circuitous method, for fields with at least $4$ elements and for local rings whose residue field has at least $5$ elements (see \cite{hut:rbl11} and \cite{hut:slr}).
The more natural definitions of  $\cconst{A}$ and $\suss{i}{u}$ in the present article instead make it  easy and immediate. This identity, nevertheless, is the starting point for the calculation of the third homology of  $\mathrm{SL}_2$ of  local fields and local rings and of $\Q$ in the papers \cite{hut:rbl11},\cite{hut:slr} and \cite{hut:sl2Q}.
\end{rem}

\begin{cor}\label{cor:sus} For any ring $A$ and any $u\in A^\times$, the image of $\suss{2}{u}$ in $\pb{A}$ is equal to the image of $\suss{1}{u}$ (and hence is equal to $\sus{u}$).
\end{cor}
\begin{cor}\label{cor:-1}
  Let $A$ be a ring. Then $\suss{1}{-1}=\suss{2}{-1}$ and   in $\rpb{A}$.
\end{cor}
\begin{proof} In the identity $\pf{-1}\cconst{A}=\suss{2}{-1}-\suss{1}{-1}$  the left side is equal to zero by Lemma \ref{lem:cconst-1}.
\end{proof}

\subsection{The constant $\bconst{A}$}

For a commutative ring $A$ we define
\[
\bconst{A}:= (1,0,\infty)+(1,\infty,0)-(0,\infty,-1)-(\infty,0,-1)\in \LL_3(A),
\]
and we will also denote the image of this element in $\rpb{A}$ or $\pb{A}$ by $\bconst{A}$.

From the expressions for $\cconst{A}$ and $\suss{1}{x}$ above, we immediately deduce 
\[
\bconst{A}=\cconst{A}-\suss{1}{-1}\in\LL_3(A)
\]
for any commutative ring $A$.

From the definitions and results  above, we  have the identities
\[
6\bconst{A}=0,\quad 2\bconst{A}=\cconst{A},\quad 3\bconst{A}=\suss{1}{-1}, \quad \bconst{A}=\cconst{A}+\suss{1}{-1}
\]
in $\rpb{A}$. 

\begin{prop}\label{prop:bconst}
Let $A$ be a commutative ring and let $u\in\wn{A}$. Then
\[
\bconst{A}=\gpb{\frac{1}{u}}+\an{-1}\gpb{\frac{1}{1-u}}+\pf{1-u}\suss{1}{u} \mbox{ in } \rpb{A}.
\]
\end{prop}
\begin{proof} 
By Proposition \ref{prop:cconst}, we have 
\[
\cconst{A}=-\gpb{u}-\an{-1}\gpb{1-u}+\suss{1}{1-u^{-1}} \mbox{ in }\rpb{A}.
\]
Thus 
\begin{eqnarray*}
\bconst{A}&=& \cconst{A}-\suss{1}{-1}\\
&=&-\gpb{u}-\an{-1}\gpb{1-u}+\suss{1}{1-u^{-1}}-\suss{1}{-1}\\
&=& \an{-1}\gpb{\frac{1}{u}}+\gpb{\frac{1}{1-u}}+\suss{1}{1-u^{-1}}-\suss{1}{-1}-\suss{1}{u}-\suss{1}{\frac{1}{1-u}}\\
\end{eqnarray*}
(using  Proposition \ref{prop:suss1} for $\suss{1}{u}$ and $\suss{1}{1/(1-u)}$).

Now 
\[
\suss{1}{1-u^{-1}}=\suss{1}{(u-1)\cdot\frac{1}{u}}=\an{u-1}\suss{1}{\frac{1}{u}}+\suss{1}{u-1}=\an{1-u}\suss{1}{u}+\suss{1}{u-1}
\]
while $\suss{1}{-1}+\suss{1}{1/(1-u)}=\suss{1}{-1}+\an{-1}\suss{1}{1-u}=\suss{1}{u-1}$, using the cocycle identity in both cases.

Thus
\[
\bconst{A}= \an{-1}\gpb{\frac{1}{u}}+\gpb{\frac{1}{1-u}}+\pf{1-u}\suss{1}{u}.
\]

Note that  $\an{-1}\bconst{A}=-\an{-1}\cconst{A}-\an{-1}\suss{1}{-1}=-\cconst{A}-\suss{1}{-1}=\bconst{A}$ and since, for any $a,b\in A^\times$, we have 
$\suss{1}{ab}=\an{a}\suss{1}{b}+\suss{1}{a}=\an{b}\suss{1}{a}+\suss{1}{b}$ and hence $\pf{a}\suss{1}{b}=\pf{b}\suss{1}{a}$,  it follows that
\[
\an{-1}\pf{a}\suss{1}{b}=\pf{a}\suss{1}{b^{-1}}=\pf{b^{-1}}\suss{1}{a}=\pf{b}\suss{1}{a}=\pf{a}\suss{1}{b}.
\]
\end{proof}

\begin{cor}\label{cor:bconst} For any $u\in \wn{A}$ 
\[
\bconst{A}=\gpb{\frac{1}{u}}+\gpb{\frac{1}{1-u}} \mbox{ in }\pb{A}.
\]
\end{cor}
\begin{cor}
 The element $\bconst{\Z}\in \rpb{\Z}$ has order $6$.
\end{cor}
\begin{proof} Certainly $6\bconst{\Z}=0$. On the other hand, under the functorial maps $\rpb{\Z}\to\pb{\Z}\to\pb{\R}$, $\bconst{\Z}$ maps to $\bconst{\R}$.
By Corollary \ref{cor:bconst}, $\bconst{\R}=\gpb{\frac{1}{u}}+\gpb{\frac{1}{1-u}}$ for any $u\in \R\setminus\{ 0,1\}=\wn{\R}$. Alternatively,
$\bconst{\R}=\gpbold{u}+\gpbold{1-u}$ when written in terms of the elements $\gpbold{u}$ (see Remark \ref{rem:gpbold} above.)  By \cite[Section 1]{sus:bloch}, $\bconst{\R}$ has exact order $6$ and hence so does $\bconst{\Z}$. 
\end{proof}

\begin{rem} In \cite{hut:blochf2} we will give a more direct algebraic proof of this last corollary.
\end{rem}

\subsection{The map $\lambda_1$ again}


\begin{lem}\label{lem:Lambda1}
Let $A$ be a ring. 
\begin{enumerate}
\item For all $u\in A^\times$, 
\[
\lambda_1(\suss{1}{u})=\lambda_1(\suss{2}{u})=-\pf{-u}\pf{u}=(1+\an{-1})\pf{u}\in \aug{A}.
\]
\item $\lambda_1(\cconst{A})=0$.
\end{enumerate}
\end{lem}
\begin{proof}\ 
\begin{enumerate}
\item Let $u\in A^\times$.

Recall that $\suss{1}{u}=(\infty,0,u)+(0,\infty,u)-(\infty,0,1)-(0,\infty,1)$. Now
\begin{eqnarray*}
D(\infty,0,u)=\an{d(\infty,0)d(0,u)d(u,\infty)}=\an{1\cdot -u\cdot -1}=\an{u}.\\
D(0,\infty,u)=\an{d(0,\infty)d(\infty,u)d(u,0)}=\an{-1\cdot 1\cdot u}=\an{-u}\\
\end{eqnarray*}
Taking $u=1$, $D(\infty,0,1)=\an{1}=1$ and $D(0,\infty,1)=\an{-1}$. So 
\[
\lambda_1(\suss{1}{u})=\an{u}+\an{-u}-1-\an{-1}=(1+\an{-1})\pf{u}=-\pf{u}\pf{-u}.
\]

Similarly, 
\[
\suss{2}{u}=D(u,\infty,0)+D(u,0,\infty)-D(1,\infty,0)-D(1,0,\infty)=(1+\an{-1})\pf{u}.
\]
\item  By definition,
\begin{eqnarray*}
\lambda_1(\cconst{A})&=&\an{D(1,\infty,0)}+\an{D(1,0,\infty)}-\an{D(\infty,0,1)}-\an{D(0,\infty,1)}\\
&=& \an{1}+\an{-1}-\an{1}-\an{-1}=0.
\end{eqnarray*}
\end{enumerate}
\end{proof}

\subsection{The modules $\ks{i}{A}$}
Let $A$ be any ring. Recall that the $1$-cocycles $\psi_i:A^\times \to \rpb{A}$, $i=1,2$ satisfy
\[
\suss{i}{1}=0\mbox{ and }\an{-1}\suss{i}{a}=\suss{i}{a^{-1}}\mbox{ for all }a\in A^\times.
\]
\begin{lem}\label{lem:suss} Let $M$ be an $\sgr{A}$-module and let $\phi:A^\times\to M$ be a $1$-cocycle satisfying $\phi(1)=0$ and $\an{-1}\phi(u)=\phi(u^{-1})$ for all $u\in A^\times$. 
Then
\begin{enumerate}
\item $\pf{a}\phi(b)=\pf{b}\phi(a)$ for all $a,b\in A^\times$.
\item $\phi(ab^2)=\phi(a)+\phi(b^2)$ for all $a,b\in A^\times$. 
\item $2\phi(-1)=0$. 
\item $\phi(a^2)=\pf{a}\phi(-1)=\pf{-1}\phi(a)$  for all $a\in A^\times$. 
\item $2\phi(a^2)=0$ for all $a\in A^\times$. If $-1\in (A^\times)^2$, then $\phi(a^2)=0$ for all $a\in A^\times$. 
\end{enumerate}
\end{lem}
\begin{proof}\ 
\begin{enumerate}
\item This follows from the cocycle condition since
 \[
\phi(ab)=\an{a}\phi(b)+\phi(a)=\an{b}\phi(a)+\phi(b).
\]
\item $\phi(b^2a)=\an{b^2}\phi(a)+\phi(b^2)=\phi(a)+\phi(b^2)$ since $\an{b^2}$ acts trivially on $M$ by assumption.
\item $0=\phi(1)=\phi(-1\cdot -1)=\an{-1}\phi(-1)+\phi(-1)=2\phi(-1)$.
\item For any $a\in A^\times$ we have 
\[
\phi(a)=\phi(a^2\cdot a^{-1})=\an{a^2}\phi(a^{-1})+\phi(a^2)=\an{-1}\phi(a)+\phi(a^2).
\]
Thus $\phi(a^2)=-\pf{-1}\phi(a)=-\pf{a}\phi(-1)=\pf{a}\phi(-1)$. 
\item $2\phi(a^2)=2\pf{-a}\phi(-1)=0$ by (3) and (4).  Furthermore, by the proof of (4), $\phi(a^2)=\pf{-1}\phi(a)$ and this vanishes if $-1$ is a square in $A$.
\end{enumerate}
\end{proof}

\begin{cor}\label{cor:ann-1} Let $A$ be a ring. Suppose that $u\in A$ satisfies $-u^2\in \wn{A}$. Then $\pf{1+u^2}\suss{i}{-1}=0$ in $\rpb{A}$. 
\end{cor}

\begin{proof} For $i=1,2$, we have $\suss{i}{-u^2}=\suss{i}{-1}+\suss{i}{u^2}=\suss{i}{-1}+\pf{u}\suss{i}{-1}=\an{u}\suss{i}{-1}$. Thus
$\suss{1}{-u^2}=\suss{2}{-u^2}=\an{u}\suss{1}{-1}$ by Corollary \ref{cor:-1}.

Now
\begin{eqnarray*}
\suss{2}{-u^2}&=& \an{-u^{-2}-1}\gpb{-u^2}+\an{1+u^2}\gpb{-u^{-2}}\\
&=& \an{-(1+u^2)}\gpb{-u^2}+\an{1+u^2}\gpb{-u^{-2}}\\
&=&\an{-(1+u^2)}\left(\gpb{-u^2}+\an{-1}\gpb{-u^{-2}}\right)\\
&=& \an{-(1+u^2)}\suss{1}{-u^2}=\an{-(1+u^2)}\suss{2}{-u^2}.\\
\end{eqnarray*}
Thus $\an{-(1+u^2)}=\an{-1}\an{1+u^2}$ acts trivially on $\suss{2}{-u^2}=\an{u}\suss{2}{-1}$, and hence also $\suss{2}{-1}$. Since $\an{-1}$ also acts 
trivially on $\suss{2}{-1}$, the result follows. 
\end{proof}

\begin{exa} Let $p$ be a prime number satisfying $p\equiv 1\mod{4}$. Then $\pf{p}\suss{1}{-1}=0$ in $\rpb{\Q}$: There exist $x,y\in \Z$ with $p=x^2+y^2$ and hence
$py^{-2}=(x/y)^2+1$ in $\Q^\times$. Thus $\an{p}=\an{py^{-2}}=\an{(x/y)^2+1}$. 
\end{exa}

For $i=1,2$ we let $\ks{i}{A}$ denote the $\sgr{A}$-submodule of $\rpb{A}$ generated by the set $\{ \suss{i}{u}\ |\ u\in A^\times\}$.

\begin{prop}\label{prop:ks} let $A$ be a ring. For $i=1,2$ there is a short exact sequence of $\sgr{A}$-modules
\[
\xymatrix{
0\ar[r]& \left( \ks{i}{A}\right)_{\tiny \mathrm{tors}}\ar[r]&\ks{i}{A}\ar^-{\lambda_1}[r]&(1+\an{-1})\aug{A}\ar[r]&0
}
\]
and $\left( \ks{1}{A}\right)_{\tiny \mathrm{tors}}=\left( \ks{2}{A}\right)_{\tiny \mathrm{tors}}=\sgr{A}\suss{1}{-1}$ (which is annihilated by $2$). 

\end{prop}

\begin{proof}
The image of $\lambda_1:\ks{i}{A}\to\sgr{A}$ is $(1+\an{-1})\aug{A}$ by Lemma \ref{lem:lambda1} (2).

Recall that $\suss{1}{-1}=\suss{2}{-1}$ by Corollary \ref{cor:-1}. Now $\lambda_1(\suss{1}{-1})=(1+\an{-1})\pf{-1}=\pf{-1}+\pf{1}-\pf{-1}=0$. 
So $\sgr{A}\suss{1}{-1}\subset \ker{\lambda_1:\ks{i}{A}\to (1+\an{-1})\aug{A}}$.

For any $u\in A^\times$, we have $(1+\an{-1})\pf{u}=\pf{u}+\pf{-u}-\pf{-1}$.  It follows at once that the abelian group $(1+\an{-1})\aug{A}$ is free with basis 
$\{ (1+\an{-1})\pf{u}\ | u\in S\}$ where $S\subset A^\times$ is any set of representatives of $ A^\times/\pm(A^\times)^2\setminus \{ 1\}$. 

Now, for any $u,v\in A^\times$, we have $\suss{i}{uv^2}=\suss{i}{u}+\suss{i}{v^2}=\suss{i}{u}+\pf{v}\suss{i}{-1}$ by Lemma \ref{lem:suss} (2),(4). 
Furthermore, for any $u\in A^\times$, 
$\suss{i}{-u}=\suss{i}{u\cdot -1}=\an{u}\suss{i}{-1}+\suss{i}{u}$. Recall also that $\an{u}\suss{i}{v}=\suss{i}{uv}-\suss{i}{u}$ for any $u,v\in A^\times$. 
 It follows that $\ks{i}{A}/(\sgr{A}\suss{i}{-1})$ is generated, as an abelian group,  by the elements 
$\{ \suss{i}{u}\ |\ u\in S\}$. Since these elements  map under $\lambda_1$ to a basis of $(1+\an{-1})\aug{A}$, the result follows.
\end{proof}

\begin{rem}
Since $(1+\an{-1})\aug{A}$ is a free abelian group, the short exact sequence of Proposition \ref{prop:ks} is split as a sequence of abelian groups.
\end{rem}

\begin{cor} Let $A$ be a ring. There is an isomorphism of $\sgr{A}$-modules\\
 $\ks{1}{A}\to \ks{2}{A}$ sending $\suss{1}{u}$ to $\suss{2}{u}$ for $a\in A^\times$.
\end{cor}
\begin{proof} Since $\ks{2}{A}$ contains no non-trivial $3$-torsion, $\ks{2}{A}\cap\aug{A}\cconst{A}=0$. By Theorem \ref{thm:key}, the inclusion map 
$\ks{1}{A}\to \ks{2}{A}+\aug{A}\cconst{A}$ induces an isomorphism
\[
\ks{1}{A}\to (\ks{2}{A}+\aug{A}\cconst{A})/(\aug{A}\cconst{A})\cong \ks{2}{A}
\]
sending $\suss{1}{u}$ to $\suss{2}{u}$. 
\end{proof}
\begin{rem} However, in general $\suss{1}{u}\not=\suss{2}{u}$ and $\ks{1}{A}\not=\ks{2}{A}$:

By Theorem \ref{thm:key} and Proposition \ref{prop:ks}, the image of $\ks{2}{A}$ in $\rpb{A}/\ks{1}{A}$ is isomorphic 
to the module $\aug{A}\cconst{A}$. Consider now  the case $A=\Q$. Then  $\aug{\Q}\cconst{\Q}$ is an infinite-dimensional $\F{3}$-vector space 
with basis  $\{ \pf{p}\cconst{\Q}\ |\  p\equiv -1\mod{3}\}$  (see \cite[Section 6.2]{hut:sl2Q}). 
\end{rem}

\begin{rem} In general, the module $\left( \ks{i}{A}\right)_{\tiny \mathrm{tors}}$ is nonzero. It always contains the element $\suss{1}{-1}$. Since $\bconst{\Z}$ had order $6$, it follows that 
$\suss{1}{-1}\not=0$ in $\rpb{\Z}$. 
\end{rem}

Now let $\qrpb{A}$ be the $\sgr{A}$-module $\rpb{A}/\ks{1}{A}$.  Let $\tilde{\lambda}_1$ be the induced homomorphism $\qrpb{A}\to \aug{A}/(1+\an{-1})\aug{A}$ 
and define $\qrpbker{A}:= \ker{\tilde{\lambda}_1}$. 

\begin{prop}\label{prop:qrpb}
There is a natural short exact sequence of $\sgr{A}$-modules 
\[
0\to \sgr{A}\suss{1}{-1}\to \rpbker{A}\to \qrpbker{A}\to 0.
\]
\end{prop}

\begin{proof} 

This follows from the following commutative diagram in which the lower two rows and all columns are exact:
\[
\xymatrix{
&0\ar[d]&0\ar[d]&0\ar[d]&\\
0\ar[r]&\sgr{A}\suss{1}{-1}\ar[d]\ar[r]&\rpbker{A}\ar[d]\ar[r]&\qrpbker{A}\ar[d]\ar[r]&0\\
0\ar[r]&\ks{1}{A}\ar[r]\ar^-{\lambda_1}[d]&\rpb{A}\ar[r]\ar^-{\lambda_1}[d]&\qrpb{A}\ar[r]\ar^-{\tilde{\lambda}_1}[d]&0\\
0\ar[r]&(1+\an{-1})\aug{A}\ar[d]\ar[r]&\aug{A}\ar[r]&\frac{\aug{A}}{(1+\an{-1})\aug{A}}\ar[r]&0\\
&0&&&\\
}
\]
\end{proof}

\section{Examples and applications}\label{sec:exa}
\subsection{The Bloch group of  $\F{2}$}

Since $G:=\pspl{2}{\F{2}}=\spl{2}{\F{2}}=\gl{2}{\F{2}}$  it follows that $\rpb{\F{2}}=\rpbker{\F{2}}=\pb{\F{2}}$ and $\rbl{\F{2}}=\bl{\F{2}}$. 
 $G$ is non-abelian of order $6$, with generators $\omega:=T_{0,\infty}=
\left[
\begin{array}{cc}
0&1\\
1&0\\
\end{array}
\right]$ and $\gamma:=M_{0,1,\infty}=\left[
\begin{array}{cc}
0&1\\
1&1\\
\end{array}
\right]$
or order $3$.

 The group $B=B_{\F{2}}$ has order $2$ generated by $\beta=\gamma\omega =
\left[
\begin{array}{cc}
1&0\\
1&1\\
\end{array}
\right]$ and $T=Z=\{ 1\}$.
By Proposition \ref{prop:xnsl2}, it follows that $L_1(\F{2})=(\infty,0)\cdot\Z[G]\cong \Z[G]$  (as right $\Z[G]$-modules)
and
 $L_2(\F{2})=(\infty,0,1)\cdot \Z[G]\cong \Z[G]$.
 The exact sequence $0\to \LL_3(\F{2})\to\cdots \to \LL_0(\F{2})\to \Z\to 0$ thus has the form
\[
0\to \LL_3(\F{2})\to \Z[G]\to \Z[G]\to \Z[B\backslash G]\to \Z\to 0.
\]
Counting $\Z$-ranks, it follows that $\LL_3(\F{2})$ is a free abelian group of rank $2$.


Now consider $D:=\cconst{\F{2}}\in L_2(A)$. Thus
$
D=C(\infty,0,1)=(1,\infty,0)+(1,0,\infty)-(\infty,0,1)-(0,\infty,1)$.

Recall that $d_2(D)=0$ and $D\in \LL_3(\F{2})$. Thus the element 
$
F:=D\cdot\gamma= D\cdot M_{0,1,\infty}= C(0,1,\infty)= (\infty,0,1)+(\infty,1,0)-(0,1,\infty)-(1,0,\infty)$
also lies in $\LL_3(\F{2})$. $D$ and $F$ are a basis for a direct summand of $L_2(\F{2})$ and thus they are a basis for the rank $2$ submodule $\LL_2(\F{2})$.

Recall that $D\cdot (1+\gamma+\gamma^2)=0$ and $D\cdot\omega=D$ by Lemmas \ref{lem:cconstM} and \ref{lem:cconst-1}. Thus the right $\Z[G]$-module structure of 
$\LL_3(\F{2})$  is determined by the identities:
\[
D\gamma =F,\ D\omega =D,\  F\gamma=-(D+F)=F\omega.
\]
If we let $\bar{D}$, $\bar{F}$ denote the images of these elements in $(\LL_3(\F{2})_G=\pb{\F{2}}=\rpb{\F{2}}=\rpbker{\F{2}}$ we deduce that 
\[
\bar{D}=\bar{F}=-\bar{D}-\bar{F}
\]
and hence $\rpb{\F{2}}$ is cyclic of order $3$ with generator $\cconst{\F{2}}=\bar{D}=\bar{F}$.

 Since $E^\infty_{2,0}$ is a quotient of $\ho{2}{B}{\Z}=0$, 
it follows that $\rpb{\F{2}}=\rpbker{\F{2}}=\rbl{\F{2}}=(\Z/3)\cdot\cconst{\F{2}}$. 

From the spectral sequence it now follows that we have a short exact sequence
\[
0\to \ho{3}{B}{\Z}\to \ho{3}{\spl{2}{\F{2}}}{\Z}\to \rbl{\F{2}}\to 0
\]
(and hence $\ho{3}{\spl{2}{\F{2}}}{\Z}$ is cyclic of order $6$ as expected).

\begin{cor}
There is a  natural  isomorphism $\kind{\F{2}}\to \bl{\F{2}}$. 
\end{cor}
\begin{proof} In \cite[Corollary 3.9]{hut:cplx13} it is shown that for a finite field $F$ of characterstic $p$, the natural homomorphism $\ho{3}{\spl{2}{F}}{\Z}\to \kind{F}$ induces an isomorphism (on tensoring by $\Z[1/p]$) $\ho{3}{\spl{2}{F}}{\Z[1/p]}\cong \kind{F}$. By the calculations above, however, there are natural isomorphisms
 \[
\kind{\F{2}}\cong \ho{3}{\spl2{\F{2}}}{\Z[1/2]}\cong\bl{\F{2}}
\]
\end{proof}

\begin{rem} Since $\suss{1}{-1}=\suss{1}{1}=0$ in $\rpb{\F{2}}$ we have $\cconst{\F{2}}=\bconst{\F{2}}$. 
\end{rem}
\subsection{The Bloch group of  $\F{3}$}
$\F{3}$ is $L_\bullet$ acyclic in dimensions $\leq 2$. So $\rpb{\F{3}}$ is generated as an $\sgr{\F{3}}$-module by $\gpb{-1}$ (since $\wn{\F{3}}=\{ -1\}$).  By Corollary \ref{cor:suss1} (2)
we have $0=2\suss{1}{-1}=2(\gpb{-1}+\an{-1}\gpb{-1})$ in $\rpb{\F{3}}$.

\begin{prop}\label{prop:rpbf3} $\rpb{\F{3}}$ is the cyclic $\sgr{\F{3}}$-module with generator $\gpb{-1}$ subject to the relation $2(\gpb{-1}+\an{-1}\gpb{-1})=0$. 
\end{prop}
\begin{proof}
Let $\mathcal{A}$ denote the cyclic $\sgr{\F{3}}$-module generated by the symbol $\gpb{-1}$, subject to the relation $2\suss{1}{-1}=0$.  Thus we have a surjective homomorphism of $\sgr{\F{3}}$-modules $\mathcal{A}\to \rpb{\F{3}}$. However the natural map $\F{3}\to \F{27}$ induces an a homomorphism of $\sgr{\F{3}}$-modules $\rpb{\F{3}}\to \rpb{\F{27}}$ which sends $\gpb{-1}$ to $\gpb{-1}$. By the results of \cite[Section 7]{hut:cplx13}, $\gpb{-1}$ has infinite order in $\rpb{\F{27}}$ and $\suss{1}{-1}$ has order $2$.  It follows that the composite homomorphism $\mathcal{A}\to \rpb{\F{27}}$ is injective and hence $\mathcal{A}\cong \rpb{\F{3}}$.
\end{proof}

Taking $\F{3}^\times$-coinvariants, we deduce:
\begin{cor}\label{cor:pbf3}  $\pb{\F{3}}$ is a cyclic abelian group of order $4$ with generator $\gpb{-1}$.
\end{cor}

\begin{cor}\label{cor:rpbkerf3}  
$\rpbker{\F{3}}$ is a trivial  $\sgr{\F{3}}$-module of order $2$ generated by $\suss{1}{-1}$.
\end{cor}

\begin{proof}
By Lemma \ref{lem:lambda1} (3), the map $\lambda_1:\rpb{\F{3}}\to\sgr{\F{3}}$ sends $\gpb{-1}$ to $-\pf{-1}^2=2\pf{-1}$, which has infinite order. It sends $\suss{1}{-1}$ to $0$.
It follows that $\rpbker{\F{3}}=\ker{\lambda_1}$ is generated by $\suss{1}{-1}$. 
\end{proof}

The group $B=B(\spl{2}{\F{3}})$ is nonabelian of order $6$. Thus we have $E^1_{2,0}(\spl{2}{\F{3}}, \LL)=\ho{2}{B}{\Z}=0$ and
 $E^1_{3,0}(\spl{2}{\F{3}}, \LL)=\ho{3}{B}{\Z}=\Z/6$. In particular, the map $d^3:E^3_{3,0}=\rpbker{\F{3}}\to E^3_{2,0}$ is the zero homomorphism. We immediately conclude:
\begin{cor}\label{cor:rblf3} 
 $\rbl{\F{3}}=\rpbker{\F{3}}$ is cyclic of order $2$ generated by $\suss{1}{-1}$.
\end{cor}

\begin{cor}\label{cor:blf3} $\bl{\F{3}}$ is cyclic of order $2$ with generator $\sus{-1}=2\gpb{-1}$.
\end{cor}

\begin{proof}
The map $\rpb{\F{3}}\to \pb{\F{3}}$ induces a map $\rbl{\F{3}}\to \bl{\F{3}}$ sending the generator $\suss{1}{-1}$ to $2\gpb{-1}$. Thus $\bl{\F{3}}$ contains the element $2\gpb{-1}$ of order $2$.

Now  $\bl{\F{3}}=\ker{d^3:\pb{\F{3}}\to E^3_{2,0}(\gl{2}{\F{3}},\LL)}$. The map $\F{3}\to\F{27}$ induces a commutative diagram
\[
\xymatrix{
\pb{\F{3}}\ar^-{d^3}[r]\ar[d]&E^3_{2,0}\ar[d]\\
\pb{\F{27}}\ar^-{\lambda_2}[r]&\asym{2}{\Z}{\F{27}^\times}=E^3_{2,0}(\gl{2}{\F{27}},\LL)\\
}
\]
where $\lambda_2(\gpb{-1})=-1\circ -1\in \asym{2}{\Z}{\F{27}^\times}$ which has order $2$. It follows that $d^3(\gpb{-1})\not=0$ and the result follows. 
\end{proof}

\begin{prop}\label{prop:kindf3}  There is a natural commutative diagram with exact rows
\[
\xymatrix{
0\ar[r]&\Z/12\ar[r]\ar@{>>}[d]&\ho{3}{\spl{2}{\F{3}}}{\Z}\ar[r]\ar@{>>}[d]&\rbl{\F{3}}\ar[r]\ar^-{\cong}[d]&0\\
0\ar[r]&\Z/4\ar[r]&\kind{\F{3}}\ar[r]&\bl{\F{3}}\ar[r]&0.\\
}
\]
\end{prop}

\begin{proof}  A direct calculation, shows that $\ho{3}{\spl{2}{\F{3}}}{\Z}\cong \Z/24$ (see, for example, \cite[Section 3]{hut:cplx13}). Furthermore, the natural map $\ho{3}{\spl{2}{\F{3}}}{\Z}\to \kind{\F{3}}$ induces an isomorphism $\Z/8\cong \ho{3}{\spl{2}{\F{3}}}{\Z[\frac{1}{3}]}\cong \kind{\F{3}}$ (again, see \cite[Section 3]{hut:cplx13}).
\end{proof}

We will use the following fact in our calculation of $\rbl{\Z}$ below:

\begin{cor}\label{cor:e12} In the spectral sequence $E(\spl{2}{\F{3}},\LL)$ we have $\Z/2\cong E^2_{1,2}=E^\infty_{1,2}$. 
\end{cor}

\begin{proof}  $E^2_{1,2}$ is the homology of
 \[
\xymatrix{\ho{1}{\spl{2}{\F{3}}}{\LL_3}\ar^-{d^1}[r] &\sgr{\F{3}}\otimes \mu_2\ar^-{d^1}[r] &\mu_2
}
\]
where the second differential is surjective. Thus $E^2_{1,2}$ has order $2$ or $0$.

Now the terms $E^\infty_{0,3}, E^\infty_{1,2},E^\infty_{2,1},E^\infty_{3,0}$ are the terms of a graded abelian group associated to a filtration on $\ho{3}{\spl{2}{\F{3}}}{\Z}\cong \Z/24$.
Since $E^1_{2,1}=\ho{2}{\mu_2}{\Z}=0$ , we have $E^\infty_{2,1}=0$. Furthermore, $E^\infty_{0,3}=\rbl{\F{3}}\cong \Z/2$ and $E^\infty_{3,0}=\ho{3}{B}{\Z}\cong \Z/6$. It follows 
that $E^\infty_{1,2}=\Z/2$ as required. 
\end{proof}
\subsection{The Bloch group of  $\Z$}

We begin by recalling the relevant classical facts about the structure and homology of the group $G=\spl{2}{\Z}$ (see, for example, \cite{serre:trees}):
\begin{thm}\label{thm:sl2z}  $\spl{2}{\Z}\cong \Z/4\star_{\Z/2}\Z/6$ where the first factor is generated by\\
 $\tilde{\omega}:=\left[
\begin{array}{cc}
0&-1\\
1&0\\
\end{array}
\right]$ and the second by $\tau:=
\left[
\begin{array}{cc}
0&-1\\
1&1\\
\end{array}
\right]$. The homomorphism $\spl{2}{\Z}\to \Z/12$ sending $\tilde{\omega} $ to $3$ and $\tau$ to $2$ induces an isomorphism on homology
\[
\ho{n}{\spl{2}{\Z}}{\Z}\cong\ho{n}{\Z/12}{\Z}=\left\{
\begin{array}{ll}
\Z,&n=0\\
\Z/12,& n\mbox{ odd }\\
0,& n>0 \mbox{ even}\\
\end{array}
\right.
\]
\end{thm}

 Now we have the spectral sequence $E^1_{p,q}(G,\LL)\implies \ho{p+q}{G}{\LL_\bullet}\cong\ho{p+q}{G}{\Z}$ 
for any  subgroup $G\subset \gl{2}{\Z}$. We consider the case $G=\spl{2}{\Z}$.

Note that $T(\spl{2}{\Z})=\mu_2:=\{ \pm I\}=Z(\spl{2}{\Z})$. Furthermore, $B:=B(\spl{2}{\Z})=\mu_2\times \rho^{\Z}$ where 
\[
\rho=\left[
\begin{array}{cc}
1&0\\
1&1\\
\end{array}
\right].
\]

Since $B$ is abelian, $\ho{1}{B}{\Z}=B$. By the K\"unnneth formula, $\ho{2}{B}{\Z}\cong \mu_2\otimes \Z=\Z/2$ and the inclusion $\mu_2\to B$ induces an isomorphism
$\Z/2=\ho{3}{\mu_2}{\Z}\cong \ho{3}{B}{\Z}$.

Thus the $E^1$-page of the spectral sequence has the form 
\[
\xymatrix{
0\ar[d]&0\ar[d]&0\ar[d]&\cdots\\
\rpb{\Z}\ar^-{d^1}[d]&\ho{1}{G}{\LL_3}\ar^-{d^1}[d]&\ho{2}{G}{\LL_3}\ar^-{d^1}[d]&\cdots\\
\sgr{\Z}\ar^-{\epsilon}[d]&\mathcal{R}\otimes_{\Z}\mu_2\ar^-{\epsilon\otimes \mathrm{id}}[d]&0\ar^-{0}[d]&\cdots\\
\Z\ar^-{0}[d]&\mu_2\ar^-{0}[d]&0\ar^-{0}[d]&\mu_2\ar^-{0}[d]\\
\Z&B&\Z/2&\Z/2\\
}
\]
and hence the $E^2$-page looks (in part) like:
\[
\xymatrix{
0&0&0&\cdots\\
\rpbker{\Z}\ar^-{d^2}[rdd]&E^2_{1,3}&E^2_{2,3}&\cdots\\
\aug{\Z}/\image{d^1}\ar^-{d^2}[rdd]&E^2_{1,2}\quad\ar^-{d^2}[rdd]&0&\cdots\\
0&0&0&\cdots\\
\Z&B&\Z/2&\Z/2\\
}
\]
\begin{lem}
 In this spectral sequence, $E^2_{3,0}=\ho{3}{\mu_2}{\Z}=\ho{3}{B}{\Z}=E^\infty_{3,0}$.
\end{lem}
\begin{proof}
Since the inclusion map $\Z/2\to \Z/12$ induces an injective map $\ho{3}{\Z/2}{\Z}\to \ho{3}{\Z/12}{\Z}$, it follows from Theorem \ref{thm:sl2z} that the inclusion map $\mu_2\to \spl{2}{\Z}$ induces an injective map $\ho{3}{\mu_2}{\Z}\to\ho{3}{\spl{2}{\Z}}{\Z}$. 
\end{proof}
\begin{lem} The map $d^1=\lambda_1:\rpb{\Z}\to \sgr{\Z}$ is the zero map, and hence
 $\rpbker{\Z}=\rpb{\Z}$.
\end{lem}
\begin{proof}
By Lemma \ref{lem:d2a}, the composite map $\aug{\Z}\to \aug{\Z}/\image{d^1}=E^2_{0,2}\to E^2_{1,0}=B$ sends $\pf{-1}$ to the matrix
$\left[ 
\begin{array}{ll}
-1&0\\
6&-1\\
\end{array}\right]$. Since $\aug{\Z}=\Z\cdot\pf{-1}\cong \Z$, this map is injective. It follows that $\image{d^1}=0$.
\end{proof}

\begin{lem}\label{lem:ze12}
 In the spectral sequence $E(\spl{2}{\Z},\LL)$, we have 
\begin{enumerate}
\item $E^3_{2,0}=E^2_{2,0}=H_2(B,\Z)$ and
\item $\Z/2\cong E^2_{1,2}=E^\infty_{1,2}$. 
\end{enumerate}
\end{lem}
\begin{proof}
\begin{enumerate}
\item  We first note that the differential $d^2:E^2_{1,2}\to E^2_{2,0}=H_2(B,\Z)$ is trivial by Corollary \ref{cor:d2b-1}. Hence $E^3_{2,0}=H_2(B,\Z)\cong B\wedge B$, which is cyclic of order $2$ with generator $-I\wedge \rho$.
\item The map $\Z\to\F{3}$ induces a map of spectral sequences $E(\spl{2}{\Z},\LL)\to E(\spl{2}{\F{3}},\LL)$. Thus we have a commutative diagram
\[
\xymatrix{
E^1_{1,3}=\ho{1}{\spl{2}{\Z}}{\LL_3}\ar[r]\ar^-{d^1}[d]&\ho{1}{\spl{2}{\F{3}}}{\LL_3}\ar^-{d^1=0}[d]\\
E^1_{1,2}=\sgr{\Z}\otimes\mu_2\ar^-{\cong}[r]\ar^-{\epsilon\otimes\mathrm{id}}[d]&\sgr{\F{3}}\otimes\mu_2=E^2_{1,2}\ar^-{\epsilon\otimes\mathrm{id}}[d]\\
E^1_{1,1}=\mu_2\ar^-{\cong}[r]&\mu_2=E^1_{1,1}\\
}
\]
which induces an isomorphism $E^2_{1,2}(\spl{2}{\Z},\LL)\cong E^2_{1,2}(\spl{2}{\F{3}},\LL)=\Z/2$ (using Corollary \ref{cor:e12}). 
Since  in $E(\spl{2}{\Z},\LL)$, we have $E^\infty_{1,2}=\ker{d^2:E^2_{1,2}\to E^2_{2,0}}$ the result follows from  Corollary \ref{cor:d2b-1} again.
\end{enumerate}
\end{proof}

\begin{lem}\label{lem:zsuss-1}
$\suss{1}{-1}\in\rpbker{\Z}\setminus\rbl{\Z}$ while $\cconst{\Z}\in\rbl{\Z}$.
\end{lem}
\begin{proof}
By Lemma \ref{lem:d3} in Appendix \ref{calc}, $d^3(\suss{1}{-1})=\rho\ast -I\not= 0$ in $E^3_{2,0}=H_2(B,\Z)$, by Lemma \ref{lem:z4}. So $\suss{1}{-1}\not\in \rbl{\Z}=\ker{d^3}$.
On the other hand, $\cconst{\Z}\in\rbl{\Z}$ by Corollary \ref{cor:cconstA}.
\end{proof}



\begin{thm}\label{thm:rblz}
 There is a short exact sequence
\[
0\to \Z/4\cong  \ho{3}{\langle \tilde{\omega} \rangle}{\Z}\to \ho{3}{\spl{2}{\Z}}{\Z}\to \rbl{\Z}\to 0
\]
and  $\rbl{\Z}$ is cyclic of order $3$, generated by the element $\cconst{\Z}$. 
\end{thm}

\begin{proof}  We again use the spectral sequence $E^1_{p,q}(\spl{2}{\Z},\LL)\implies \ho{p+q}{\spl{2}{\Z}}{\Z}$.

Now $E^1_{3,0}=\ho{3}{B}{\Z}$ where $B=B(\spl{2}{\Z})=\langle \pm I\rangle \times \langle \rho\rangle$ where 
\[
\rho=\left[
\begin{array}{cc}
1&0\\
1&1\\
\end{array}
\right].
\]
Thus $B\cong \mu_2\times \Z$ and, by the K\"unneth formula, the map $\mu_2\to B$ induces an isomorphism $\ho{3}{\mu_2}{\Z}\cong \ho{3}{B}{\Z}$. Since the map $\ho{3}{\mu_2}{\Z}\to \ho{3}{\spl{2}{\Z}}{\Z}$ is injective, it follows that $E^\infty_{3,0}=\ho{3}{\mu_2}{\Z}\cong \Z/2$.

We have $T(\spl{2}{\Z})=Z=\{ \pm I\}=\mu_2$.  Thus $E^1_{2,1}\cong\ho{2}{\mu_2}{\Z}=0$ and hence $E^\infty_{2,1}=0$.


 The spectral sequence thus gives us a  a short exact 
sequence 
\[
0\to H \to \ho{3}{\spl{2}{\Z}}{\Z}\to E^\infty_{0,3}:= \rbl{\Z}\to 0,
\]
 where $H$, in turn, fits into an exact sequence
\[
0\to E^\infty_{0,3}\cong \ho{3}{\mu_2}{\Z}\to H\to E^\infty_{1,2}\cong\Z/2\to 0
\]
by Lemma  \ref{lem:ze12}.
 Recall that 
\[
\ho{3}{\spl{3}{\Z}}{\Z}\cong \ho{3}{\langle \tilde{\omega} \rangle}{\Z}\oplus \ho{3}{\langle t \rangle}{\Z}\cong \Z/4\oplus \Z/3 \cong \Z/12.
\]

 It follows that  $H$ is cyclic of order $4$ and that $\rbl{\Z}\cong \Z/3$ as required. 
\end{proof}

\begin{rem}  Note that it follows that $\Z$ is not $L_\bullet$-acyclic in dimension $2$. For $\wn{\Z}=\emptyset$ and hence $X_3(\Z)=\emptyset$ (i.e., there are no $4$-cliques in the graph 
$\Gamma(\Z)$), but $\rpb{\Z}\not=0$ and hence $\LL_3(\Z)\not=0$.
\end{rem}


\begin{cor}\label{cor:rpz}  $\rpb{\Z}=\rpbker{\Z}=\pb{\Z}$ is cyclic of order $6$ generated by the element $\bconst{\Z}=\cconst{\Z}+\suss{1}{-1}$.
\end{cor}
\begin{proof} Recall that  $\bconst{\Z}\in \rpbker{\Z}=\rpb{\Z}$ has order $6$. Since there is a short exact sequence $0\to\rbl{\Z}\to\rpbker{\Z}\to E^3_{2,0}=\Z/2\to 0$, 
where $\rbl{\Z}$ is cyclic of order $3$, it follows that $\rpbker{\Z}$ has order $6$ and is generated by $\bconst{\Z}$. Thus $\rpb{\Z}$ is a trivial $\sgr{\Z}$-module and hence
$\rpb{\Z}=\pb{\Z}$ also.
\end{proof}

\begin{lem}\label{lem:h2b} Let $B=B(\spl{2}{\Z})$ as above and let $\tilde{B}:=B(\gl{2}{\Z})$. Then the  inclusion homomorphism $B\to\tilde{B}$ induces an injective map
$\ho{2}{B}{\Z}\to \ho{2}{\tilde{B}}{\Z}$.
\end{lem}
\begin{proof}
Consider the map of (split) group extensions
\[
\xymatrix{
1\ar[r]&U\ar[r]\ar^-{=}[d]&B\ar[r]\ar[d]&Z\ar[r]\ar[d]&1\\
1\ar[r]&U\ar[r]&\tilde{B}\ar[r]&C_1\times C_2\ar[r]&1
}
\]
where $\Z\cong U:=\an{\rho}$, $\mu_2\cong Z$ is generated by $\mathrm{diag}(-1,-1)$, $C_1$  is generated by $\sigma_1:=\mathrm{diag}(-1,1)$ and $C_2$ is generated by
 is generated by $\sigma_2:=\mathrm{diag}(1,-1)$.  This induces a map of the associated Hochschild-Serre spectral sequences. Considering the terms $E^2_{i,j}=\ho{i}{Z}{\ho{j}{U}{\Z}}$ of the spectral sequence of the top extension, we deduce that there is an isomorphism $E^2_{1,1}=\ho{1}{Z}{U}\cong \ho{2}{B}{\Z}$. Likewise, the spectral sequence for the lower extension 
yields a (split) short exact sequence
\[
0\to E^2_{1,1}=\ho{1}{C_1\times C_2}{U}\to \ho{2}{\tilde{B}}{\Z}\to \ho{2}{C_1\times C_2}{\Z}\to 0.
\]

Thus the Lemma is equivalent to the statement that the map $\ho{1}{Z}{U}\to \ho{1}{C_1\times C_2}{U}$ is injective. To see this, consider the (again split) extension 
\[
\xymatrix{
1\ar[r]& Z\ar[r]& C_1\times C_2\ar^-{p}[r]& C\ar[r]& 1\\
}
\]
where $C$ is cyclic of order $2$ generated by $\sigma$ and $p(\sigma_1^i,\sigma_2^j):= \sigma^{i+j}$. The exact sequence of terms of low degree for the associated spectral sequence (for the module $U$) has the form
\[
\ho{2}{C_1\times C_2}{U}\to \ho{2}{C}{U}\to \ho{1}{Z}{U}\to \ho{1}{C_1\times C_2}{U}\to \cdots
\]
But the leftmost arrow is a surjection since the morphism of pairs $(C_1\times C_2,U)\to (C,U)$ has a right-inverse. Thus the third arrow is an injection as claimed. 
\end{proof}

\begin{cor}\label{cor:blz} The natural map $\rbl{\Z}\to \bl{\Z}$ is an isomorphism. 
\end{cor}
\begin{proof} 
The functorial maps $\rbl{\Z}\to\bl{\Z}\to\bl{\Q}$ sends $\cconst{\Z}$ to $\cconst{\Q}\not=0$. So the map $\rbl{\Z}\to \bl{\Z}$ is injective.

Next we show that  $\sus{-1}\in \pb{\Z}\setminus\bl{\Z}$.

 Recall that $\bl{\Z}:= E^\infty_{0,3}(\gl{2}{\Z},\LL)=\ker{d^3:\pb{\Z}\to E^3_{2,0}(\gl{2}{\Z},\LL)}$. 
We first note that in the spectral sequence $E(\gl{2}{\Z},\LL)$, we have  
\[
\ho{2}{\tilde{B}}{\Z}=E^1_{2,0}=E^2_{2,0}=E^3_{2,0}:
\]
 (i)  The map $d^1:E^1_{2,1}=\ho{2}{\tilde{T}}{\Z}\to \ho{2}{\tilde{B}}{\Z}$ is the map $c_\omega-\mathrm{id}$ followed by corestriction.
 Here $c_\omega$ is the map induced by conjugation by $\omega$. This map is inversion on $\tilde{T}$ and hence induces the identity on $\ho{2}{\tilde{T}}{\Z}$. So the $d^1$ map is trivial and 
$E^1_{2,0}=E^2_{2,0}$.\\
(ii) The map $d^1:E^1_{1,2}=\ho{1}{\mu_2}{\Z}\to \ho{1}{\tilde{T}}{\Z}=E^1_{1,1}$ is just the inclusion map and hence $E^2_{1,2}=0$. It follows of course that 
$d^2:E_{1,2}^2\to E^2_{2,0}$ is the zero map and hence $E^2_{2,0}=E^3_{2,0}$. 

Thus we have a commutative diagram
\[
\xymatrix{
E^3_{0,3}(\spl{2}{\Z},\LL)=\rpb{\Z}\ar^-{d^3}[d]\ar^-{\cong}[r]&\pb{\Z}=E^3_{0,3}(\gl{2}{\Z},\LL)\ar^-{d^3}[d]\\
E^2_{2,0}(\spl{2}{\Z},\LL)=\ho{2}{B}{\Z}\ar[r]&\ho{2}{\tilde{B}}{\Z}=E^3_{2,0}(\gl{2}{\Z},\LL).\\
}
\]
where the lower horizontal arrow is an injection by Lemma \ref{lem:h2b}.
It follows that $\sus{-1}$ is not in the kernel of the right-hand vertical map, since $\suss{1}{-1}$  is not in the kernel of the left-hand vertical map; i.e., $\suss{1}{-1}\not\in \bl{\Z}$. 
Since $\pb{\Z}$ is cyclic of order $6$  generated by $\bconst{\Z}$, it follows that $\bl{\Z}$ is cyclic of order $3$ generated by $2\bconst{\Z}=\cconst{\Z}$.
\end{proof}

\begin{cor}\label{cor:sl2z} There is a commutative diagram with exact rows and columns
\[
\xymatrix{
&&0\ar[d]&0\ar[d]&\\
0\ar[r]&\ho{3}{\an{\tilde{\omega}}}{\Z}\ar[r]\ar^-{\cong}[d]&\ho{3}{\spl{2}{\Z}}{\Z}\ar[r]\ar[d]&\bl{\Z}\ar[r]\ar[d]&0\\
0\ar[r]&\widetilde{\mathrm{tor}(\mu_2,\mu_2)}\ar[r]&\kind{\Q}\ar[d]\ar[r]&\bl{\Q}\ar[d]\ar[r]&0\\
&&\Z/2\ar[d]\ar^-{\cong}[r]&\langle 2\gpb{-1} \rangle\ar[d]&\\
&&0&0&\\
}
\]
\end{cor}

\begin{proof}  The diagram commutes, the bottom row is exact  and the left hand top vertical arrow is an isomorphism by Theorem \ref{thm:kind} and Lemma \ref{lem:omega}. The top row is exact and the map $\rbl{\Z}=\bl{\Z}\to \bl{\Q}$ is injective by Theorem \ref{thm:rblz}. 
\end{proof}

\begin{rem}\label{rem:psl2z} If we let $\rbl{\pspl{2}{\Z}}:=E^\infty_{0,3}(\pspl{2}{\Z},\LL)$, then we have 
\[
\rbl{\pspl{2}{\Z}}=\ker{d^3:\rpbker{\Z}=E^3_{0,3}(\pspl{2}{\Z},\LL)\to E^3_{2,0}(\pspl{2}{\Z},\LL)}.
\]
Now $\rpbker{\Z}$ is generated by $\cconst{\Z}$ and $\suss{1}{-1}$.  As above $\cconst{\Z}\in\ker{d^3}$. But $\suss{1}{-1}\in \ker{d^3}$ also by Lemma \ref{lem:d3}, since $-I=1$ in $\pspl{2}{\Z}$.

 Thus $\rbl{\pspl{2}{\Z}}=\rpbker{\Z}=\Z/6\cdot\bconst{\Z}$. In particular, $\suss{1}{-1}\in\rbl{\pspl{2}{\Z}}$. It follows that the edge homomorphism $\ho{3}{\pspl{2}{\Z}}{\Z}\to E^\infty_{0,3}=\rpbker{\Z}\cong\Z/6$ is 
surjective. Since $\pspl{2}{\Z}$ is the free product of a group of order $2$ and a group of order $3$, we know that $\ho{3}{\pspl{2}{\Z}}{\Z}\cong \Z/6$ and thus the given edge 
homomorphism is an isomorphism. 

It follows that we have a commutative diagram 
\[
\xymatrix{
0\ar[r]&\ho{3}{\an{\tilde{\omega}}}{\Z}\ar[r]&\ho{3}{\spl{2}{\Z}}{\Z}\ar[r]\ar[d]&\rbl{\Z}\ar[r]\ar[d]&0\\
&&\ho{3}{\pspl{2}{\Z}}{\Z}\ar^-{\cong}[r]&\rpbker{\Z}\cong\bl{\Q}&\\
}
\]
where the right vertical arrow is an inclusion. However, comparing with corollary \ref{cor:sl2z}, we see that there is no natural compatible homomorphism $\ho{3}{\pspl{2}{\Z}}{\Z}\to\kind{\Q}$.
\end{rem}
\subsection{The Bloch group of $\Z[\frac{1}{2}]$}
By Example \ref{exa:acycz1m} the ring $\Z[\frac{1}{2}]$ is $\LL_\bullet$-acyclic. So there is a natural surjective map $\ho{3}{\spl{2}{\Z[\frac{1}{2}]}}{\Z}\to \rbl{\Z[\frac{1}{2}]}$. 

We will rely on the following description of the structure of $\ho{3}{\spl{2}{\Z[\frac{1}{2}]}}{\Z}$ as a $\Z$-module, due to Adem and Naffah (\cite{adem:naffah}):
\begin{prop}\label{prop:adnaf}
$\ho{3}{\spl{2}{\Z[\frac{1}{2}]}}{\Z}\cong \Z/8\oplus\Z/3\oplus \Z/3$.
\end{prop}

\begin{lem}\label{lem:X}  The element $\suss{1}{-1}$ lies in $\rbl{\Z[\frac{1}{2}]}$ and has  order $2$. There exists a class $X$ of order $8$ in $\ho{3}{\spl{2}{\Z[\frac{1}{2}]}}{\Z}$ mapping to 
$\suss{1}{-1}$. 
\end{lem}

\begin{proof}  
By Lemmas \ref{lem:1/2} and \ref{lem:d3}, $d^3(\suss{1}{-1})=0$ in $E^3_{2,0}(\spl{2}{\Z[\frac{1}{2}]},\LL)$ and hence $\suss{1}{-1}\in\rbl{\Z[\frac{1}{2}]}$ by definition.
Furthermore, $\suss{1}{-1}$ has order $2$ since its image under the map $\rbl{\Z[\frac{1}{2}]}\to \rbl{\Q}\to \bl{\Q}$ is $\sus{-1}=2\gpb{-1}$ which has order $2$. 

Thus, there exists $X$ in $\ho{3}{\spl{2}{\Z[\frac{1}{2}]}}{\Z}$ mapping to $\suss{1}{-1}$ in $\rbl{\Z[\frac{1}{2}]}$. Replacing $X$ by $3X$, if needed, we may assume $X$ has order dividing $8$. On the other hand, the image of $X$ under the map $\ho{3}{\spl{2}{\Z[\frac{1}{2}]}}{\Z}\to \ho{3}{\spl{2}{\Q}}{\Z}\to \kind{\Q}$ maps to $\sus{-1}\in\bl{\Q}$ and hence this image has order $8$ by Theorem \ref{thm:kind}. Thus $X$ has order $8$.
\end{proof}

\begin{lem} \label{lem:dz12}The elements $\cconst{\Z[\frac{1}{2}]}$, $\an{2}\cconst{\Z[\frac{1}{2}]}\in \rbl{\Z[\frac{1}{2}]}$ generate a subgroup of type $\Z/3\oplus \Z/3$.
\end{lem}

\begin{proof}  Let $T$ denote the composite homomophism $\rbl{\Z[\frac{1}{2}]}\to \rbl{\Q}\to\bl{\Q}$.  Note that $T$ is a $\sgr{\Z[\frac{1}{2}]}$-homomorphism where 
$\bl{\Q}$ has the trivial module structure. Thus $T(\cconst{\Z[\frac{1}{2}]})=T(\an{2}\cconst{\Z[\frac{1}{2}]})=\cconst{\Q}$, which has order $3$. 
Thus $\cconst{\Z[\frac{1}{2}]}$ and $\an{2}\cconst{\Z[\frac{1}{2}]}$ each have order $3$ in $\rbl{\Z[\frac{1}{2}]}$.

Now give $\pb{\F{2}}$ the $\sgr{\Q}$-module structure: $\an{q}x:= (-1)^{v_2(q)}x$ for $q\in \Q^\times$, $x\in \pb{\F{2}}$. Recall that $\pb{\F{2}}$ is cyclic of order $3$ with generator $\bconst{\F{2}}=-\cconst{\F{2}}$.  There is a well-defined $\sgr{\Q}$-homomorphism 
\[
S_2:\rbl{\Q}\to \pb{\F{2}},
\gpb{q}\mapsto
\left\{
\begin{array}{ll}
\bconst{\F{2}},& v_2(q)>0\\
0,&v_2(q)=0\\
-\bconst{\F{2}},& v_2(q)<0.\\
\end{array}
\right.
\]
Composing this with the map $\rbl{\Z[\frac{1}{2}]}\to \rbl{\Q}$ gives a $\sgr{\Z[\frac{1}{2}]}$-homomorphism $S':\rbl{\Z[\frac{1}{2}]}\to\pb{\F{2}}$. We have 
$S'(\cconst{\Z[\frac{1}{2}]})=S_2(\cconst{\Q})=\cconst{\F{2}}$. By the module structure on $\pb{\F{2}}$ it follows that 
\[
S'(\an{2}\cconst{\Z[\frac{1}{2}]})-S'(\cconst{\Z[\frac{1}{2}]})=S'(\pf{2}\cconst{\Z[\frac{1}{2}]})=\pf{2}\cconst{\F{2}}=-2\cconst{\F{2}}=\cconst{\F{2}}
\]
has order $3$.
\end{proof}
\begin{prop}\label{prop:rblz12}
We have $\rbl{\Z[\frac{1}{2}]}\cong \Z/2\oplus/\Z/3\oplus \Z/3$ with generators $\suss{1}{-1}$, $\cconst{\Z[\frac{1}{2}]}$ and $\an{2}\cconst{\Z[\frac{1}{2}]}$. The square class $\an{-1}$ acts trivially on $\rbl{\Z[\frac{1}{2}]}$. The square class $\an{2}$ is trivial on the first factor and interchanges the last two factors.

There is a commutative diagram  of $\sgr{\Z[\frac{1}{2}]}$-modules with exact rows  (where $\sgr{\Z[\frac{1}{2}]}$ acts trivially on the terms of the bottom row) and columns:
\[
\xymatrix{
&&0\ar[d]&0\ar[d]&\\
&&\langle \pf{2}\cconst{\Z[\frac{1}{2}]}\rangle\ar[d]\ar^-{\cong}[r]&\langle \pf{2}\cconst{\Z[\frac{1}{2}]}\rangle\cong \Z/3\ar[d]&\\
0\ar[r]& \ho{3}{\langle \omega \rangle}{\Z}\ar^-{\cong}[d]\ar[r]& \ho{3}{\spl{2}{\Z\left[\frac{1}{2}\right]}}{\Z}\ar@{>>}[d]\ar[r]& \rbl{\Z\left[\frac{1}{2}\right]}\ar@{>>}[d]\ar[r]& 0\\
0\ar[r]&\widetilde{\mathrm{tor}(\mu_2,\mu_2)}\ar[r]&\kind{\Q}\ar[d]\ar[r]&\bl{\Q}\ar[d]\ar[r]&0\\
&&0&0.&\\
}
\]
Thus there is a natural isomorphism 
$
\ho{3}{\spl{2}{\Z\left[\frac{1}{2}\right]}}{\Z}\cong \rbl{\Z[\frac{1}{2}]}\times_{\bl{\Q}}\kind{\Q}.
$
\end{prop}
\begin{proof} By Lemmas \ref{lem:X} and \ref{lem:dz12}, $\rbl{\Z[\frac{1}{2}]}$ contains a subgroup generated by $\suss{1}{-1}$, $\cconst{\Z[\frac{1}{2}]}$  of type 
$\Z/2\oplus/\Z/3\oplus \Z/3$. The map $ \Z/4\cong \ho{3}{\langle \omega \rangle}{\Z}\to \ho{3}{\spl{2}{\Z\left[\frac{1}{2}\right]}}{\Z}$ is injective since the injective map 
$\ho{3}{\langle \omega \rangle}{\Z}\to \kind{\Q}$ factors through this map. Furthermore, the image of this map is in the kernel of the surjective map
$\ho{3}{\spl{2}{\Z\left[\frac{1}{2}\right]}}{\Z}\to \rbl{\Z\left[\frac{1}{2}\right]}$. By Proposition \ref{prop:adnaf}, it follows that the subgroup in question is all of 
$\rbl{\Z\left[\frac{1}{2}\right]}$, and that the top row is a short exact sequence. 

With regard to the $\sgr{\Z[\frac{1}{2}]}$-module structure on $\rbl{\Z\left[\frac{1}{2}\right]}$, it is known (see section \ref{sec:prebloch}) that $\an{-1}$ acts trivially on $\suss{1}{-1}$ and $\cconst{\Z[\frac{1}{2}]}$. Clearly $\an{2}$ must fix the unique element $\suss{1}{-1}$ 
of order $2$. (Alternatively, deduce this from Corollary \ref{cor:ann-1}.) Finally, $\an{2}$ interchanges the last two factors by Lemma \ref{lem:dz12}.

Finally, the bottom row is known to be exact and the diagram clearly commutes. The left-hand vertical arrow is an isomorphism by Lemma \ref{lem:omega}. The right-hand vertical arrow is surjective since $\suss{1}{-1}-\cconst{\Z[\frac{1}{2}]}=\bconst{\Z[\frac{1}{2}]}\in \rbl{\Z[\frac{1}{2}]}$ maps to the generator $\bconst{\Q}$ of $\bl{\Q}$. 
\end{proof}

\begin{thm}\label{thm:rpz12}  The $\sgr{\Z[\frac{1}{2}]}$-module $\rpb{\Z[\frac{1}{2}]}$ is generated by $\gpb{-1}$, $\gpb{2}$ and $\gpb{\frac{1}{2}}$; i.e., it is generated by $\{ \gpb{u}\ |\ u\in \wn{\Z[\frac{1}{2}]}\}$.  As an abelian group it has the structure
\[
\rpb{\Z\left[\frac{1}{2}\right]}=\rbl{\Z\left[\frac{1}{2}\right]}\oplus \Z\cdot\gpb{-1}\oplus \Z\cdot\gpb{\frac{1}{2}}\cong \Z/2\oplus\Z/3\oplus\Z/3\oplus\Z\oplus\Z
\]
where the generators of the factors are $\suss{1}{-1}$, $\cconst{\Z[\frac{1}{2}]}$, $\an{2}\cconst{\Z[\frac12]}$, $\gpb{-1}$ and $\gpb{2}$ respectively. 
\end{thm}

The theorem is proven in Lemmas \ref{lem:-1u} to \ref{lem:pf}:

In Lemmas \ref{lem:-1u} to \ref{lem:pf}, $\mathcal{R}$ will denote $\sgr{\Z[\frac{1}{2}]}$, $D$ will denote $\cconst{\Z[\frac{1}{2}]}$ and $C$ will denote $\bconst{\Z[\frac{1}{2}]}$.
\begin{lem}\label{lem:-1u}  Let $\epsilon:=\pf{-1}\gpb{-1}=\an{-1}\gpb{-1}-\gpb{-1}\in \mathcal{R}\cdot\gpb{-1}\subset \rpb{\Z[\frac{1}{2}]}$. Then\\
$\pf{-1}\gpb{u}=\epsilon$ in $\rpb{\Z[\frac{1}{2}]}$ for $u=-1,2$ and $1/2$.
\end{lem}
\begin{proof}
Certainly $\pf{-1}\gpb{-1}=\epsilon$ by definition.  By Corollary \ref{cor:ann-1} we have $\pf{2}\suss{1}{-1}=0=\pf{-1}\suss{1}{2}$. By Proposition \ref{prop:bconst} we  have
\[
C=\gpb{\frac12}+\an{-1}\gpb{-1}+\pf{-1}\suss{1}{2}=\gpb{-1}+\an{-1}\gpb{\frac12}+\pf{2}\suss{1}{-1}
\]
which thus  implies $\pf{-1}\gpb{\frac12}=\pf{-1}\gpb{-1}=\epsilon$.

Again $0=\pf{-1}\suss{1}{2}$ gives $\an{-1}(\gpb{2}+\an{-1}\gpb{\frac{1}{2}})=\gpb{2}+\an{-1}\gpb{\frac{1}{2}}$ which implies $\pf{-1}\gpb{\frac{1}{2}}=\pf{-1}\gpb{2}=\epsilon$.
\end{proof}

\begin{lem}  \label{lem:D}
$D=\gpb{\frac12}-\gpb{-1}$ in $\rpb{\Z[\frac{1}{2}]}$.
\end{lem}
\begin{proof}
We have $
D+\suss{1}{-1}=C= \gpb{\frac12}+\an{-1}\gpb{-1}=(\gpb{\frac12}-\gpb{-1})+\suss{1}{-1}$.
\end{proof}
\begin{lem} The ring homomorphism $\Z[\frac{1}{2}]\to\F{3}$ induces an isomorphism of $\mathcal{R}$-modules $\mathcal{R}\cdot\gpb{-1}\cong \rpb{\F{3}}$.
\end{lem}
\begin{proof} By Theorem \ref{thm:key}  we have 
\begin{eqnarray*}
\pf{2}D=\suss{2}{2}-\suss{1}{2}&=&\gpb{2}+\an{2}\gpb{\frac12}-\gpb{2}-\an{-1}\gpb{\frac12}\\
&=&
\an{2}\gpb{\frac12}-\an{-1}\gpb{\frac12}.\\
\end{eqnarray*}
Thus $\pf{2}(\gpb{\frac12}-\gpb{-1})=\an{2}\gpb{\frac12}-\an{-1}\gpb{\frac12}$ by Lemma \ref{lem:D}. This gives 
$\pf{2}\gpb{-1}=\pf{-1}\gpb{\frac12}$ and thus $\pf{2}\gpb{-1}=\epsilon=\pf{-1}\gpb{-1}$ by Lemma \ref{lem:-1u}. It follows that $\an{2}\gpb{-1}=\an{-1}\gpb{-1}$
 in $\rpb{\Z[\frac{1}{2}]}$. It follows that the $\mathcal{R}$-module structure of $\mathcal{R}\cdot\gpb{-1}$ factors through 
$\sgr{\F{3}}$.

 The result now follows since $\rpb{\F{3}}$ is generated by $\gpb{-1}$ as an $\sgr{\F{3}}$-module subject only to the relation $2\suss{1}{-1}=0$. 
\end{proof}

Now let $\mathcal{M}$ denote the $\mathcal{R}$-submodule $\rbl{\Z[\frac{1}{2}]}+\mathcal{R}\cdot\gpb{-1}=\rbl{\Z[\frac{1}{2}]}\oplus \Z\cdot\gpb{-1}$
of $\rpb{\Z[\frac{1}{2}]}$.

\begin{lem}\label{lem:2}  $\mathcal{M}$ is the submodule of $\rpb{\Z[\frac12]}$ generated by  $\gpb{\frac12}$. 
\end{lem}

\begin{proof}  From the description of $\rbl{\Z[\frac{1}{2}]}$, $\mathcal{M}$ is generated by
$\gpb{-1}$ and $D$.  Since $\gpb{-1}=\gpb{\frac12}-D$, it is enough to show that $D\in \mathcal{R}\cdot\gpb{\frac12}$: Since $\gpb{\frac12}=\gpb{-1}+D$ and since $\an{-1}D=D$, we have 
$\gpb{\frac12}+\an{-1}\gpb{\frac12}=\suss{1}{-1}+2D$. Multiplying both sides by $2$ shows that $4D=D$ lies in $ \mathcal{R}\cdot\gpb{\frac12}$.
\end{proof}

\begin{lem}\label{lem:123}  The element  $\gpb{2}$ has infinite order in $\rpb{\Z[\frac{1}{2}]}$ and $\Z\cdot\gpb{2}\cap\mathcal{M}=0$. Furthermore, 
$\mathcal{N}:=\mathcal{M}+\Z\cdot\gpb{2}=\mathcal{M}\oplus\Z\cdot\gpb{2}$ is an $\mathcal{R}$-submodule.
\end{lem}

\begin{proof}  We use the $\mathcal{R}$-homorphism $\lambda_1:\rpb{\Z[\frac{1}{2}}\to \aug{\Z[\frac{1}{2}]}$. Recall that $\aug{\Z[\frac{1}{2}]}$ is a free $\Z$-module with 
basis $\{ \pf{-1},\pf{2},\pf{-2}\}$. Since $\lambda_1(\gpb{2}=-\pf{\frac{1}{2}}^2=-\pf{2}^2=2\pf{2}$ it follows that $\gpb{\frac{1}{2}}$ has infinite order in 
$\rpb{\Z[\frac{1}{2}]}$. Since $\lambda_1(x)=0$ for $x\in \rbl{\Z[\frac{1}{2}]}$  and since $\lambda_1(\gpb{-1})=-\pf{-1}\pf{2}=-\pf{-2}+\pf{-1}+\pf{2}$, the second statement follows.

Finally, we must prove that $\mathcal{R}\cdot\gpb{2}\subset \mathcal{N}$: Certainly $\an{-1}\gpb{2}=\gpb{2}+\epsilon\in \mathcal{N}$.
The cocycle relation gives the identity $\an{2}\suss{1}{\frac{1}{2}}=\suss{1}{1}-\suss{1}{2}=-\suss{1}{2}$. Thus, by definition, $\an{2}\gpb{\frac{1}{2}}+\an{-2}\gpb{2}=-\gpb{2}-\an{-1}\gpb{\frac{1}{2}}$ and hence $\an{-2}\gpb{2}\in\mathcal{N}$ also, since the other three terms lie in $\mathcal{N}$.
\end{proof}

\begin{lem}\label{lem:e02} In the spectral sequence $E(\spl{2}{\Z[\frac{1}{2}]},\LL)$ we have $E^3_{2,0}=0$. Hence $\rpbker{\Z[\frac{1}{2}]}=\rbl{\Z[\frac{1}{2}]}$.
\end{lem}

\begin{proof}  $E^1_{2,0}=\ho{2}{B}{\Z}\cong \ho{2}{T}{\Z}$ (by Lemma \ref{lem:1/2})$\cong\ho{2}{\Z[\frac{1}{2}]^\times}{\Z}\cong \Lambda^2(\Z[\frac{1}{2}]^\times)=
-1  \wedge \Z[\frac{1}{2}]^\times$. Since the differential $d^2:E^2_{1,2}=\aug{\Z[\frac{1}{2}]}\otimes \mu_2 \to E^2_{2,0}\cong  \Lambda^2(\Z[\frac{1}{2}]^\times)$ sends $\pf{u}\otimes -1$ to $-1\wedge u$ by Corollary \ref{cor:d2b}, the 
first statement follows. The second statement now follows since $\rbl{A}=\ker{d^3:\rpbker{A}\to E^3_{2,0}}$.
\end{proof}

\begin{lem}\label{lem:b12}
 We have $\ho{1}{B}{\Z}=\ab{B}=C\oplus T$ where $C$ is cyclic of order $3$ with generator $\rho:=
\left[
\begin{array}{cc}
1&0\\
1&1\\
\end{array}
\right]$.
\end{lem}
\begin{proof} Recall that we have a (split) extension of groups $1\to U\to B\to T\to 1$ where $U =\{ E(a)\ |\ a\in \Z[\frac12]\}\cong\Z[\frac12]$, where
$E(a):=
\left[
\begin{array}{cc}
1&0\\
a&1\\
\end{array}
\right]$ and $\ab{B}=B/[B,B]=(U/[B,B])\oplus T$.

We have $T\cong\Z[\frac12]^\times$ via $D(u)\leftrightarrow u$ where $D(u):=\mathrm{diag}(u,u^{-1})$. It is generated by $D(-1)$ and $D(2)$. Now
\[
[D(u^{-1}),E(a)]=D(u^{-1})E(a)D(u)E(-a)=E((u^2-1)a).
\]
Since $2^2-1=3$ and $(-1)^2-1=0$, it follows that $U/[B,B]\cong \Z[{\frac12}]/3\Z[\frac12]\cong \Z/3$ as required.
\end{proof}

\begin{cor}\label{cor:I} The image of the map $\lambda_1:\rpb{\Z[\frac12]}=E^1_{0,3}\to E^1_{0,2}=\mathcal{R}$ lies in $\aug{\Z[\frac{1}{2}]}^2$ and $\aug{\Z[\frac{1}{2}]}^2/\image{\lambda_1}\cong \Z$. 
\end{cor}
\begin{proof} We have $E^2_{0,2}=\frac{\aug{\Z[\frac12]}}{\image{\lambda_1}}$ and \[
E^2_{1,0}=\ab{B}/T^2=C\oplus T/T^2\cong C\oplus \left(\Z[\frac12]^\times/(\Z[\frac12]^\times)^2\right).
\]
By Lemma \ref{lem:d2a} the map $d^2:E^2_{0,2}\to E^2_{1,0}$ is then given by
\begin{eqnarray*}
d^2(\pf{u})&=& \left[
\begin{array}{cc}
u&0\\
3(1-u)&u^{-1}\\
\end{array}
\right]\\
&=&E(3u^{-1}(1-u))D(u)\mbox{ in } \ab{B}=C\oplus T/T^2\\
&=& {D(u)}\in T/T^2
\end{eqnarray*}
since $E(3a)=0$ in $\ab{B}$ by Lemma \ref{lem:b12}. But the kernel of the map\\ 
$
\aug{A}\to A^\times/(A^\times)^2, \pf{u}\mapsto\bar{u}
$
 is
$\aug{A}^2$ and so the first statement follows.

Thus $E^\infty_{0,2}=E^3_{0,2}=\aug{\Z[\frac{1}{2}]}^2/\image{\lambda_1}$. Since, in the spectral sequence  $E(\spl{2}{\Z[\frac{1}{2}]},\LL)$  we have $E^\infty_{1,1}=E^\infty_{2,0}=0$ it follows that  $\ho{2}{\spl{2}{\Z[\frac{1}{2}]}}{\Z}\cong 
E^\infty_{0,2}$. But $\ho{2}{\spl{2}{\Z[\frac{1}{2}]}}{\Z}\cong \Z$ (see, for example, \cite{adem:naffah}) and so the second statement follows.
\end{proof}

We complete the proof of Theorem \ref{thm:rpz12}:
\begin{lem} \label{lem:pf} $\mathcal{N}=\rpb{\Z[\frac{1}{2}]}$.
\end{lem}
\begin{proof}
Since $\rbl{\Z\frac12]}=\rpbker{\Z[\frac12]}=\ker{\lambda_1}$ by Lemma \ref{lem:e02}, there is a commutative diagram with exact rows
\[
\xymatrix{
0\ar[r]&\frac{\mathcal{N}}{\rbl{\Z[\frac12]}}\ar^-{\lambda_1}[r]\ar[d]&\aug{\Z[\frac{1}{2}]}^2\ar[r]\ar^-{=}[d]&\frac{\aug{\Z[\frac{1}{2}]}^2}{\lambda_1(\mathcal{N})}\ar[r]\ar[d]&0\\
0\ar[r]&\frac{\rpb{\Z[\frac12]}}{\rbl{\Z[\frac12]}}\ar^-{\lambda_1}[r]&\aug{\Z[\frac{1}{2}]}^2\ar[r]&\frac{\aug{\Z[\frac{1}{2}]}^2}{\image{\lambda_1}}\ar[r]&0.\\
}
\]

Now $\aug{\Z[\frac{1}{2}]}^2$ is a free $\Z$-module with basis $\pf{-1}\pf{2}$, $2\pf{-1}$, $2\pf{2}$.  We have $\lambda_1(\gpb{-1})=\pf{-1}\pf{2}$ and $\lambda_1(\gpb{2})=
\pf{2}^2=-2\pf{2}$. Since these two elements already generate a free summand with quotient $\Z$, it follows from Corollary \ref{cor:I} that they generate the whole image of $\lambda_1$.
Since $\gpb{-1},\gpb{2}\in\mathcal{N}$, it follows that $\lambda_1(\mathcal{N})=\lambda_1(\rpb{\Z[\frac{1}{2}]}:=\image{\lambda_1}$. Thus the left-hand vertical arrow is also an isomorphism, and the result follows.
\end{proof}


\begin{cor}\label{cor:pbz12}  The abelian group $\pb{\Z[\frac{1}{2}]}\cong \Z/12\oplus \Z/2$ has generators $\gpb{2}$, $\gpb{\frac{1}{2}}$, both of order $12$  and satisfying
$2(\gpb{2}+\gpb{\frac{1}{2}})=0$. We have  $\bconst{}=10\gpb{\frac12}$ of order 6,  $\cconst{}=4\gpb{\frac12}$ of order $3$ and $\gpb{-1}=9\gpb{\frac12}$ of order $4$.
\end{cor}

\begin{proof} By Theorem \ref{thm:rpz12}, $\rpb{\Z[\frac{1}{2}]}\cong \Z/2\oplus\Z/3\oplus\Z/3\oplus\Z\oplus\Z$  where the generators of the factors are $\suss{1}{-1}$, $\cconst{}$, $\an{2}\cconst{}$, $\gpb{-1}$ and $\gpb{2}$ respectively.  Let $P$ be the abelian group $\Z/12\oplus \Z/2$ and let $f:\rpb{\Z[\frac{1}{2}]}\to P$ be the group homomorphism\\
 $f(\suss{1}{-1})=(6,0)$, $f(\cconst{})=f(\an{2}\cconst{})=(4,0)$, $f(\gpb{-1})=(9,0)$ and\\
 $f(\gpb{2})=
(-1,1)$.    So $f(\gpb{\frac12})=f(\gpb{-1}+\cconst{})=(9,0)+(4,0)=(1,0)$ and $f(\gpb{2}+\gpb{\frac{1}{2}})=(0,1)$.

Now $\aug{\Z[\frac{1}{2}]}\rpb{\Z[\frac{1}{2}]}$ is generated, as a $\Z$-module, by the following elements:\\
$\pf{-1}\gpb{-1}=\pf{2}\gpb{-1}=\pf{-1}\gpb{2}=\epsilon= \suss{1}{-1}-2\gpb{-1}$,\\
 $\pf{2}\cconst{}
=\an{2}\cconst{}-\cconst{}$ and $\pf{-2}\gpb{2}$.\\
(Recall that that $\pf{-1}\suss{1}{-1}=0=\pf{2}\suss{1}{-1}$ and  $\pf{-1}\cconst{}=0$.)

Now $f(\epsilon)=(6,0)-2(9,0)=-(12,0)=0$ and  $f(\an{2}\cconst{}-\cconst{})=0$ by construction.
Note also that it follows that 
\[
 f\left(\an{2}\gpb{\frac12}\right)=f(\an{2}(\gpb{-1}+\cconst{}))=f(\gpb{-1}+\cconst{})=f\left(\gpb{\frac12}\right).
\]
Now $\pf{-2}\gpb{2}=-2\gpb{2}-\an{-1}\gpb{\frac12}-\an{2}\gpb{\frac12}$ by the proof of Lemma \ref{lem:123}. Thus 
\begin{eqnarray*}
\pf{-2}\gpb{2}&=&-2\gpb{2}-\an{-1}(\gpb{-1}+\cconst{})-\an{2}(\gpb{-1}+\cconst{})\\
&=&-2\gpb{2}-\epsilon-\gpb{-1}-\cconst{}-\epsilon-\gpb{-1}-\an{2}\cconst{}\\
&=&-2(\gpb{2}+\gpb{-1}+\cconst{})-(\pf{2}\cconst{}+2\epsilon).
\end{eqnarray*}
Since the second term lies in $\ker{f}$ and since $f(\gpb{2}+\gpb{-1}+\cconst{})=(0,1)$ of order $2$, it follows that $f(\pf{-2}\gpb{2})=0$. Thus $\aug{\Z[\frac{1}{2}]}\rpb{\Z[\frac{1}{2}]}\subset \ker{f}$.

Since the elements $\pf{2}\cconst{}$, $\epsilon=\suss{1}{-1}-2\gpb{-1}$ and $2(\gpb{2}+\gpb{-1}+\cconst{})$ generate $\ker{f}$, it follows that $\ker{f}=\aug{\Z[\frac{1}{2}]}\rpb{\Z[\frac{1}{2}]}$ and hence $f$ induces an isomorphism
\[
\pb{\Z[\mbox{$\frac{1}{2}$}]}=\frac{\rpb{\Z[\frac{1}{2}]}}{\aug{\Z[\frac{1}{2}]}\rpb{\Z[\frac{1}{2}]}}\cong P.
\]
\end{proof}

\begin{cor}\label{cor:blz12} The ring homomophism $\Z[\frac{1}{2}]\to \Q$ induces an injective map\\
 $\pb{\Z[\frac{1}{2}]}\to\pb{\Q}$ and an isomorphism $\bl{\Z[\frac{1}{2}]}\cong\bl{\Q}$.
In particular, $\bl{\Z[\frac{1}{2}]}$ is cyclic of order $6$ with generator $\bconst{\Z[\frac{1}{2}]}$.
\end{cor}

\begin{proof}
We note first that $\gpb{\frac12}\in\pb{\Q}$ has order $12$ since $\bconst{\Q}=\gpb{\frac12}+\gpb{-1}$ and hence $\gpb{\frac12}=\bconst{\Q}-\gpb{-1}$ where $3\bconst{\Q}=\sus{-1}=2\gpb{-1}$ 
and $2\sus{-1}=0$.  Furthermore, in $\pb{\Q}$ we have $0=2\sus{2}=2(\gpb{2}+\gpb{\frac{1}{2}})$.

Consider now the homomorpism\\
 $\lambda:\pb{\Q}\to E^3_{2,0}(\gl{2}{\Q},L)=\asym{2}{\Z}{\Q^\times}, \gpb{u}\mapsto (1-u^{-1})\asymm u^{-1}$.  (See Proposition \ref{prop:lambda}.)

 Since $\lambda(\gpb{\frac12})=-1\asymm 2$ and $\lambda(\gpb{2})=\frac{1}{2}\circ\frac{1}{2}=2\circ 2$, which  is not a multiple of $-1\circ 2$ in $\asym{2}{\Z}{\Q^\times}$, it follows that $\gpb{2}$ is not a multiple of $\gpb{\frac{1}{2}}$ in $\pb{\Q}$. Thus the map $\pb{\Z[\frac{1}{2}]}\to\pb{\Q}$ is injective. 

It follows that the induced  map $\bl{\Z[\frac{1}{2}]}\to\bl{\Q}$ is injective. But $\bconst{\Z[\frac{1}{2}]}\in \rbl{\Z[\frac{1}{2}]}$  and hence $\bconst{\Z[\frac{1}{2}]}\in \bl{\Z[\frac{1}{2}]}$. Since this maps to the generator $\bconst{\Q}$ of $\bl{\Q}$, the result follows.
\end{proof}

\appendix

\section{The calculation of some differentials}\label{calc}

Let $A$ be a ring. For a group $G$, we let $C_\bullet(G)$ denote the (left) standard resolution of $\Z$ as a $\Z[G]$-module. Thus, for any group $G$ mapping to $\pgl{2}{A}$,
 the spectral sequence $E(G,\LL)$ can be constructed from the underlying double complex $D_{\bullet,\bullet}:=\LL_\bullet\otimes_{\Z[G]}C_\bullet(G)$, whose associated total complex is just the total tensor product of the complexes. We will let $d^h:=\id{\LL}\otimes d$ and $d^v=d\otimes\id{C}$ be the associated horizontal and vertical differentials. Thus we obtain $E^1_{p,q}=\ho{p}{G}{\LL_q}$ by first taking homology with respect to $d^h$. The differentials $d^1:E^1_{p,q}\to E^1_{p,q-1}$ are then induced by $d^v$.

\emph{In this appendix, $G$ will always denote the group $\spl{2}{A}$}.

Below we will use several times the elements 
\[
-I:=\mathrm{diag}(-1,-1),\quad \tilde{\omega}:=
\left[
\begin{array}{ll}
0&-1\\
1&0\\
\end{array}
\right]
\mbox{ and }
\rho:=
E_{21}(1)=
\left[
\begin{array}{ll}
1&0\\
1&1\\
\end{array}
\right]
\]
in $G:=\spl{2}{A}$.

We recall that $\mu_2=\mu_2(A)=\{ \epsilon\in A^\times \ |\ \epsilon^2=1\}$ and also identify $\mu_2$ with $\mu_2\cdot I=Z(G)$, the centre of $G$.

\subsection{The $d^1$-differentials}
By Corollary \ref{cor:lnapsl}, the modules $\LL_q$, $q\leq 2$ and $L_q$, $q\geq 0$ are permutation modules and the $d^1$ differentials are maps $E^1_{p,q}=\ho{p}{G}{L_q}\to \ho{p}{G}{L_{q-1}}=E^1_{p,q-1}$ induced by differentials  $L_q\to L_{q-1}$ which are maps of right $\Z[G]$ permutation modules. Thus the following proposition allows us to describe explicitly the maps $d^1$ in the spectral sequence $E(G,\LL)$ (and in $E(G,L)$).

 Let $H$ be a subgroup of $G$. For $g\in G$, we let $^gh:=ghg^{-1}$ (resp. $h^g:=g^{-1}hg$) for any $h\in H$ and $^gH:=gHg^{-1}$ (resp. $H^g:=g^{-1}Hg$).
We let $g\cdot$ denote the functorial map $\ho{\bullet}{H}{\Z}\to\ho{\bullet}{{}^gH}{\Z}$ induced by the conjugation by $g$  isomorphism 
$H\to {}^gH$.

Let $X_1$ and $X_2$ be transitive right $\Z[G]$-modules. For $i=1,2$, let $x_i\in X_i$ have stabiizer subgroup $H_i$ in $G$.

Now let $\phi:\Z[X_1]\to \Z[X_2]$ be a homomorphism of $\Z[G]$-modules.  Since $G$ act transitively on $X_1$, $\phi$ is determined by its value at $x_1$. Suppose that 
$\phi(x_1)=\sum_{g\in H_2\backslash G}n_g\cdot (x_2g)$ with $n_g\in \Z$.  Let 
\[
\phi_\bullet:\ho{\bullet}{H_1}{\Z}\cong \ho{\bullet}{G}{\Z[X_1]}\to\ho{\bullet}{G}{\Z[X_2]}\cong\ho{\bullet}{H_2}{\Z}
\]
be the map induced by $\phi$, and the Shapiro isomorphisms, on the homology of the groups $H_i$. Let $\supp{\phi}:=\{ g\in H_2\backslash G\ |\ n_g\not=0\}$.
\begin{prop}\label{prop:phi}\cite[Lemma 3]{hut:mat}
\begin{enumerate}
\item Let $E:=H_2\backslash G/H_1$. Then $n_g$ depends only on the class of $g$ in $E$.
\item For all $g\in \supp{\phi}$, the index $[H_1:H_1\cap H_2^g]$ is finite.
\item For $g\in G$, Let $K_g:={}^gH_1\cap H_2$. Then 
\[
\phi_\bullet(z)=\sum_{g\in E}n_g\Cor{K_g}{H_2}\Res{^gH_1}{K_g}\circ g\cdot z, \quad z\in \ho{\bullet}{H_1}{\Z}.
\]
\end{enumerate}
\end{prop}

\begin{cor}\label{cor:phi} Let $G$, $X_i$, $\phi$ be as above.

 Suppose that for all $g\in \supp{\phi}$ we have $H_1\subset H_2^g$. Then the map $\phi_\bullet:\ho{\bullet}{H_1}{\Z}\to
\ho{\bullet}{H_2}{\Z}$ is the map $\sum_{g\in E}\Cor{{}^gH_1}{H_2}g\cdot$.
\end{cor}

By Corollary \ref{cor:lnapsl} again, $L_0$ is the permutation module $\Z[X_0]$ where $X_0$ is a transitive $G$-set and the stabilizer of $(\infty)$ is $B$ while $L_1=\Z[X_1]$ where $X_1$ is  transitive and the stabilizer of $(\infty,0)$ is $T$. The differential $L_1=\Z[X_1]\to\Z[X_0]=L_0$ sends $(\infty,0)$ to $(0)-(\infty)=(\infty)\cdot(\tilde{\omega}-I)$. By Corollary \ref{cor:phi}, we deduce:

\begin{cor}\label{cor:d1p1} The differential 
\[
d^1:\ho{p}{T}{\Z}\cong E^1_{p,1}\to E^1_{p,0}\cong\ho{p}{B}{\Z}
\]
 is $\Cor{T}{B}\circ (\tilde{\omega} -1)$.
\end{cor}

\begin{cor}\label{cor:e2p0} In the spectral sequence $E(G,\LL)$ we have
\begin{enumerate}
\item $E^2_{1,0}\cong \ab{B}/T^2$, and
\item $E^2_{2,0}=E^1_{2,0}\cong \ho{2}{B}{\Z}$.
\end{enumerate}
\end{cor}
\begin{proof}
\begin{enumerate}
\item Conjugation by $\tilde{\omega}$ induces the map $u\to u^{-1}$ on $T\cong A^\times$ and hence $\tilde{\omega}-1:T=\ho{1}{T}{\Z}\to\ho{1}{T}{\Z}=T$ is the map 
$u\mapsto u^{-2}$ with image $T^2$. The map $\Cor{T}{B}:T=\ho{1}{T}{\Z}\to\ho{1}{B}{\Z}=\ab{B}$ is split injective with image (isomorphic to) $T$.
\item Conjugation by $\tilde{\omega}$ induces the identity map on $\ho{2}{T}{\Z}\cong T\wedge T$ and hence the $d^1$-map is zero.
\end{enumerate}
\end{proof}

\begin{lem}\label{lem:d1p2} There is a natural isomorphism $E^1_{p,2}\cong \sgr{A}\otimes \ho{p}{\mu_2}{\Z}$ and with this isomorphism, the differential 
\[
d^1_{p,2}:\sgr{A}\otimes \ho{p}{\mu_2}{\Z}\cong E^1_{p,2}\to E^1_{p,1}\cong \ho{p}{T}{\Z}=\Z\otimes \ho{p}{T}{\Z}
\]
 is the map $\epsilon\otimes \Cor{\mu_2}{T}$ where $\epsilon:\sgr{A}\to\Z$ is the augmentation map.
\end{lem}

\begin{proof} By Corollary \ref{cor:lnapsl}, $X_2=\sqcup_{u\in A^\times/(A^\times)^2}(\infty,0,u)\cdot G$ and the stabilizer of $(\infty,0,u)$ is $\mu_2=Z(G)$. Thus 
$L_2=\bigoplus_{u\in A^\times/(A^\times)^2}(\infty,0,u)\cdot\Z[\mu_2\backslash G]$ and hence
\begin{eqnarray*}
E^1_{p,2}=\ho{p}{G}{L_2}&=&\bigoplus_{u\in A^\times/(A^\times)^2}\an{u}\Z\otimes\ho{p}{G}{\Z[\mu_2\backslash G]}\\
&\cong&\bigoplus_{u\in A^\times/(A^\times)^2}\an{u}\Z\otimes\ho{p}{\mu_2}{\Z}\\
&\cong& \sgr{A}\otimes\ho{p}{\mu_2}{\Z}.\\
\end{eqnarray*}
Now, for $u\in A^\times$, $(\infty,0,u)\in L_2$ maps to $(0,u)-(\infty,u)+(\infty,0)=(\infty,0)\cdot( g_1^u-g_2^u+1)$ in $L_1$ for some $g_i^u\in G$. By Corollary \ref{cor:phi}, the resulting homomorphism 
\[
\an{u}\Z\otimes \ho{p}{\mu_2}{\Z}\cong \ho{p}{G}{(\infty,0,u)\cdot \Z[\mu_2\backslash G]}\to \ho{p}{T}{\Z}\cong\ho{p}{G}{L_1}
\]
 sends $\an{u}\otimes z$ to 
$\Cor{\mu_2}{T}\circ(g_1^u-g_2^u+1)\cdot z$ which is just $\Cor{\mu_2}{T}(z)$ since conjugation by $g_i^u$ on $\mu_2$ is the identity map.
\end{proof}

\subsection{Some $d^2$ calculations}
\begin{lem}\label{lem:d2a}  Let $A$ be a ring. 
 In the spectral sequence $E(\spl{2}{A},\LL)$ we have $E^2_{1,0}=B^{\mathrm ab}/T^2$ and $E^2_{0,2}$ is naturally a quotient of $\aug{A}\subset\sgr{A}=E^1_{0,2}$. 
The composite homomorphism
\[
\xymatrix{
\aug{A}\ar@{>>}[r]&E^2_{0,2}\ar^-{d^2}[r]&E^2_{1,0}=B^{\mathrm ab}/T^2
}
\]
sends $\pf{u}$ to the class of the matrix 
$
\left[
\begin{array}{ll}
u&0\\
3(1-u)&u^{-1}\\
\end{array}
\right]
$.
\end{lem}

\begin{proof} Let $G=\spl{2}{A}$ and let $C_\bullet=C_\bullet(G)$. For $u\in A^\times$, $\an{u}\in \sgr{A}=E^1_{0,2}$ is represented by $Z_u:=(\infty,0,u)\otimes 1\in L_2\otimes C_0$. Then
\begin{eqnarray*}
d^1(Z_u)&=&\left[ (0,u)-(\infty,u)+(\infty,0)\right]\otimes 1\in L_1\otimes C_0\\
&=& (\infty,0)\otimes (g_1(u)-g_2(u)+1)
\end{eqnarray*}
where
\[
g_1(u):=
\left[
\begin{array}{ll}
0&-u^{-1}\\
u&1\\
\end{array}
\right],\quad
g_2(u):=
\left[
\begin{array}{ll}
1&0\\
u&1\\
\end{array}
\right]\in G.
\]
Thus
\[
d^1(\pf{u})=d^1(Z_u-Z_1)=(\infty,0)\otimes \left[ (g_1(u)-g_2(u)-g_1(1)+g_2(1)\right]\in L_1\otimes C_0.
\]
This is the image under $d^h$  of $(\infty,0)\otimes \left[ (g_1(1),g_2(1))-(g_1(u),g_2(u)\right]=(\infty,0)\otimes X \in L_1\otimes C_1$.

In turn, this maps by $d^v$ to 
\begin{eqnarray*}
\left( (0)-(\infty)\right)\otimes X&=& (\infty)\otimes (\tilde{\omega}-1)X\\
&=& (\infty)\otimes \left[ (\tilde{\omega}g_1(1),\tilde{\omega}g_2(1))-(\tilde{\omega}g_1(u),\tilde{\omega}g_2(u))-(g_1(1), g_2(1))+(g_1(u), g_2(u))\right]\\
&:=&(\infty)\otimes Y\in L_0\otimes C_1=(\infty)\cdot \Z[B\backslash G]\otimes_{\Z[G]} C_1.
\end{eqnarray*}
Now $\Z\otimes_{\Z[B]}C_\bullet\cong (\infty)\cdot \Z[B\backslash G]\otimes_{\Z[G]} C_\bullet$ via $1\otimes x\mapsto (\infty)\otimes x$ (inducing the Shapiro isomorphism). 
Thus our last term corresponds to $1\otimes Y\in \Z\otimes_{\Z[B]}C_1$, representing an element in $E^2_{1,0}=H_1(B)/T^2$.  

In order to identify this element, we first construct an augmentation-preserving chain map of left $\Z[B]$-complexes $\Psi_\bullet: C_\bullet\to C_\bullet(B)$:

Let $s:B\backslash G \to G$ be any section satisfying 
\begin{eqnarray*}
s(Bg)=1& \mbox{ if } \infty\cdot g=\infty\\
s(Bg)=\tilde{\omega}&\mbox{ if } \infty\cdot g= 0\\
s(Bg)=E(u):=\left[
\begin{array}{ll}
-u&-1\\
1&0\\
\end{array}
\right]&\mbox{ if }\infty\cdot g=u.\\
\end{eqnarray*}
Then, for $g\in G$, let $\rho(g):=gs(Bg)^{-1}\in B$  and \[
\Psi_n:C_n\to C_n(B),\quad (g_0,\ldots,g_n)\mapsto (\rho(g_0),\ldots,\rho(g_n))
\]
is the required chain map.

Now
\begin{eqnarray*}
\infty\cdot g_1(u)=0\imp s(g_1(u))=\tilde{\omega}&\imp&\rho(g_1(u))=g_1(u)\tilde{\omega}^{-1}=
\left[
\begin{array}{ll}
u^{-1}&0\\
-1&u\\
\end{array}
\right]\\
\infty\cdot g_2(u)=\infty\imp s(g_2(u)=1&\imp& \rho(g_2(u))=g_2(u)\\
\infty\cdot \tilde{\omega}g_1(u)=0\cdot g_1(u)=u\imp s(\tilde{\omega}g_1(u))=E(u)&\imp & \rho(\tilde{\omega}g_1(u))=\tilde{\omega}g_1(u)E(u)^{-1}=
\left[
\begin{array}{ll}
1&0\\
u^{-1}&1\\
\end{array}
\right]\\
\infty\cdot\tilde{\omega}g_2(u)=0\cdot g_2(u)=u\imp s(\tilde{\omega}g_2(u))=E(u)&\imp &\rho(\tilde{\omega}g_2(u))=\tilde{\omega}g_2(u)E(u)^{-1}=I.\\
\end{eqnarray*}
Thus
\begin{eqnarray*}
1\otimes \Psi(Y)&=&1\otimes
\left(\left[
\begin{array}{ll}
1&0\\
1&1\\
\end{array}
\right],\left[
\begin{array}{ll}
1&0\\
0&1\\
\end{array}
\right]\right)-
\left(\left[
\begin{array}{ll}
1&0\\
u^{-1}&1\\
\end{array}
\right],\left[
\begin{array}{ll}
1&0\\
0&1\\
\end{array}
\right]\right)\\
&&-\left(\left[
\begin{array}{ll}
1&0\\
-1&1\\
\end{array}
\right],\left[
\begin{array}{ll}
1&0\\
1&1\\
\end{array}
\right]\right)+
\left(\left[
\begin{array}{ll}
u^{-1}&0\\
-1&u\\
\end{array}
\right],\left[
\begin{array}{ll}
1&0\\
u&1\\
\end{array}
\right]\right)\in \Z\otimes_{\Z[B]}C_1(B).
\end{eqnarray*}
Finally, the map $\Z\otimes_{\Z[B]}C_1(B)\to \ho{1}{B}{\Z}=\ab{B}$ sends $1\otimes (b_0,b_1)$ to $b_0^{-1}b_1\mod{[B,B]}$. Thus, under this map, $1\otimes \Psi(Y)$ is sent to
\begin{eqnarray*}
\left[
\begin{array}{ll}
1&0\\
-1&1\\
\end{array}
\right]\cdot
\left[
\begin{array}{ll}
1&0\\
-u^{-1}&1\\
\end{array}
\right]^{-1}\cdot
\left[
\begin{array}{ll}
1&0\\
2&1\\
\end{array}
\right]^{-1}\cdot
\left[
\begin{array}{ll}
u&0\\
2&u^{-1}\\
\end{array}
\right]\\
=\left[
\begin{array}{ll}
1&0\\
u^{-1}-1&1\\
\end{array}
\right]\cdot
\left[
\begin{array}{ll}
u&0\\
2(1-u)&u^{-1}\\
\end{array}
\right]=
\left[
\begin{array}{ll}
u&0\\
3(1-u)&u^{-1}\\
\end{array}
\right]\in \ab{B},\\
\end{eqnarray*}
as required.
\end{proof}

It is convenient here to recall the natural (graded) module structure on the spectral sequence $E(\spl{2}{A},\LL)$ over the Pontryagin algebra $\ho{\bullet}{\mu_2}{\Z}$ (see also \cite[Section 4.3]{hut:cplx13}):  

Since the group $\mu_2$ is central in $G=\spl{2}{A}$, multiplication induces a group homomorphism $G\times \mu_2\to G$. Thus if $M$ is any right $\Z[G]$-module 
on which $\mu_2$ acts trivially, there is an induced product
\[
\ho{p}{G}{M}\otimes_\Z\ho{q}{\mu_2}{\Z}\to\ho{p+q}{G\times\mu_2}{M}\to\ho{p+q}{G}{M}.
\]
This map is induced from a map of underlying complexes: Let $F_\bullet$ (resp. $F'_\bullet$) be a (left) projective resolution of $\Z$ over $\Z[G]$ (resp. over $\Z[\mu_2]$). Note that 
the total tensor product complex $F_\bullet\otimes F'_\bullet$ is a projective resolution of $\Z$ over $\Z[G\times\mu_2]=\Z[G]\otimes\Z[\mu_2]$. Let $\tau:F_\bullet\otimes F'_\bullet\to F_\bullet$ be any augmentation-preserving map of $\Z[G\times\mu_2]$-complexes (wheree the right-hand complex is a $\Z[G\times\mu_2]$-module via the multiplication homomorphism.
Then we have natural maps of complexes
\[
\xymatrix{
\left( F_\bullet\otimes_{\Z[G]}M\right)\otimes\left( F'_\bullet\otimes_{\Z[\mu_2]}\Z\right)\ar^-{\cong}[r]&(F_\bullet\otimes F'_\bullet)\otimes_{\Z[G\times\mu_2]}M\ar[r]&
F_\bullet\otimes_{\Z[G]}M
}
\]
where the first map is the obvious isomorphism and the second is induced by $\tau$. 

Now, we substitute for $M$ the complex $\LL_\bullet$ of $\Z[G]$-modules on which $\mu_2$ acts trivially. Recall that our spectral sequence arises from a filtration on the double complex
$D_{p,q}:=\LL_q\otimes_{\Z[G]}F_p$ (so that $E^1_{p,q}=\ho{p}{G}{\LL_q}$). We thus have maps
\[
D_{p,q}\otimes (\Z\otimes_{\Z[\mu_2]}F'_k)\to D_{p+k,q}
\]
inducing homomorphisms 
\[
E^r_{p,q}\otimes\ho{k}{\mu_2}{\Z}\to E^r_{p+k,q}, \quad (z,\alpha)\mapsto z\ast\alpha
\]
for all $r\geq 1$, satisfying $d^r(z\ast \alpha)=d^r(z)\ast\alpha$ for all $z,\alpha$.

Applying this product structure to the case $z=\pf{u}\in \aug{A}/\image{d_1}=E^2_{0,2}$, we immediately deduce:
\begin{lem}\label{lem:d2b} Let $A$ be a ring.
Since $\mu_2$ is central in $B$, multiplication gives a group homomorphism $B\times \mu_2\to B$ and thus there is an induced multiplication on homology
\[
H_1(B,\Z)\otimes H_1(\mu_2,\Z)\to H_2(B\times \mu_2,\Z)\to H_2(B,\Z),\quad (x,y)\mapsto x\ast y.
\]
 Now $E^2_{1,2}$
 is naturally a quotient of $\aug{A}\otimes\mu_2\subset\sgr{A}\otimes \mu_2=E^1_{1,2}$. 
For any $\epsilon\in\mu_2$, the composite homomorphism
\[
\xymatrix{
\aug{A}\otimes \mu_2\ar@{>>}[r]&E^2_{1,2}\ar^-{d^2}[r]&E^2_{2,0}=H_2(B)
}
\]
sends $\pf{u}\otimes \epsilon$ to $\left[
\begin{array}{ll}
u&0\\
3(1-u)&u^{-1}\\
\end{array}
\right]\ast  \epsilon$. 
\end{lem}

\begin{cor}\label{cor:d2b} Let $A$ be a ring.
The (split surjective) map $\ho{2}{B}{\Z}\to\ho{2}{T}{\Z}$ induces a well-defined surjective homomorphism $E^2_{2,0}\to \ho{2}{T}{\Z}=T\wedge T$.  
For any $\epsilon\in \mu_2(A)$, the composite homomorphism
\[
\xymatrix{
\aug{A}\otimes \mu_2\ar@{>>}[r]&E^2_{1,2}\ar^-{d^2}[r]&E^2_{2,0}\ar[r]&T\wedge T
}
\]
sends $\pf{u}\otimes \epsilon$ to $u\wedge \epsilon$. 
\end{cor}

\begin{cor}\label{cor:d2b-1} For any ring $A$, the element $\pf{-1}\otimes -I$ lies in the kernel of  $d^2:E^2_{1,2}\to E^2_{2,0}$.
\end{cor}

\begin{proof} 
By Lemma \ref{lem:d2b}
\begin{eqnarray*}
d^2(\pf{-1}\otimes -I)&=& \left[
\begin{array}{ll}
-1&0\\
6&{-1}\\
\end{array}
\right]\ast  -I\\
&=&(\rho^{-6}\cdot -I)\ast -I \mbox{ where }\rho:=
\left[
\begin{array}{ll}
1&0\\
1&1\\
\end{array}
\right]\\
&=& (\rho^{-3})^2\ast-I+-I\ast -I=\rho^{-3}\ast (-I)^2+-I\wedge -I=0+0.
\end{eqnarray*}
\end{proof}

\begin{cor}\label{cor:E320}  We have $E^3_{2,0}={\qho{2}{B}{\Z}}$ where ${\qho{2}{B}{\Z}}$ denotes $\ho{2}{B}{\Z}$ modulo the subgroup
$\left\{ \left[
\begin{array}{ll}
u&0\\
3(1-u)&u^{-1}\\
\end{array}
\right]\ast  \e\  | u\in A^\times, \e\in\mu_2\right\}$.
\end{cor}

\begin{cor}\label{cor:E320T}
The natural map $\ho{2}{B}{\Z}\to\ho{2}{T}{\Z}=T\wedge T$ gives rise to a commutative diagram of surjective homomorphisms
\[
\xymatrix{
\ho{2}{B}{\Z}\ar@{>>}[r]\ar@{>>}[d]&T\wedge T\ar@{>>}[d]\\
E^3_{2,0}=\qho{2}{B}{\Z}\ar@{>>}[r]&\frac{T\wedge T}{T\wedge \mu_2}.\\
}
\]
\end{cor}
\subsection{The map $d^3:\rpbker{A}=E^3_{0,3}\to E^3_{2,0}$}
In this section, we let $\tilde{G}:=\gl{2}{A}$, and $B_A$, $T_A$ are the subgroups
\[
B_A=B(\tilde{G}):=
\left\{
\left[
\begin{array}{cc}
a&0\\
b&c\\
\end{array}
\right]\in \tilde{G}
\right\},\quad
T_A=T(\tilde{G}):=
\left\{
\left[
\begin{array}{cc}
a&0\\
0&b\\
\end{array}
\right]\in \tilde{G}
\right\}.
\]

\begin{lem}\label{lem:d3} 
Under the  map $d^3=d^3_{0,3}:\rpbker{A}=E^3_{0,3}\to E^3_{2,0}={\qho{2}{B}{\Z}}$, 
we have $d^3(\suss{1}{-1})=\rho\ast -I$.
\end{lem}

\begin{proof}
The element  $\suss{1}{-1}\in E^3_{0,3}$ is, by definition,  represented by 
\begin{eqnarray*}
(\infty,0,-1)+(0,\infty,-1)-(\infty,0,1)-(0,\infty,1)=
\left[ (\infty,0,1)+(0,\infty,1)\right]\cdot (\tilde{\omega}-I)\in Z_2=L_3^\tau.
\end{eqnarray*}

This maps under $d^v$ to $\left[ (\infty,0,1)+(0,\infty,1)\right]\otimes (\tilde{\omega}-I)$ in $L_2\otimes C_0(G)$. This  is the image, under $d^h$,  of 
\[
z_1:=\left[ (\infty,0,1)+(0,\infty,1)\right]\otimes (I,\tilde{\omega})-(\infty,0,1)\otimes (I,-I)\in L_2\otimes C_1(G).
\] 
(The second term is in the kernel of $d^h$, and is an adjustment term to ensure $d^v(z_1)$ represents the trivial class in $H_1(\mu_2,\Z)$.)

Then, letting $h:=\left[
\begin{array}{ll}
0&-1\\
1&1\\
\end{array}
\right]\in \spl{2}{A}$,
\begin{eqnarray*}
z_2:=d^v(z_1)&=&[(\infty,0)+(0,\infty)]\otimes  (I,\tilde{\omega})-(\infty,0)(h-\rho+I)\otimes (I,-I)\\
&=&(\infty,0)\cdot (\tilde{\omega}+I)\otimes  (I,\tilde{\omega})-(\infty,0)(h-\rho+I)\otimes (I,-I)\\
&=&(\infty,0)\otimes [ (\tilde{\omega},-I)+(I,\tilde{\omega})-(h,-h)+(\rho,-\rho)-(I,-I)]\in L_1\otimes C_1(G)
\end{eqnarray*}
This is the image under $d^h$ of 
\begin{eqnarray*}
z_3:&=&(\infty,0)\otimes[(I,\tilde{\omega},I)+(\tilde{\omega},-I,-I)-(\tilde{\omega},I,-I)-(h,-h,-I)+(h,I,-I)\\
&&+(\rho,-\rho,-I)-(\rho,I,-I)-(I,-I,-I)+(I,I,-I)]\\
&=& (\infty,0)\otimes X\in L_1\otimes C_2(G).
\end{eqnarray*}
Finally, $d^3(\suss{1}{-1})$ will be represented by the class
\begin{eqnarray*}
d^v(z_3)&=&(0)-(\infty)\otimes X= (\infty)\cdot (\tilde{\omega}-I)\otimes X\\
&=&(\infty)\otimes (\tilde{\omega}-I)X\in L_0\otimes C_2=(\infty)\cdot \Z[B\backslash G]\otimes_{\Z[G]}C_2(G).
\end{eqnarray*}
Now 
\[
(\infty)\cdot \Z[B\backslash G]\otimes_{\Z[G]}C_2(G)\cong \Z\otimes_{\Z[B]}C_2(G)
\]
via a map $f$ defined by $f((\infty)\otimes x)= 1\otimes x$. Let $s:B\backslash G\to G$ be any section satisfying $s(b)=1$ for all $b\in B$. Let $r: G\to B$ be the map $g\mapsto gs(g)^{-1}$. Then 
$r$ is a map of left $B$-sets and the map $\psi:C_\bullet(G)\to C_\bullet(B)$, $(g_0,\ldots,g_n)\mapsto (r(g_0),\ldots,r(g_n))$ defines an augmentation-preserving 
homomorphism of left $\Z[B]$-complexes. It follows that the composite $(\Z\otimes \psi)\circ f: L_0\otimes C_2(G)\to \Z\otimes C_2(B)$ will induce the Shapiro Lemma isomorphism
$E^3_{2,0}=E^1_{2,0}\cong H_2(B,\Z)$. 

Now let $s:B\backslash G\to G$ be any section satisfying: $s(g)=I$ if $\infty\cdot g=\infty$, $s(g)=\tilde{\omega}$ if $\infty\cdot g=0$, and $s(g)=k:=
\left[
\begin{array}{ll}
1&1\\
0&1\\
\end{array}
\right]$ if $\infty\cdot g=1$.
Then we have:
\begin{eqnarray*}
r(\tilde{\omega})=I,\quad r(h)=\rho^{-1},\quad r(\rho)=\rho,\quad r(\tilde{\omega}h)=-I,\quad r(\tilde{\omega}\rho)=-\rho^{-1}.
\end{eqnarray*}
Observe also that $r(bx)=br(x)$ for $b\in B$. In particular, taking $b=-I$, $r(-x)=-r(x)$. Thus applying $(\Z\otimes \psi)\circ f$ to $(\infty)\otimes (\tilde{\omega}-I)X$ 
gives the element
\begin{eqnarray*}
z_4&=&1\otimes [(-\rho^{-1},\rho^{-1},-I)-(-\rho^{-1},I,-I)-(I,-I,-I)+(I,I,-I)\\
&&+(\rho^{-1},-\rho^{-1},-I)-(\rho^{-1},I,-I)-(\rho,-\rho,-I)+(\rho,I,-I)]\in \Z\otimes_{\Z[B]}C_2(B)\\
&=&1\otimes [(I,-I,\rho)-(I,-\rho,\rho)-(I,-I,-I)+(I,I,-I)\\
&&+(I,-I,-\rho)-(I,\rho,-\rho)-(I,-I,-\rho^{-1})+(I,\rho^{-1},-\rho^{-1})]
\end{eqnarray*}

Now for a  group $G$, and two pairwise commuting subgroups $A$ and $B$, the product 
\[
\ab{A}\otimes\ab{B}=\ho{1}{A}{\Z}\otimes\ho{1}{B}{\Z}\to\ho{2}{A\times B}{\Z}\to \ho{2}{G}{\Z},\quad a\otimes b\mapsto a\ast b
\]
 is described at the level of standard chains as follows: For $a\in A$, $b\in B$, the element $a\ast b$ is represented by the cycle $1\otimes (1,a,ab)-(1,b,ab)\in \Z\otimes_{\Z[G]}C_2(G)$.
 Thus the cycle $z_4$ represents
\[
(-I\ast -\rho)- (-I\ast I )+-I\ast \rho-(-I\ast \rho^{-1})=-I\ast \rho=\rho\ast -I \in \ho{2}{B}{\Z}
\]
and thus $d^3(\suss{1}{-1})=\rho\ast -I$ as claimed.
\end{proof}

\begin{cor}\label{cor:suss} Let $A$ be a ring.
\begin{enumerate}
\item If $2\in A^\times$ then $\suss{1}{-1}\in \rbl{A}$.
\item If $A$ admits a ring homomorphism to $\Z/4$ then $\suss{1}{-1}\not\in\rbl{A}$.
\end{enumerate}
\end{cor}
\begin{proof}
\begin{enumerate}
\item $\suss{1}{-1}\in \rbl{\Z[\frac{1}{2}]}$ by Lemma \ref{lem:d3} and Lemma \ref{lem:1/2}. If $2\in A^\times$, there is a ring homomorphism $\Z[\frac{1}{2}]\to A$ and hence 
$\suss{1}{-1}\in\rbl{A}$ by functoriality.
\item By Lemma \ref{lem:d3} and Lemma \ref{lem:z4}, $d^3(\suss{1}{-1})\not= 0$ in $E^3_{2,0}(\spl{2}{A},\LL)$ .
\end{enumerate}
\end{proof}

The remaining calculations below are devoted to calculating explicitly, \emph{for any ring $A$},  the differential  $d^3:\apb{A}=E^3_{0,3}(\tilde{G},L)\to E^3_{2,0}(\tilde{G},L)$.

\begin{lem} In the spectral sequence $E(\tilde{G},L)$ we have $E^2_{p,2}=0$ for all $p$.
\end{lem}
\begin{proof} $L_2=\Z[X_2]$ where $X_2$ is a transitive $\tilde{G}$-set and the stabilizer of $(\infty,0,1)$ is $Z(\tilde{G})=A^\times\cdot I\cong A^\times$. Thus 
$E^1_{p,2}=\ho{p}{\tilde{G}}{L_2}\cong \ho{p}{A^\times}{\Z}$. Similarly, $X_1$ is a transitive $\tilde{G}$-set and the stabilizer of $(\infty,0)$ is $T_A$. Thus 
$^1_{p,1}\cong\ho{p}{T_A}{\Z}$. The map $d^1_{p,2}:\ho{p}{A^\times}{\Z}\to\ho{p}{T_A}{\Z}$ is $\Cor{A^\times}{T_A}\circ(g_1-g_2+1)=\Cor{A^\times}{T_A}$ (since $A^\times$ is central. Since the map $A^\times\to T_A$ is a split injection, the map $\Cor{A^\times}{T_A}$ is always injective.
\end{proof}

It follows that all of the differentials $d^2_{p,2}:E^2_{p,2}\to E^2_{p+1,0}=\ho{p+1}{B_A}{\Z}/\Cor{}{}(\omega-1)$ are zero (where here $\Cor{}{}(\omega-1)$ denotes the image of 
$d^1=\Cor{T_A}{B_A}\circ (\omega -1):\ho{n}{T_A}{\Z}=E^1_{n,1}\to E^1_{n,0}=\ho{n}{B_A}{\Z}$).
\begin{cor}\label{cor:e3p0}
In the spectral sequence $E(\tilde{G},L)$ we have 
\[
E^3_{p,0}=E^2_{p,0}=\frac{\ho{p}{B_A}{\Z}}{\Cor{T_A}{B_A}(\omega-1)}
\]
for all $p\geq 0$.
\end{cor}

We identify the group of diagonal matrices $T_A$ with $A^\times\times A^\times$. 

Thus $\ho{2}{T_A}{\Z}\cong (A^\times\times A^\times)\wedge (A^\times\times A^\times)$.


 Let $G$ be a (multiplicative) abelian group. Let $\pi: (G\times G)\wedge (G\times G)\to\asym{2}{\Z}{G}$ be the homomorphism
$(a,b)\wedge (c,d)\mapsto a\asymm d+b\asymm c$. Note that if $\tau$ is the map on $G\times G$ induced by $(a,b)\mapsto (b,a)$ in $G\times G$, then $\pi\circ(\tau-1)=\pi\circ\tau-\pi=0$.

Let $d:(G\times G)\wedge (G\times G\to G\wedge G$ be the map $(a,b)\wedge (c,d)\mapsto ab\wedge cd$ (i.e., the map induced by multiplication). 

\begin{lem}\label{lem:ext2gxg} For any abelian group $G$  the maps $\pi$ and $d$ induce  a natural isomorphism
\[
\frac{(G\times G)\wedge (G\times G)}{\tau -1}\cong \asym{2}{\Z}{G}\oplus (G \wedge G),\quad  \alpha\mapsto (\pi(\alpha),d(\alpha)).
\]
\end{lem}

\begin{proof} Since exterior powers commute with direct sums, we have
\[
(G\times G)\wedge (G\times G)=\left[(G\times 1)\wedge (G\times 1)\right]\oplus \left[(G\times 1)\otimes (1\times G)\right]\oplus \left[(1\times G)\wedge (1\times G)\right].
\]
Furthermore, 
\begin{eqnarray*}
\tau((a,1)\wedge (b,1))= (1,a)\wedge (1,b))\\
\tau((a,1)\wedge (1,b))=-((b,1)\wedge (1,a)).\\
\end{eqnarray*}
Thus, modulo $\tau -1$, the first and last factor are identified and the middle factor  becomes isomorphic to $\asym{2}{\Z}{G}$. 
\end{proof}

 For any  (multiplicative) abelian group $G$, let $C_\bullet(G)$ denote the (left) homogeneous standard resolution of $\Z$ over $\Z[G]$. Recall that the isomorphism (induced by the Pontryagin product) 
$G\wedge G\cong \ho{2}{G}{\Z}$ is given by sending $a\wedge b$ to the homology class represented by $(1,a,ab)-(1,b,ab)\in C_2(G)$.

\begin{lem}\label{lem:h2gxg} Let $G$ be an abelian group. Let $\Pi:C_2(G\times G)\to G\otimes G$ be the map 
\[
((g_0,h_0),(g_1,h_1),(g_2,h_2))\mapsto (g_0g_1^{-1})\otimes (h_1h_2^{-1}).
\]
Then $\Pi$ induces a well-defined homomorphism $\Z\otimes_{\Z[G\times G]}C_2(G\times G)\to G\otimes G$ which vanishes on boundaries from $\Z\otimes_{\Z[G\times G]}C_3(G\times G)$ and such that 
the resulting induced composite map
\[
\xymatrix{
(G\times G)\wedge (G\times G)\ar^-{\cong}[r]&\ho{2}{G\times G}{\Z}\ar^-{\Pi}[r]&G\otimes G\ar[r]&\asym{2}{\Z}{G}\\
}
\]
is the map $\pi$.
\end{lem}

\begin{proof}  It is clear,  from the definition, that $\Pi((g,h)\cdot w)=\Pi(w)$ for all $w\in C_2(G\times G)$. 

For the second statement, if $g_i,h_i\in G$, $i=1,2,3$ then 
\begin{eqnarray*}
\Pi\left( d_3((1,1),(g_1,h_1),(g_2,h_2),(g_3,h_3)\right)&=& \Pi\left( (g_1,h_1),(g_2,h_2),(g_3,h_3)\right)-\Pi\left( (1,1),(g_2,h_2),(g_3,h_3)\right)\\
&&+\Pi\left( (1,1),(g_1,h_1),(g_3,h_3)\right)-\Pi\left( (1,1),(g_1,h_1),(g_2,h_2)\right)\\
&=& g_1g_2^{-1}\otimes h_2h_3^{-1}-g_2^{-1}\otimes h_2h_3^{-1}+g_1^{-1}\otimes h_1h_3^{-1}-g_1^{-1}\otimes h_1h_2^{-1}\\
&=&0
\end{eqnarray*}
as required. 

Finally, the given composite map sends $(a,b)\wedge (c,d)$ first to 
\[
1\otimes \left( (1,1),(a,b),(ac,bd)\right)-\left( (1,1),(c,d),(ac,bd)\right)\in\Z\otimes_{\Z[G\times G]} C_2(G\times G)
\]
which is sent by $\Pi$ to $a^{-1}\otimes d^{-1}-c^{-1}\otimes b^{-1}=a\otimes d-c\otimes b\in G\otimes G$ which in turn maps to $a\asymm d+b\asymm c$ in
$\asym{2}{\Z}{G}$.
\end{proof}

Now let $p$ be the natural map 
\[
E^3_{2,0}(\tilde{G},L)=\frac{\ho{2}{B}{\Z}}{\Cor{}{}(\omega-1)}\to\frac{\ho{2}{T_A}{\Z}}{(\omega-1)}
\]
and let $\bar{\pi}:E^3_{2,0}(\tilde{G},L)\to \asym{2}{\Z}{A^\times}$ be the map $\pi\circ p$.
\begin{prop} \label{prop:lambda}
 Let $A$ be a ring.
\begin{enumerate}
\item The image of $p\circ d^3:\apb{A}= E^3_{3,0}\to \frac{\ho{2}{T_A}{\Z}}{(\omega-1)}$ is contained in $\ker{d}\cong \asym{2}{\Z}{A^\times}$.
\item The composite homomorphism $\lambda:=\bar{\pi}\circ d^3:\apb{A}\to \asym{2}{\Z}{A^\times}$ 
is given by \\
$\lambda(\agpb{u})=(1-u^{-1})\asymm u^{-1}$ for all $u\in\wn{A}$.
\end{enumerate}
\end{prop}

\begin{proof}
\begin{enumerate}
\item Let 
\[
\bar{H}_2(B_A,\Z):=E^3_{2,0}=\frac{\ho{2}{B_A}{\Z}}{\Cor{}{}(\omega-1)}\mbox{ and } \bar{H}_2(T_A,\Z):=\frac{\ho{2}{T_A}{\Z}}{(\omega-1)}.
\]
Since $E^\infty_{2,0}=E^3_{2,0}/\image{d^3}$, the statement follows from the following commutative diagram in which the top row is exact:
\[
\xymatrix{
\apb{A}\ar^-{d^3}[r]&\bar{H}_2(B_A,\Z)\ar[r]\ar^-{p}[d]&\ho{2}{\gl{2}{A}}{L_\bullet}\ar[d]\\
&\bar{H}_2(T_A,\Z)\ar^-{d}[rd]&\ho{2}{\gl{2}{A}}{\Z}\ar^-{\det}[d]\\
&&\ho{2}{A^\times}{\Z}.\\
}
\]
\item As above, for any group $S$,  $C_\bullet(S)$ is the left standard homogeneous resolution of $\Z$ as a $\Z[S]$-module. We let $C_\bullet:=C_\bullet(\tilde{G})$.

Now fix $u\in \wn{A}$. Then $\agpb{u}\in\apb{A}=E^3_{0,3}$ is represented by $(\infty,0,1,u)\otimes 1\in L_3\otimes_{\Z[\tilde{G}]}C_0$. This map, under $d^v$, to
\[
\left[ (0,1,u)-(\infty,1,u)+(\infty,0,u)+(\infty,0,1)\right]\otimes 1=(\infty,0,1)\otimes\left[ g_1-g_2+g_3-1\right]\in L_2\otimes C_0
\]
where
\[
g_1:=
\left[
\begin{array}{cc}
0&1-u\\
u&u\\
\end{array}
\right],\quad
g_2:=
\left[
\begin{array}{cc}
u(u-1)&0\\
u&u\\
\end{array}
\right],\quad
g_3:=
\left[
\begin{array}{cc}
u&0\\
0&1\\
\end{array}
\right].
\]
This, in turn, is the image, under $d^h$, of $(\infty,0,1)\otimes[(g_2,g_1)-(g_3,1)]\in L_2\otimes C_1$. This latter term maps, under $d^v$, to 
\[
(\infty,0)(g_1-g_2+1)\otimes [(g_2,g_1)-(g_3,1)]=(\infty,0)\otimes Y\in L_1\otimes C_1
\]
where 
\begin{eqnarray*}
Y&:=&(g_2-g_1+1) [(g_2,g_1)-(g_3,1)]\\
&=&(g_1g_2,g_1^2)-(g_1g_3,g_1)-(g_2^2,g_2g_1)+(g_2g_3,g_2)+(g_2,g_1)-(g_3,1)\in C_1.\\
\end{eqnarray*}
$(\infty,0)\otimes Y$ is the image, under $d^h$, of $(\infty,0)\otimes X\in L_1\otimes C_2$ where 
\begin{eqnarray*}
X&:=& (g_1g_2,g_1^2,u^{-1}g_3)-(g_1g_2,1.u^{-1}g_3)-(g_1g_3,g_1,1)-(g_2^2,g_2g_1,t_2)\\
&&+(g_2^2,t_1,t_2)+(g_2g_3,g_2,1)+(g_2,g_1,1)-(g_3,1,t)\\
\end{eqnarray*}
where
\[
t_1:=\mathrm{diag}(u-1,u),\quad t_2:=\mathrm{diag}(u(u-1),1)\mbox{ and } t:=t_1^{-1}t_2=\mathrm{diag}(u,u^{-1})\in T_A.
\]
Finally, $(\infty,0)\otimes X$ maps, under $d^v$, to the term
\[
z:=[(0)-(\infty)]\otimes X=(\infty)\otimes [\omega\cdot X-X]\in L_0\otimes C_2
\]
representing an element of $E^3_{2,0}$.

 We wish now to compute the image of $z$ under the composite homomorphism
\[
\xymatrix{
E^3_{2,0}\ar^-{\cong}[r]&\frac{\ho{2}{B_A}{\Z}}{\Cor{}{}(\omega -1)}\ar[r]&\frac{\ho{2}{T_A}{\Z}}{\omega-1}\ar^-{\pi}[r]&\asym{2}{\Z}{A^\times}.\\
}
\]
$z$ corresponds to $1\otimes [\omega\cdot X-X]$ under the isomorphism 
\[
L_0\otimes_{\Z[\tilde{G}]}C_2=\Z[B_A\backslash\tilde{G}]\cdot(\infty)\otimes C_2\cong \Z\otimes_{\Z[B_A]}C_2.
\]
We construct an augmentation-preserving map, $\Psi:C_\bullet\to C_\bullet(B_A)$, of left $\Z[B_A]$ resolutions of $\Z$ as follows: Let $s:B_A\backslash \tilde{G}\to\tilde{G}$ be a section of the 
natural projection $\tilde{G}\to B_A\backslash\tilde{G}$. Let $r:\tilde{G}\to B_A$ be the map $\bar{g}\mapsto g\cdot s(\bar{g})^{-1}$ for $g\in \tilde{G}$. Then the map
$\Psi_n(g_0,\ldots,g_n):=(r(g_0),\ldots,r(g_n))$ has the required properties. Specifically, in our case, we choose a section $s$ satisfying:\\
$s(\bar{g})=1$ if $\infty\cdot g=\infty$.\\
$s(\bar{g})=\omega$ if $\infty\cdot g=0$.\\
$s(\bar{g})=g_1g_2$ if $\infty\cdot g=1$.\\
$s(\bar{g})=\omega g_1^2$ if $\infty\cdot g=u$.\\

Now, for $g\in \tilde{G}$,  let $\bar{r}(g)$ denote the image of $r(g)\in B_A$ in $T_A=A^\times\times A^\times$ under the projection on the diagonal map. Then if
$\bar{\Psi}_\bullet:C_\bullet \to C_\bullet(T_A)$, is the map sending $(g_0,\ldots,g_n)$ to $(\bar{r}(g_0),\ldots,\bar{r}(g_n))$, the image of $z$ in $\bar{H}_2(T_A,\Z)$ is 
represented by $w:=1\otimes \bar{\Psi}_2(\omega\cdot X-X)\in \Z\otimes_{\Z[T_A]}C_2(T_A)$. Using the above section $s$, we calculate
\begin{eqnarray*}
w&=&1\otimes\left[\left( (u,1),(1.1),(u^{-1},1)\right)-\left( (u,1),(1.1),(u^{-1},1)\right)\right.\\
&&-\left( (1,(u-1)^{-1}),((1-u)^{-1},u^{-2}),(1,1)\right)
-\left((u,u),(1,u),(1,(u(u-1))\right)\\
&&+\left((u,u),(u,u-1),(1,u(u-1))\right)+\left( (1,(1-u^{-1})^{-1}),((1-u)^{-1},u^{-1}),(1,1)\right)\\
&&+\left(((1-u)^{-1},u^{-1}),((1-u)^{-1},u^{-2}),(1,1)\right)-\left((1,u),(1,1),(u^{-1},u)\right)\\
&&-\left((1,1),(1,u^{-1}),(1,u^{-1})\right)+\left( (1,1),(1,1),(1,u^{-1})\right)+\left( (1-u,u^2),(1-u,u),(u(u-1),1)\right)\\
&&-\left((u^2(u-1)^2,u^2),(u-1,u),(u(u-1),1)\right)-\left((u^2(u-1),u),(u(u-1),u),(1,1)\right)\\
&&\left. -\left((u(u-1),u),(1-u,u),(1,1)\right)+\left((u,1),(1,1),(u,u^{-1})\right)\right].\\
\end{eqnarray*}
Finally, to calculate the image of this in $\asym{2}{\Z}{A^\times}$ we apply the map $\Pi:C_2(A^\times\times A^\times)\to \asym{2}{\Z}{A^\times}$ of Lemma \ref{lem:h2gxg}, given by\\ 
$\Pi((a_0,b_0),(a_1,b_1),(a_2,b_2))=a_0a_1^{-1}\asymm b_1b_2^{-1}$, to get the element
\[
2[(1-u)\asymm u]+u\asymm(u-1)-(1-u)\asymm u-(u-1)\asymm u-u\asymm u-(-u)\asymm u+u\asymm u.
\]
in $\asym{2}{\Z}{A^\times}$. 
Since $-(u\asymm(u-1))=(u-1)\asymm u=(1-u)\asymm u+(-1)\asymm u$, this simplifies to
\[
-[(1-u)\asymm u+(-u)\asymm u]=-[(1-u)\asymm u-(-u)\asymm u]=-\left[\frac{u-1}{u}\asymm u\right]=(1-u^{-1})\asymm u^{-1}
\]
as required.
\end{enumerate}
\end{proof}





\begin{thebibliography}{10}
\bibitem{adem:naffah}
A. Adem and N. Naffah.
\newblock On the cohomology of {$\mathrm{SL}_2(\mathbb{Z}[\frac{1}{p}])$}.
\newblock {\em London Math. Soc. Lecture Note Ser.}, 252, Cambridge Univ. Press, Cambridge, 1-9, 1998.

\bibitem{sah:dupont}
Johan~L. Dupont and Chih~Han Sah.
\newblock Scissors congruences. {II}.
\newblock {\em J. Pure Appl. Algebra}, 25(2):159--195, 1982.

\bibitem{czz}
L. Cossu, P. Zanardo and U. Zannier.
\newblock Products of elementary matrices and non-Euclidean principal ideal domains.
\newblock  {\em J. Algebra}, 501, 182 -- 205, (2018).




\bibitem{cohn:gln}
P. M. Cohn.
\newblock On the structure of the $\mathrm{GL}_2$ of a ring.
\newblock {\em Inst. Hautes {\'E}tudes Sci. Publ. Math.}, 30 (1966), 5-33.

\bibitem{hut:mat}
Kevin Hutchinson.
\newblock A new approach to Matsumoto's theorem.
\newblock {\em K-Theory} 4 (1990), no. 2, 181–200.



\bibitem{hut:h3sl}
Kevin Hutchinson.
\newblock The third homology of the special linear group of a field.
\newblock {\em J. Pure Appl. Algebra}, 213(9):1665-1680, (2009).

\bibitem{hut:cplx13}
Kevin Hutchinson.
\newblock A {B}loch-{W}igner complex for {$\mathrm{SL}\sb 2$}.
\newblock {\em J. K-Theory}, 12(1):15--68, (2013).

\bibitem{hut:rbl11}
Kevin Hutchinson.
\newblock {A refined Bloch group and the third homology of $\mathrm{SL}_2$ of a
  field}.
\newblock {\em J. Pure Appl. Algebra}, 217:2003--2035, (2013).


\bibitem{hut:slr}
Kevin Hutchinson.
\newblock The third homology of {${\rm SL}_2$} of local rings.
\newblock {\em J. Homotopy Relat. Struct.}, 12(4):931--970, (2017).

\bibitem{hut:sl2Q}
Kevin Hutchinson.
\newblock The third homology of {${\rm SL}_2(\Bbb Q)$}.
\newblock {\em J. Algebra}, 570, 366--396, (2021).

\bibitem{hut:ge2}
Kevin Hutchinson.
\newblock $\mathrm{GE}_2$-rings and a graph of unimodular rows.
\newblock J. Pure Appl. Algebra 226 (2022).

\bibitem{hut:blochf2}
Kevin Hutchinson.
\newblock Bloch Groups of local rings with residue field $\F{2}$.
\newblock {In preparation}.

\bibitem{hut:blochf3}
Kevin Hutchinson.
\newblock Bloch Groups of local rings with residue field $\F{3}$.
\newblock {In preparation}.

\bibitem{levine:k3}
Marc Levine.
\newblock The indecomposable ${K_3}$ of fields.
\newblock {\em Ann. Sci. École Norm. Sup.}, (4) 22 (1989), no. 2, 255–344.


\bibitem{mirzaii:bwlocal}
Behrooz Mirzaii.
\newblock A Bloch-Wigner exact sequence over local rings.
\newblock J. Algebra 476 (2017), 459–493.

\bibitem{mirzaii:perezchar2}
Behrooz Mirzaii and Elvis Torres P\'{e}rez.
\newblock A refined Bloch-Wigner exact sequence in characeristic  $2$.
\newblock J. Algebra 657 (2024), 141-158.

\bibitem{mirzaii:perezpg}
Behrooz Mirzaii and Elvis Torres P\'{e}rez.
\newblock The low-dimensional homology of projective linear group of degree $2$.
\newblock Doc. Math. 30 (2025), no. 5, 1157–1199.


\bibitem{morita:k2zs}
J. Morita.
\newblock{On the group structure of rank one $K_2$ of some $\mathbb{Z}_S$.}
\newblock{\em Bull. Soc. Math. Belg. Sér. A} 42 (1990), no. 3, 561–575.

\bibitem{morita:mab}
J. Morita.
\newblock {Meta-abelianizations of $\mathrm{SL}(2,Z[1/p])$ and Dennis-Stein symbols. }
\newblock{\em Tsukuba J. Math.} 20 (1996), no. 1, 71–76.

\bibitem{moritarehmann:laurent}
J. Morita. and U. Rehmann
\newblock { Symplectic {$K_2$} of Laurent polynomials, associated Kac-Moody groups and Witt rings}
\newblock{\em Math. Z.} 206 (1991), no. 1, 57–66.

\bibitem{serre:trees}
Jean-Pierre Serre.
\newblock Trees.
\newblock Springer-Verlag, Berlin-New York, 1980

\bibitem{sus:bloch}
A.~A. Suslin.
\newblock {$K\sb 3$} of a field, and the {B}loch group.
\newblock {\em Trudy Mat. Inst. Steklov.}, 183:180--199, 229, 1990.
\newblock Translated in Proc.\ Steklov Inst.\ Math.\ {\bf 1991}, no.\ 4,
  217--239, Galois theory, rings, algebraic groups and their applications
  (Russian).



\bibitem{vas}
 L. Vaserstein.
\newblock The group $\mathrm{SL}_2$ over Dedekind rings of arithmetic type. (Russian)
\newblock  {\em Mat. Sb. (N.S.)} 89(131) (1972), 313–322, 351.
\end{thebibliography}
\end{document}